\documentclass[12pt]{article}

\usepackage{amsmath,amsfonts,yhmath,geometry,hyperref}
\usepackage{amsthm}
\usepackage{amssymb}
\usepackage[latin9]{inputenc}
\usepackage[nottoc,notlof,notlot]{tocbibind}
\usepackage{mathtools}
\usepackage{ulem}
\usepackage{enumitem}
\usepackage{tikz-cd}
\usepackage{multirow}

\usepackage{pgf,tikz}

\usepackage{mathrsfs}
\usetikzlibrary{arrows}
\usetikzlibrary[patterns]

\definecolor{qqqqff}{rgb}{0.,0.,1.}
\definecolor{cqcqcq}{rgb}{0.7529411764705882,0.7529411764705882,0.7529411764705882}
\definecolor{ttqqqq}{rgb}{0.2,0.,0.}
\definecolor{qqqqff}{rgb}{0.,0.,1.}
\definecolor{xdxdff}{rgb}{0.49019607843137253,0.49019607843137253,1.}
\definecolor{zzttqq}{rgb}{0.6,0.2,0.}
\definecolor{cqcqcq}{rgb}{0.7529411764705882,0.7529411764705882,0.7529411764705882}

\definecolor{yqyqyq}{rgb}{0.5019607843137255,0.5019607843137255,0.5019607843137255}
\definecolor{uuuuuu}{rgb}{0.26666666666666666,0.26666666666666666,0.26666666666666666}
\definecolor{xdxdff}{rgb}{0.49019607843137253,0.49019607843137253,1.}
\definecolor{qqqqff}{rgb}{0.,0.,1.}

 \font\ncsc=cmcsc10
 \font\ntt=cmtt12

\usepackage{fancybox}

\newcommand{\PP}{\mathbb{P}}

\newcommand{\ZZ}{\mathbb{Z}}
\newcommand{\RR}{\mathbb{R}}
\newcommand{\CC}{\mathbb{C}}
\newcommand{\KK}{\mathbb{K}}

\newcommand{\HH}{\mathbb{H}}

\newcommand{\minus}[1]{\left[ #1 \right]_-}
\newcommand{\plus}[1]{\left[ #1 \right]_+}

\newcommand{\dpar}{\Delta^\mathrm{par}(m_1,m_2,m_3)}
\newcommand{\dcon}{\Delta^\mathrm{con}(m_1,m_2,m_3)}

\newcommand{\dd}{ \mathrm{d} }

\newcommand{\val}{\mathrm{val}}

\newcommand{\fix}{\mathrm{Fix}(\sigma)}

\newcommand{\tg}{ {\mathfrak{t}(\gamma)} }
\newcommand{\hg}{ {\mathfrak{h}(\gamma)} }
\newcommand{\tgs}{ {\mathfrak{t}(\gamma^\sigma)} }
\newcommand{\hgs}{ {\mathfrak{h}(\gamma^\sigma)} }
\newcommand{\gs}{ {\gamma^\sigma} }

\newcommand{\re}{\mathfrak{Re}}
\newcommand{\im}{\mathfrak{Im}}

\newcommand{\parfrac}[2]{\frac{\partial #1}{\partial #2}}

\newcommand{\qd}{q^{1/2}-q^{-1/2}}

\newtheorem{theo}{Theorem}[section]
\newtheorem*{theom}{Theorem}
\newtheorem{prop}[theo]{Proposition}
\newtheorem{coro}[theo]{Corollary}
\newtheorem{lem}[theo]{Lemma}

\theoremstyle{definition}
\newtheorem{defi}[theo]{Definition}
\newtheorem{prob}[theo]{Problem}

\theoremstyle{remark}
\newtheorem{remark}[theo]{Remark}
\newenvironment{rem}[1]{
    \begin{remark}#1}{
    \xqed{\blacklozenge}\end{remark}
}

\theoremstyle{remark}
\newtheorem{example}[theo]{Example}
\newenvironment{expl}[1]{
    \begin{example}#1}{
    \xqed{\lozenge}\end{example}
}

\newcommand{\xqed}[1]{
    \leavevmode\unskip\penalty9999 \hbox{}\nobreak\hfill
    \quad\hbox{\ensuremath{#1}}}

\usepackage[off]{auto-pst-pdf}

\begin{document}
 
 
\title{A Tropical Computation of Refined Toric Invariants II}
\author{Thomas Blomme}
\maketitle

\begin{abstract}
In 2015, G.~Mikhalkin introduced a refined count for real rational curves in toric surfaces. The counted curves have to pass through some real and complex points located on the toric boundary of the surface, and the count is refined according to the value of a so called quantum index. This count happens only to depend on the number of complex points on each toric divisors, leading to an invariant. First, we give a way to compute the quantum index of any oriented real rational curve, getting rid of the previously needed "purely imaginary" assumption on the complex points. Then, we use the tropical geometry approach to relate these classical refined invariants to tropical refined invariants, defined using Block-G\"ottsche multiplicity. This generalizes the result of Mikhalkin relating both invariants in the case where all the points are real, and the result of the author where complex points are located on a single toric divisor.
\end{abstract}

\tableofcontents

\section{Introduction}

\subsection{Enumeration of rational curves in toric surfaces}

In this paper we consider rational curves in $(\CC^*)^2$. Such a curve corresponds to the choice of two rational functions $F$ and $G$ on the projective line $\CC P^1$. We thus get a parametrized curve $\varphi:t\longmapsto (F(t),G(t))\in (\CC^*)^2$. The function is not defined at the poles and zeros of the rational functions. The collection of orders of vanishing of $F$ and $G$ at each of these points is called the degree $\Delta$ of the curve. It is then possible to define a toric surface $\CC\Delta$, which is a compactification of $(\CC^*)^2$, so that $\varphi$ extends at the zeros and poles of $F$ and $G$, but with value in $\CC\Delta$ instead of $(\CC^*)^2$. For instance, if $(a_i,b_i,c_i)_{1\leqslant i\leqslant d}$ are $3d$ distinct scalars, we can consider parametrized curves of the form
$$t\longmapsto\left( \prod\frac{t-a_i}{t-c_i},\prod\frac{t-b_i}{t-c_i}\right),$$
which is a rational curve of degree $d$. The associated toric surface is the standard projective plane $\CC P^2$, and $a_i,b_i,c_i$ are sent to the coordinate axis of $\CC P^2$. However, this representation tends to use the canonical basis of $(\CC^*)^2$. We prefer to adopt a viewpoint independent of this choice. To do this, let $N$ be a two dimensional lattice and $M$ its dual lattice. They are respectively called the lattices of co-characters and characters. The version independent of coordinates of $(\CC^*)^2$ is $N\otimes\CC^*$. As the group law on $\CC^*$ is multiplicative, we prefer to denote the tensor product $n\otimes z$ by $z^n$. To each element $n\in N$ is then associated a function $z\in\CC^*\longmapsto z^n=n\otimes z\in N\otimes\CC^*$, called a co-character, while to any element $m$ of $M$ is associated a character, also called monomial, $\chi^m: z^n\in N\otimes \CC^*\mapsto z^{\langle m,n\rangle}\in\CC^*$. In the case of $(\CC^*)^2$, these monomials are the classical monomials in the two coordinates $z$ and $w$. Using this point of view, a rational curve can be parametrized in the following way:
$$\varphi:t\in\CC P^1\longmapsto \chi\prod_{i=1}^m (t-\alpha_i)^{n_i}\in N\otimes\CC^*,$$
where $\chi\in N\otimes\CC^*$, $\alpha_i\in\CC$ are some scalars and $n_i\in N$ are some co-characters. the multiset of lattice vectors $\Delta=(n_i)\subset N$ is called the degree of the curve. To $\Delta$ is associated a polygon $P_\Delta$ in $M$, unique up to translation, characterized as the unique polygon such that the outer normal vectors to the sides of $P_\Delta$ are the vectors of $\Delta$, counted with a multiplicity equal to their lattice length. For instance, the polygon associated to $\Delta_d=\{ (-1,0)^d,(0,-1)^d,(1,1)^d\}$ is the standard triangle of side length $d$. To $P_\Delta$ is associated a toric surface $\CC\Delta\supset N\otimes\CC^*$. The toric divisor, \textit{i.e.} the coordinate axis of $\CC P^2$, are in bijection with the directions of the vector $\Delta$. Thus, to each co-character of $\Delta$ is associated a toric divisor.\\

The conjugation $t^n\mapsto\overline{t}^n$ extends to $\CC\Delta$ and makes it into a real surface. We say that a rational curve is real if the parametrization can also be chosen invariant by conjugation. In the above notation, this means that $\chi\in N\otimes\RR^*$, and $\alpha_i\in\RR$, or $\alpha_i$ and $\overline{\alpha_i}$ both appear with the same exponent co-character $n_i\in\RR$.\\

Unless stated otherwise, the vectors in the multiset $\Delta$ will be chosen to be primitive, \textit{i.e.} of lattice length $1$. Some lattice length would correspond to a tangency between the curve and the toric divisor where they meet. The dimension of the space of rational curves of degree $\Delta$ is $m-1$: there is $m$ points $\alpha_i$ to choose, $\chi$ which varies in a $2$-dimensional space, but the space of reparametrization of $\CC P^1$ has dimension $3$. Thus, it seems natural to count the number of curves that satisfy  $m-1$ chosen constraints. The easiest constraint that comes to mind is to pass through a given point. Let $\mathcal{P}$ be a set of generic points in $\CC\Delta$, called a configuration. Let $\mathcal{S}^\CC(\mathcal{P})$ be the set of complex rational curves that pass through $\mathcal{P}$. Then, the cardinal $N_\Delta$ of this set does not depend on the choice of $\mathcal{P}$ as long as it is generic. This number is equal to the degree of some varieties of Severi, fact which explains their invariance.\\

Over the real numbers, the situation is more complicated, since the number of real curves passing through a real configuration may depend on the choice of the configuration. By real configuration, we mean the following: a set of real points and pairs of complex conjugated points. Indeed, if a real curve pass through some point, as it is invariant by conjugation, it also passes through the conjugate point. Let $\mathcal{S}^\RR(\mathcal{P})$ be the number of real curves passing through $\mathcal{P}$. The cardinal of $\mathcal{S}^\RR(\mathcal{P})$ depends on $\mathcal{P}$. However, J-Y.~Welschinger \cite{welschinger2005invariants} proved that for toric del Pezzo surfaces, if the curves are counted with a suitable sign, their number only depends on the number $s$ of pairs of conjugated points in the configuration $\mathcal{P}$. This invariant is dented by $W_{\Delta,s}$.\\

The properties and the computation of both invariants $N_\Delta$ and $W_{\Delta,s}$ have the been at the center of many works. One particular approach to their computation unveiled a tremendously rich field of studies, and could to some extent be considered as the ground point of tropical geometry: the correspondence theorem of G.~Mikhalkin \cite{mikhalkin2005enumerative}. He proved a theorem relating the computation of $N_\Delta$ and $W_{\Delta,0}$ (only real points) to a count of tropical curves along with an algorithm to effectively compute these tropical numbers. Although the values of $N_\Delta$ were already known, it was the first computation of $W_{\Delta,0}$. The computation of the values of $W_{\Delta,s}$ for any value of $s$ was dealt with by E.~Shustin in \cite{shustin2004tropical}. Since, the correspondence has first been generalized in any dimension by T.~Nishinou and B.~Siebert \cite{nishinou2006toric} using logarithmic stable maps. Other approaches to the correspondence theorem have also lead to different proofs of it. The reader can refer to G.~Mikhalkin \cite{mikhalkin2006tropical}, E.~Shustin \cite{shustin2002patchworking} or I.~Tyomkin \cite{tyomkin2012tropical}.\\

The computation of $N_\Delta$ and $W_{\Delta,0}$ through the tropical geometry approach counts tropical curves using two distinct multiplicities. This multiplicity is a product over the vertices of the tropical curves. Following this computation, F.~Block and L.~G\"ottsche \cite{block2016refined} proposed to refine these multiplicities into a polynomial one, also a product over the vertices, and this new refined multiplicity specifies itself on both previous multiplicities when evaluated at $\pm 1$. Moreover, this new multiplicity was proved \cite{itenberg2013block} to also lead to an invariant in the tropical world, but whose analog in the classical world remains mysterious. Moreover, several other enumerative problems can be solved using the tropical geometry approach, and the tropical side can be \textit{refined}, in the sense that multiplicities can be changed into polynomial ones and we still get an invariant count of solution, making refined invariants appear in several other situations, naturally, or less naturally, see for example \cite{schroeter2018refined}, \cite{blechman2019refined} or \cite{bousseau2019tropical}. Conjecturally, the refined invariant in our situation should coincide with the refinement of the Euler characteristic of Severi degrees by the Hirzebrich genera, as proposed by L.~G\"ottsche and V.~Shende in \cite{gottsche2014refined}. Some advances in that direction have been made by J.~Nicaise, S.~Payne and F.~Schroeter in \cite{nicaise2018tropical}.\\

This conjectural meaning, does not prevent these invariants to bear other interpretations, although it would probably emphasize some deep connection between them. One of these other meanings is given by G.~Mikhalkin in \cite{mikhalkin2017quantum}. He provides a way of refining the count of real oriented curves passing through some configuration of point on the boundary of a toric surface, and proved that the computation in some particular case can be dealt with tropically using the refined multiplicities. Details are explained below.

\subsection{Quantum indices of real curves}

In \cite{mikhalkin2017quantum}, Mikhalkin introduced a quantum index for oriented real curves. However, the definition in \cite{mikhalkin2017quantum} is restricted to real oriented curves which have real or purely intersection points with the toric boundary. Let $\varphi:\CC P^1\rightarrow\CC\Delta$ be an oriented real rational curve of degree $\Delta$ and $\mathrm{Sq}:\CC\Delta\rightarrow\CC\Delta$ be the extension of the square map $z^n\mapsto z^{2n}$ to $\CC\Delta$. We say that $\varphi$ has real or purely imaginary intersection points with the toric boundary if all the intersection points of $\mathrm{Sq}\circ\varphi$ with the toric boundary are real. This means that the coordinates of the intersection points, which are the evaluations of some monomials $\chi^m$, are real or purely imaginary. The definition can be extended to any real oriented curve provided that we allow a slight modification.\\

In the case of real or purely imaginary intersection points, the quantum index is defined as the log-area of the curve, which happens to be an half-multiple of $\pi^2$. Let $\varphi:\CC P^1\dashrightarrow N\otimes\CC^*$ be an oriented real rational curve. The dense torus orbit $N\otimes\CC^*$ is endowed with the logarithmic map, and the argument map:
$$\text{Log}:z^n\in N\otimes \CC^*\longmapsto \log|z|n\in N\otimes\RR,$$
$$\arg :z^n\in N\otimes \CC^*\longmapsto \arg(z)n\in N\otimes(\RR/2\pi\ZZ).$$
The image of a curve under these maps are respectively called the \textit{amoeba} and the \textit{coamoeba}. Let $\omega$ be a generator of $\Lambda^2 M$, which is thus a non-degenerate $2$-form on $N$. It extends to respective volume forms $\omega_{|\bullet|}$ and $\omega_\theta$ on the vector space $N_\RR$ and the argument torus $N\otimes(\RR/2\pi\ZZ)$. We can look at the (signed) area of the amoeba and coamoeba for these area forms. Let $\HH=\{\mathfrak{Im}t>0\}\subset\CC P^1$ be the Poincar\'e half-plain, which is one of the two connected components of $\CC P^1\backslash\RR P^1$. We define
$$\mathcal{A}_\text{Log}(\HH,\varphi)=\int_\HH \varphi^*\text{Log}^*\omega_{|\bullet|},$$
$$\mathcal{A}_{\arg} (\HH,\varphi) = \int_\HH \varphi^*\mathrm{arg}^*\omega_\theta.$$

The reason to integrate only over $\HH$ and not $\CC P^1$ is that the integral over the conjugated connected component $\overline{\HH}$ is the opposite, and the whole area is zero. Moreover, these two areas are equal:
$$\mathcal{A}_\text{Log}(\HH,\varphi)=\mathcal{A}_{\arg} (\HH,\varphi).$$
Their common value is denoted $\mathcal{A}(\HH,\varphi)$.\\

If the intersection points with the toric boundary are real or purely imaginary, these common area are an half-integer multiple of $\pi^2$. Otherwise, we can shift the log-area to get an half-integer multiple of $\pi^2$ and we have the following Theorem, which is a variation of Mikhalkin's theorem in \cite{mikhalkin2017quantum}. 

\begin{theom}(Mikhalkin, \ref{prop existence quantum index})
Let $\varphi:\CC P^1\rightarrow\CC\Delta$ be a oriented real rational curve. Let $\varepsilon_j\theta_j$, with $\varepsilon_j=\pm 1$ and $0<\theta_j<\pi$, be the arguments of the complex intersection points of $\HH$ with the toric boundary. Then one has
$$\mathcal{A}(\HH,\varphi)-\pi\sum \varepsilon_j(2\theta_j-\pi) =k(\HH,\varphi)\pi^2\in\frac{\pi^2}{2}\ZZ.$$
The half-integer $k(\HH,\varphi)$ is called the quantum index of the oriented curve.
\end{theom}

\begin{rem}
If the complex intersection points are purely imaginary, we recover the fact from \cite{mikhalkin2017quantum} that the log-area is an half-integer multiple of $\pi^2$ since the correction term is zero.
\end{rem}

Due to the additive character of the area, the quantum index is also additive. This can be used to compute the quantum index of a family of curves close to the tropical limit.\\

Furthermore, the behavior of the log-area with respect to monomial maps allows one to compute the quantum index of any rational curve with at least one real intersection point with the toric boundary, using only the computation of rational curves of degree less than $2$, namely: a line, a parabola, an ellipse tangent to one of the coordinate axis not intersecting the other ones. This allows us to provide a formula to compute the log-area of any rational curve provided there is at least one real intersection point with the toric boundary. See Theorem \ref{log-area rational curve}.\\

\subsection{Refined enumerative geometry} 

Now, let $\mathcal{P}$ be a real configuration, \textit{i.e.} consisting of real points and pairs of complex conjugated points, of $m$ points on the toric boundary, such that any toric divisor contains a number of points equal to the integral length of the corresponding side in the polygon $P_\Delta$. Label the sides of $P_\Delta$ by $1,\dots,p$. Let $s_i$ be the number of pairs of complex conjugated points on the toric divisor corresponding to the $i$th side. The Vi\`ete formula ensures that there exists a curve of degree $\Delta$ passing through $\mathcal{P}$ if and only if the points of $\mathcal{P}$ satisfy the Menelaus condition, which we therefore assume: the product of the coordinates of the points is $(-1)^m$. Concretely, this means that the position of the last point is forced by the choice of the other points.\\ 

Let $\mathcal{S}(\mathcal{P})$ be the set of oriented real rational curves of degree $\Delta$ such that for each point $p\in\mathcal{P}$, the curve passes through $p$ or $-p$. Notice that if a real curve passes through a complex point, it also passes through the conjugated point. This set of curves could also be defined in the following way: recall the square map $\mathrm{Sq}:\CC\Delta\rightarrow\CC\Delta$, then $\mathcal{S}(\mathcal{P})$ is the set of oriented real rational curves $C$ such that $\mathrm{Sq}(C)$ passes through $\mathrm{Sq}(\mathcal{P})$. The curves $\mathrm{Sq}(C)$ are tangent to the toric divisors at all intersection points.\\

For an oriented curve $(S,\varphi)\in\mathcal{S}(\mathcal{P})$, with $\varphi:\CC P^1\rightarrow \CC C$ a parametrization of the curve, and the orientation given by $S\subset\CC P^1\backslash\RR P^1$ one of the connected components. We consider the composition of the parametrization $\varphi$ with the logarithmic map, restricted to the real locus:
$$\text{Log}|\varphi|:\RR P^1\rightarrow N\otimes\RR=N_\RR.$$
Its image is $\text{Log}(\varphi(\RR P^1))\subset N_\RR$. We now consider the logarithmic Gauss map that associates to each point of $\RR P^1$ the tangent direction to $\text{Log}(\varphi(\RR P^1))$, which is a point in $\PP^1(N_\RR)\simeq\RR P^1$, oriented by $\omega$. Since both copies of $\RR P^1$ are oriented, one can consider the degree $\text{Rot}_\text{Log}(S,\varphi)$ of the logarithmic Gauss map, \textit{i.e. }when the domain $\RR P^1$ is oriented by the choice of the complex orientation of the curve, and the target $\mathbb{P}^1(N_\RR)\simeq\RR P^1$ oriented by $\omega$. Let $\sigma(S,\varphi)=(-1)^\frac{m-\text{Rot}_\text{Log}(S,\varphi)}{2}$, which is equal to $\pm 1$. We then set
$$R_{\Delta,k}(\mathcal{P})=\sum_{(S,\varphi)\in\mathcal{S}_k(\mathcal{P})} \sigma(S,\varphi)\in\ZZ,$$
where $\mathcal{S}_k(\mathcal{P})$ denotes the subset of $\mathcal{S}(\mathcal{P})$ of oriented real rational curves having quantum index $k$. Finally, we define
$$R_\Delta(\mathcal{P})=\frac{1}{4}\sum_{k}R_{\Delta,k}(\mathcal{P})q^k\in \ZZ\left[ q^{\pm\frac{1}{2}}\right].$$

Although it is not stated in these terms, the following invariance statement is proven by Mikhalkin in \cite{mikhalkin2017quantum}. More precisely, using the given above definition of quantum indices, the exact same proof applies.

\begin{theo}[Mikhalkin\cite{mikhalkin2017quantum}]
As long as the configuration $\mathcal{P}$ is generic, the Laurent polynomial $R_{\Delta}(\mathcal{P})$ only depends on $s=(s_1,\dots,s_p)$.
\end{theo}

The above Laurent polynomial only depending on $s$ is denoted by $R_{\Delta,s}$.

\begin{rem}
It is important for the result and for its proof to consider not only curves passing through $\mathcal{P}$ but also through the symmetric points, otherwise the invariance might fail.
\end{rem}

In the case of a totally real configuration of points, \textit{i.e.} $s_i=0$, using the correspondence theorem, Mikhalkin proved that the invariant $R_{\Delta,(0)}$ coincides, up to a normalization, with the tropical refined invariant $N_\Delta^{\partial,\text{trop}}$. This invariant is obtained as follows. (For a more complete description, see section \ref{tropical enumerative inv}.) We consider tropical curves of degree $\Delta\subset N$. For each vector $n_j$ in $\Delta$, one chooses  line directed by $n_j$, and such that the configuration of lines is generic. Then we count rational tropical curves of degree $\Delta$ whose unbounded ends are contained in the chosen lines, using the Block-G\"ottsche multiplicities from \cite{block2016refined}. The result does not depend on the chosen configuration, and is denoted by $N_\Delta^{\partial,\text{trop}}$.

\begin{theo}[Mikhalkin\cite{mikhalkin2017quantum}]
One has
$$R_{\Delta,(0)}=(q^{1/2}-q^{-1/2})^{m-2}N_\Delta^{\partial,\text{trop}}.$$
\end{theo}

\begin{rem}
The exponent $m-2$ corresponds to the number of vertices of a tropical curve involved in the enumerative problem defining $N_\Delta^{\partial\text{trop}}$. Thus, the normalization $(q^{1/2}-q^{-1/2})^{m-2}$ amounts to clear the denominators of the Block-G\"ottsche multiplicities from \cite{block2016refined}. We recall their definition in section \ref{tropical enumerative inv}.
\end{rem} 

The results of Mikhalkin reduce the computation of the invariants $R_{\Delta,(0)}$ to a tropical computation of $N_\Delta^{\partial,\text{trop}}$. In this paper we prove that the computation all the $R_{\Delta,s}$ can also be carried out using tropical geometry.\\

More precisely, we have the following theorem. Let $\Delta=(n_j)\subset N$ still be a family of primitive vectors whose total sum is zero, allowing us to consider curves of degree $\Delta$ inside the toric surface $\CC\Delta$, and choose decompositions $r_j+2s_j=l(n_j)$. We denote by $\Delta(s)$ the family $(\Delta\backslash\{n_j^{2s_j}\}_j)\cup\{(2n_j)^{s_j}\}_j$, where $2s_j$ copies of $n_j$ are replaced by $s_j$ copies of $2n_j$. Assume that for some $i$, we have $r_i>0$, meaning there is at least one real intersection point with the toric boundary. The degree $\Delta(s)$, not consisting only of primitive vectors anymore, still allows us to consider tropical curves of degree $\Delta(s)$, and the associated refined invariant $N_{\Delta(s)}^{\partial,\text{trop}}$.

\begin{theom}[\ref{theorem paper}]
For $\Delta,s$ chosen as above, the refined invariants satisfy
$$R_{\Delta,s}=2^{|s|}\frac{(\qd)^{m-2-|s|}}{(q-q^{-1})^{|s|}}N^{\partial,\text{trop}}_{\Delta(s)} , $$
where $|s|=\sum_1^p s_i$ is the total number of complex pairs.
\end{theom}

This paper is a sequel to the paper \cite{blomme2020tropical} in which we computed the invariants $R_{\Delta,s}$ when $s$ is of the form $(s_1,0,\dots,0)$, meaning that all complex points lie on the same toric divisor, and complex points were purely imaginary. Briefly, generic configurations $\mathcal{P}$ for which the refined count gives the value of the invariant need to be regular values of the function that maps a curve to the coordinates of its intersection point with the toric boundary. Moreover, one previously needed that intersection points with the toric boundary were purely imaginary to define the quantum index as the logarithmic area. For those specific values of $s$, one can show the existence of regular values for which complex points are purely imaginary. The purely imaginary assumption also prevents an effective computation using tropical geometry. The new general definition of the quantum index to the case of curves with non purely imaginary intersection points with the toric boundary allows one to compute the invariant in the general case.\\

The paper is organized as follows. In the second section we recall the basics about tropical curves, real tropical curves and tropicalization. We also describe the possible real structures one can put on a rational tropical curve. Then, the third section is devoted to the definition of quantum indices in the general setting, \textit{i.e.} complex intersection points with the boundary may not be purely imaginary. We also compute the quantum index of  some auxiliary curves, leading to a general formula in Theorem \ref{log-area rational curve}. In the fourth section we define properly the classical and tropical enumerative problems that give birth to refined invariants in both settings, allowing us to state our main result. The rest of the paper is devoted to the proof of this same result. In the fifth section we prove an extension of the correspondence theorem from \cite{blomme2020tropical}, which was itself a real version of the correspondence theorem from I.~Tyomkin \cite{tyomkin2017enumeration}. In the last section, we use this theorem and make the refined count of solutions close to a tropical curve to relate the multiplicities in both enumerative problems and finally prove Theorem \ref{theorem paper}. \\

{\it Acknowledgments} I am grateful to Ilia Itenberg for numerous discussions leading to this result, and to Grigory Mikhalkin for several explanations about the quantum indices. The author is funded by an ENS PhD grant.

\section{Tropical curves  and real tropical curves}

The present section, included for completeness, can also be found in \cite{blomme2020tropical}.

	\subsection{Real abstract tropical curves}

Let $\overline{\Gamma}$ be a finite connected graph without bivalent vertices. Let $\overline{\Gamma}_\infty^0$ be the set of $1$-valent vertices of $\overline{\Gamma}$, and $\Gamma=\overline{\Gamma}\backslash\overline{\Gamma}^0_\infty$. If $m$ denotes the cardinal of $\overline{\Gamma}_\infty^0$, its elements are labeled with integers from $[\![1;m]\!]$. We denote by $\Gamma^0$ the set of vertices of $\Gamma$, and by $\Gamma^1$ the set of edges of $\Gamma$. The non-compact edges resulting from the eviction of $1$-valent vertices are called \textit{unbounded ends}. The set of unbounded ends is denoted by $\Gamma^1_\infty$, while its complement, the set of bounded edges, is denoted by $\Gamma_b^1$. Notice that $\Gamma_\infty^1$ is also labeled by $[\![1;m]\!]$. Let $l:\gamma\in\Gamma_b^1\mapsto |\gamma|\in\RR_+^*=]0;+\infty[$ be a function, called length function. It endows $\Gamma$ with the structure of a metric graph by decreting that a bounded edge $\gamma$ is isometric to $[0;|\gamma|]$, and an unbounded end is isometric to $[0;+\infty[$.

\begin{defi}
Such a metric graph $\Gamma$ is called an \textit{abstract tropical curve}.
\end{defi}

An isomorphism between two abstract tropical curves $\Gamma$ and $\Gamma'$ is an isometry $\Gamma\rightarrow\Gamma'$. In particular an automorphism of $\Gamma$ does not necessarily respect the labeling of the unbounded ends since it only respects the metric. Therefore, an automorphism of $\Gamma$ induces a permutation of the set $I=[\![1;m]\!]$ of unbounded ends.

\begin{defi}
Let $\Gamma$ be an abstract tropical curve. A \textit{real structure} on $\Gamma$ is an involutive isometry $\sigma:\Gamma\rightarrow\Gamma$. A \textit{real abstract tropical curve} is an abstract tropical curve enhanced with a real structure.
\end{defi}

Since a real structure $\sigma:\Gamma\rightarrow\Gamma$ has to preserve the metric, for any bounded edge $\gamma$, one has $|\gamma|=|\sigma(\gamma)|$. The real structure also induces an involution on the set of ends $I=[\![1;m]\!]$ of $\Gamma$. The fixed ends are called \textit{real ends} and the pairs of exchanged ends are called the \textit{conjugated ends}, or \textit{complex ends}. The fixed locus of $\sigma$ is denoted by $\fix$. It is a subgraph of $\Gamma$.\\

\begin{expl}
\begin{itemize}[label=-]
\item The trivial real structure $\sigma=\text{id}_\Gamma$ is the most common example, useful despite its simplicity.
\item If $\Gamma$ is an abstract tropical curve and $e,e'\in\Gamma^1_\infty$ are two unbounded ends adjacent to the same vertex $w$, another example is given by permuting the two unbounded ends $e$ and $e'$, and leaving the rest of the graph invariant.
\item See Figure \ref{real tropical curve} for another example. Two exchanged edges are draw nearby each other.
\end{itemize}
\end{expl}

	\subsection{Real parametrized tropical curves}

Recall that we have two dual lattices $N$ and $M$, and $N_\RR=N\otimes\RR$. We now define parametrized tropical curves in $N_\RR$.

\begin{defi}
A \textit{parametrized tropical curve} in $N_\RR\simeq\RR^2$ is a pair $(\Gamma,h)$, where $\Gamma$ is an abstract tropical curve and $h:\Gamma\rightarrow\RR^2$ is a map satisfying the following requirements:
\begin{itemize}
\item For every edge $E\in\Gamma^1$, the map $h|_E$ is affine. If we choose an orientation of $E$, the value of the differential of $h$ taken at any interior point of $E$, evaluated on a tangent vector of unit length, is called the slope of $h$ alongside $E$. This slope must lie in $N$.
\item We have the so called \textit{balancing condition}: at each vertex $V\in\Gamma^0$, if $E$ is an edge containing $V$, and $u_E$ is the slope of $h$ along $E$ when $E$ is oriented outside $V$, then 
$$\sum_{E:\partial E\ni V} u_E=0\in N.$$
\end{itemize}
\end{defi}

Two parametrized curves $h:\Gamma\rightarrow N_\RR$ and $h':\Gamma'\rightarrow N_\RR$ are isomorphic if there exists an isomorphism of abstract tropical curves $\varphi:\Gamma\rightarrow\Gamma'$ such that $h=h'\circ\varphi$.\\

\begin{defi}
A \textit{real parametrized tropical curve} is a triplet $(\Gamma,\sigma,h)$, where $(\Gamma,h)$ is a parametrized tropical curve, $\sigma$ is a real structure on $\Gamma$, and $h$ is $\sigma$-invariant: $h\circ\sigma=h$.
\end{defi}

\begin{rem}
In particular, two vertices that are exchanged by $\sigma$ have the same image under $h$, and two edges that are exchanged by $\sigma$ have the same slope and the same image. Such edges are called \textit{double edges}. If they are unbounded, we call them a \textit{double end}. Thus, the image $h(\Gamma)\subset N_\RR$ may not be sufficient to recover $\Gamma$ and the real structure, since for instance there is no way of distinguishing a double end from a simple end with twice their slope.
\end{rem}

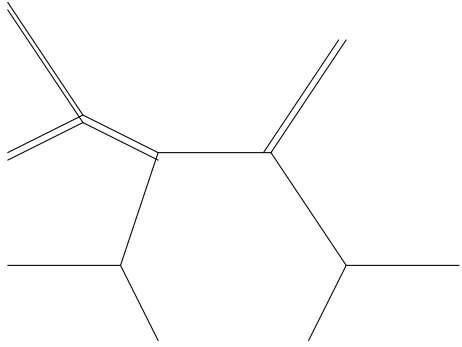
\begin{figure}
\begin{center}

\begin{tikzpicture}[line cap=round,line join=round,>=triangle 45,x=0.5cm,y=0.5cm]
\clip(0.,0.) rectangle (12.,9.);
\draw (0.,9.)-- (2.,6.);
\draw (2.,6.)-- (0.,5.);
\draw (2.,6.)-- (4.,5.);
\draw (0.,8.8)-- (2.,5.8);
\draw (2.,5.8)-- (0.,4.8);
\draw (2.,5.8)-- (4.,4.8);

\draw (4.,5.)-- (3.,2.);
\draw (3.,2.)-- (0.,2.);
\draw (3.,2.)-- (4.,0.);
\draw (4.,5.)-- (7.,5.);

\draw (7.,5.)-- (9.,8.);
\draw (6.8,5.)-- (8.8,8.);

\draw (7.,5.)-- (9.,2.);
\draw (9.,2.)-- (8.,0.);
\draw (9.,2.)-- (12.,2.);
\end{tikzpicture}
\caption{\label{real tropical curve} Abstract real tropical curve with its real structure depicted by doubling the exchanged edges.}
\end{center}
\end{figure}

\begin{rem}
We could assume that $M=N=\ZZ^2$, but the distinction is now useful since the lattice $M$ is a set of functions on the space $N_\RR$ where the tropical curves live, while $N$ is the space of the slopes of the edges of a tropical curve. Moreover, notice that we deal with tropical curves in the affine space $N_\RR$, identified with its tangent space at $0$.
\end{rem}

If $e\in\Gamma^1_\infty$ is an unbounded end of $\Gamma$, let $n_e\in N$ be the slope of $h$ alongside $e$, oriented out of its unique adjacent vertex, \textit{i.e.} toward infinity. The multiset
$$\Delta=\{n_e\in N|e\in\Gamma^1_\infty \}\subset N,$$
is called the \textit{degree} of the parametrized curve. It is a multiset since an element may appear several times. Let $(e_1,e_2)$ be a basis of $N$. We say that a curve is of degree $d$ if $\Delta=\Delta_d=\{(-e_1)^d,(-e_2)^d,(e_1+e_2)^d\}$. Using the balancing condition, one can show that $\sum_{n_e\in\Delta}n_e=0$.\\


\begin{defi}
\begin{itemize}
\item[-] Let $\Gamma$ be an abstract tropical curve. The \textit{genus} of $\Gamma$ is its first Betti number $b_1(\Gamma)$.
\item[-] A curve is \textit{rational} if it is of genus $0$.
\item[-] A parametrized tropical curve $(\Gamma,h)$ is \textit{rational} if $\Gamma$ is rational.
\end{itemize}
\end{defi}

\begin{rem}
A parametrized tropical curve is then rational if the graph that parametrizes it is a tree.
\end{rem}

	\subsection{Plane tropical curves}
	
	We now look at plane tropical curves, which are the images of the parametrized tropical curves. To define a plane tropical curve, we consider a \textit{tropical polynomial}: for $x\in N_\RR$, put
	$$P(x)=\max_{u\in P_\Delta}(a_u+\langle u,x\rangle ),$$
	where $P_\Delta\subset M$ is a convex lattice polygon, and $a_u\in\RR\cup\{-\infty\}$ are the coefficients of the polynomial, different from $-\infty$ if $m\in M$ is a vertex of $P_\Delta$. The polygon $P_\Delta$ is called the \textit{Newton polygon} of the plane tropical curve. If we choose a basis of $M$ and $N$, the tropical polynomial $P$ takes the following form:
	$$P(x,y)=\max_{(i,j)\in P_\Delta}(a_{ij}+ix+jy).$$
	The tropical polynomial $P$ is then a piecewise affine convex function which is the maximum of a finite number of affine functions. We assume that $P_\Delta$ contains more than one point, otherwise $P$ is an affine function. The tropical polynomial $P$ induces a subdivision of $P_\Delta$ with the following rule: $u,u'\in P_\Delta$ are connected by an edge if $\{x\in N_\RR : P(x)=a_u+\langle u,x\rangle=a_{u'}+\langle u',x\rangle\}\neq\emptyset$. The \textit{corner locus} $C$ of $P$, \textit{i.e.} the set where at least two of the affine functions realize the maximum, is a rectilinear graph in $N_\RR$. Equivalently, this is the set of points where $P$ is not differentiable. The subdivision of $P_\Delta$ induced by $P$ is \textit{dual} to the corner locus $C$ in the sense that there exists natural bijections between the following pairs of sets: edges of $C$ and edges of the subdivision, vertices of $C$ and polygons of the subdivision, components where $P$ is smooth equal to one of the affine functions and vertices of the subdivision.
	
	\begin{defi}
The \textit{plane tropical curve} $C$ associated to a tropical polynomial $P$ is the corner locus of $P$, enhanced with the following weights on the edges of $C$: the weight of an edge is the lattice length of the dual edge in the subdivision of $P_\Delta$. The polygon $P_\Delta$ is called the \textit{degree} of the curve $C$.
\end{defi}

Plane tropical curves can be characterized as finite weighted graphs (weights on the edges) with unbounded ends in $N_\RR$, such that the edges are affine with slope in $N$, and the vertices satisfy the following balancing condition: if $E$ is an edge of weight $w_E$ adjacent to a vertex $V$, and $u_E$ is a primitive lattice vector directing $E$ oriented outward from $V$, we have
$$\sum_{E\ni V}w_Eu_E=0.$$
For more details on plane tropical curves, see \cite{brugalle2014bit}. \\

One can show that the image of a parametrized tropical curve is indeed a plane tropical curve with this definition. Moreover, the relation between $\Delta$, degree of the parametrized tropical curve, and $P_\Delta$, degree of the plane curve, is as described in the introduction. Conversely, a plane tropical curve can always be parametrized by an abstract tropical curve, leading to a parametrized tropical curve. Moreover, if $C$ is a plane curve parametrized by $h:\Gamma\rightarrow N_\RR$, the weight of an edge $E$ of $C$ can be recovered as the sum of the lattice lengths of the slopes of $h$ on the edges $\gamma\in\Gamma$ which project onto $E$. However, there are often many ways of choosing a parametrization of a plane tropical curve, by non-isomorphic tropical curves. In fact, even the degree $P_\Delta\subset M$ does not uniquely determine the degree $\Delta\subset N$ of a parametrized curve parametrizing $C$: \textit{e.g.} if $C$ has un unbounded end of weight $2$, a parametrizing graph $\Gamma$ could have either an end $e$ with $h|_e$ having slope of lattice length $2$, or two ends of primitive slope with the same image by $h$.\\

We now define the usual concepts associated to plane classical curves in the case of plane tropical curves, starting with reducible curves. 

\begin{defi}
\begin{itemize}[label=-]
\item A plane tropical curve is \textit{reducible} if it can be represented as the union of two distinct plane tropical curves. 
\item A plane tropical curve is \textit{irreducible} if it is not reducible.
\end{itemize}
\end{defi}

Going on with the definition of the genus, one needs to be careful since it depends on the chosen parametrization.

\begin{defi}
\begin{itemize}[label=-]
\item The \textit{genus} of a plane tropical curve is the smallest genus among its possible parametrizations.
\item A plane tropical curve is \textit{rational} if it is irreducible and can be parametrized by a rational tropical curve.
\end{itemize}
\end{defi}

One can show that if $C$ is an irreducible rational plane tropical curve with unbounded ends of weight $1$, it admits a unique rational parametrization. More generally, we have the following statement.

\begin{prop}\cite{mikhalkin2005enumerative}\label{unique parametrization}
Let $C$ be a rational plane tropical curve, and let $u_e$ be a directing primitive lattice vector for each unbounded end $e$, oriented toward infinity. Let $w_e$ be the weight of $e$. Then $C$ is the image of a unique rational parametrized tropical curve of degree $\Delta=\{w_e u_e\}_e$.
\end{prop}

\subsection{Real parametrizations of a plane tropical curve}
 
In this subsection, we extend Proposition \ref{unique parametrization} by describing the possible real rational parametrizations of an irreducible rational plane tropical curve, with unbounded ends of weights $1$ or $2$.\\

Let $C$ be a rational plane tropical curve with unbounded ends of weight $1$ or $2$. Let $u_1$, $\dots$, $u_r$, $2v_1$, $\dots$, $2v_s$ be the weighted directing vectors of the unbounded ends of $C$, with vectors $u_i,v_j$ being primitive vectors in $N$. We assume that $r\geqslant 1$. Let $h_0:\Gamma_0\rightarrow N_\RR$ be the unique rational parametrization of $C$ given by Proposition \ref{unique parametrization}, which is of degree $\{u_i,2v_j\}_{i,j}$. We now describe the parametrizations of $C$ by real parametrized rational curves of degree $\{u_i,v_j^2\}_{i,j}$, which means that now all vectors are primitive, and each unbounded end of weight $2$ is replaced with two ends of weight $1$.\\



We define the subgraph $\Gamma_\text{even}$ of $\Gamma_0$ as the minimal subgraph satisfying both following requirements:
	\begin{itemize}[label=-]
	\item Every unbounded end of $\Gamma_0$ of even weight (\textit{i.e.} mapped to an end of $C$ directed by $2v_j$ for some $j$) is in $\Gamma_\text{even}$.
	\item If $V$ is a vertex of $\Gamma_0$ and all edges adjacent to $V$ but one are in $\Gamma_\text{even}$, then the remaining adjacent edge also is in $\Gamma_\text{even}$. Following \cite{shustin2004tropical}, such a vertex is called an \textit{extendable vertex}.
	\end{itemize}

\begin{rem}
The subgraph $\Gamma_\text{even}$ is the maximal graph on which we can "cut $\Gamma_0$ in two" in order to obtain a new graph $\Gamma$, used to parametrize $C$. Notice that on all the edges of $\Gamma_\text{even}$, the map
$h_0$ has an even slope.
\end{rem}

As $C$ admits at least one odd unbounded end, each connected component $\Gamma_\text{even}^i$ of $\Gamma_\text{even}$ contains a unique \textit{stem}, which is a non-extendable vertex. We orient the edges of $\Gamma_\text{even}^i$ away from the stem. Then we say that a subset of points $\mathcal{R}_i\subset\Gamma_\text{even}^i$ is \textit{admissible} if no point of $\mathcal{R}_i$ is joint to another by an oriented path, and for each unbounded end $e$ in $\Gamma_\text{even}^i$, there is at least (and thus exactly one) point of $\mathcal{R}_i$ on the shortest path between the stem and $e$. Let $\mathcal{R}=\bigcup_i\mathcal{R}_i$. We then define a real abstract tropical curve $(\Gamma(\mathcal{R}),\sigma)$ with a map $h_\mathcal{R}:\Gamma(\mathcal{R})\rightarrow N_\RR$ that factors through $\Gamma(\mathcal{R})\rightarrow \Gamma_0\rightarrow N_\RR$ and makes it a real parametrized tropical curve.\\

Let $\Gamma_\text{fix}(\mathcal{R})$ be the closure of the union of the connected components of $\Gamma_0-\mathcal{R}$ not containing any even end. The abstract tropical curve $\Gamma(\mathcal{R})$ is obtained as the disjoint union of two copies of $\Gamma_0$, glued along $\Gamma_\text{fix}(\mathcal{R})$:
$$\begin{tikzcd}
\Gamma_\text{fix}(\mathcal{R}) \arrow{r} \arrow{d} & \Gamma_0\arrow{d} \\
\Gamma_0 \arrow{r} & \Gamma(\mathcal{R}) \\
\end{tikzcd}.$$

 In other terms, $\Gamma(\mathcal{R})=\Gamma_0\coprod_{\Gamma_\text{fix}(\mathcal{R})}\Gamma_0$. It means that we have doubled the components of $\Gamma_0-\mathcal{R}$ containing the even ends. We denote by $\pi:\Gamma(\mathcal{R})\rightarrow\Gamma_0$ the map obtained by gluing the identity maps of $\Gamma_0$. The complement of $\Gamma_\text{fix}(\mathcal{R})$ in $\Gamma_0$ is called the \textit{splitting graph}. It is a subset of $\Gamma_\text{even}$. The splitting graph is maximal if its closure is equal to $\Gamma_\text{even}$. The length function on $\Gamma(\mathcal{R})$ is defined as follows: we consider points of $\mathcal{R}$ as vertices of $\Gamma_0$, then, the length of an edge $\gamma$ of $\Gamma(\mathcal{R})$ is the length of its image $\pi(\gamma)$ if it is an edge of $\Gamma_\text{fix}(\mathcal{R})$ and twice the length of $\pi(\gamma)$ otherwise. The involution $\sigma$ is the automorphism of $\Gamma(\mathcal{R})$ that exchanges the two antecedents whenever there are two. The parametrized map $h_\mathcal{R}:\Gamma(\mathcal{R})\rightarrow N_\RR$ is the composition of $\pi$ and $h_0$.
 
\begin{rem}
The map $\pi$ really looks like a tropical cover, as defined in \cite{cavalieri2010tropical} and \cite{buchholz2015tropical}. However, it is not always the case. This is normal since the purpose of the notion of tropical cover is to mimick ramified covers between complex curves. The map $\pi$ here plays the role of the quotient map by a real involution, which is not a ramified cover.
\end{rem}

Let $\gamma$ be an edge of $\Gamma(\mathcal{R})$, and $n\in N$ be the slope of $\pi(\gamma)$. Then, one can easily check that the choice of length on $\Gamma(\mathcal{R})$ ensures that $h_\mathcal{R}$ has slope $n$ if $\gamma\in\fix$ and $\frac{n}{2}$ otherwise. However, as the edges of $\Gamma_\text{even}$ have an even slope, it is still an element of $N$. One can check that the balancing condition is still satisfied. Therefore, $(\Gamma(\mathcal{R}),h_\mathcal{R},\sigma)$ is a real parametrized tropical curve, of image $C$, and of degree $\{u_i,v_j^2\}_{i,j}$.

\begin{prop}\cite{blomme2020tropical}\label{realstruc}
Let $C$ be an irreducible rational plane tropical curve of degree $P_\Delta\subset M$ having unbounded ends of weight $1$ or $2$. Let $\Delta\subset N$ be the degree associated to $P_\Delta$ consisting only of primitive lattice vectors. let $h_0:\Gamma_0\rightarrow N_\RR$ be the unique rational parametrization of $C$ given by Proposition \ref{unique parametrization}. Using previous notations, every real rational parametrized curve of degree $\Delta$ having the image $C$ is one of the curves $\Gamma(\mathcal{R})$.
\end{prop}

	\subsection{Moment of an edge}
	
	Let $\omega$ be a generator of $\Lambda^2 M$, \textit{i.e.} a non-degenerated $2$-form on $N$. It extends to a volume form on $N_\RR\simeq\RR^2$. Let $e\in\Gamma^1_\infty$ be an unbounded end oriented toward infinity, directed by $n_e$. Then the moment of $e$ is the scalar
	$$\mu_e=\omega(n_e,p)\in\RR,$$
	where $p\in e$ is any point on the edge $e$. Remember that we identify the affine space $N_\RR$ with its tangent space at $0$, allowing us to plug in $\omega$ a tangent vector $n_e$ and a point $p$. We similarly define the moment of a bounded edge if we specify its orientation. The moment of a bounded edge is reversed when its orientation is reversed.\\
	
	Intuitively, the moment of an unbounded end is just a way of measuring its position alongside a transversal axis. Thus, fixing the moment of an unbounded end amounts to impose on the curve that it goes through some point at infinity, or equivalently the unbounded end is contained in a fixed line. In a way, this allows us to do toric geometry in a compactification of $N_\RR$ but staying in $N_\RR$. It provides a coordinate on the components of the toric boundary without even having to introduce the concept of toric boundary in the tropical world. Following this observation, the moment has also a definition in complex toric geometry, where it corresponds to the coordinate of the intersection point of the curve with the toric divisor. Let

$$\begin{array}{rccl}
\varphi: & \CC P^1 & \dashrightarrow & N\otimes\CC^*\simeq(\CC^*)^2 \\
 & t & \mapsto & \chi\prod_{1}^r(t-\alpha_j)^{n_j}. \\
 \end{array}$$
	 be a parametrized rational curve. This is a curve of degree $\Delta=(n_j)\subset N$. The degree $\Delta$ defines a fan $\Sigma_\Delta$ and a toric surface $\CC\Delta$ to which the map $\varphi$ naturally extends. The toric divisors $D_k$ of $\CC\Delta$ are in bijection with the rays of the fan, which are directed by the vectors $n_j$. Several vectors $n_j$ may direct the same ray. Moreover, the map $\varphi$ extends to the points $\alpha_j$ by sending $\alpha_j$ to a point on the toric divisor $D_k$ corresponding to the ray directed by $n_j$. A coordinate on $D$ is a primitive monomial $\chi^m\in M$ in the lattice of characters such that $\langle m,n_j\rangle=0$. This latter equality ensures that the monomial $\chi^m$ extends on the divisor $D_k$. If $n_j$ is primitive, $\iota_{n_j}\omega\in M$ is such a monomial, and then the complex moment is the evaluation of the monomial at the corresponding point on the divisor:
	$$\mu_j=\left(\varphi^*\chi^{\iota_{n_j}\omega}\right)(\alpha_j).$$
	
	The Weil reciprocity law, or an easy computation, gives us the following relation between the moments:
$$\prod_{i=1}^m \mu_i=(-1)^m.$$	
	We could also prove the relation using Vi\`ete formula. In the tropical world we have an analog called the \textit{tropical Menelaus theorem}, which gives a relation between the moments of the unbounded ends of a parametrized tropical curve.
	
	\begin{prop}[Tropical Menelaus Theorem \cite{mikhalkin2017quantum}] For a parametrized tropical curve of degree $\Delta$, we have
	$$\sum_{n_e\in\Delta} \mu_e =0.$$
	\end{prop}
	
	In the tropical case as well as in the complex case, a configuration of $m$ points on the toric divisors is said to satisfy the \textit{Menelaus condition} if this relation is satisfied. Be careful that in tropical case, the moment of a point depends on the vector of $M$ used to compute it: it is of lattice length $1$ if the point is real, and of lattice length $2$ for a non-real point.\\
	
	If the slope $n_j$ of the edge is non-primitive, it corresponds to a tangency to the toric divisor. We then distinguish two moments in the following definition.
	\begin{defi}
	\label{definition primitive moment}
	\begin{itemize}
	\item The \textit{"primitive moment"} is computed with the monomial $\iota_{\frac{n_j}{l(n_j)}}\omega$, \textit{i.e.} the associated normal vector but primitive. It corresponds to the coordinate of the point in the toric divisor, where the tangency occurs.
	\item The \textit{"Menelaus moment"}, or briefly the moment, is computed with the monomial $\iota_{n_j}\omega$. It is a power of the primitive moment, and it satisfies the Menelaus theorem.
	\end{itemize}
	\end{defi}
	
	\subsection{Moduli space of tropical curves and refined multiplicity of a simple tropical curve}
	
	Let $(\Gamma,h)$ be a parametrized tropical curve such that $\Gamma$ is trivalent, and has no \textit{flat vertex}. A flat vertex is a vertex whose outgoing edges have their slope contained in a common line. It just means that for any two outgoing edges of respective slopes $u,v$, we have $\omega(u,v)=0$. In particular, when the curve is trivalent, no edge can have a zero slope since it would imply that its extremities are flat vertices. A plane tropical curve is a \textit{simple nodal curve} if the dual subdivision of its Newton polygon consists only of triangles and parallelograms. The unique rational parametrization (given by Proposition \ref{unique parametrization}) of a plane rational nodal curve has a trivalent underlying graph, and has no flat vertex.\\
	
	\begin{defi}
	The \textit{combinatorial type} of a tropical curve is the homeomorphism type of its underlying labeled graph $\Gamma$, \textit{i.e.} the labeled graph $\Gamma$ without the metric.
	\end{defi}
	
	To give a graph a tropical structure, one just needs to specify the lengths of the bounded edges. If the curve is trivalent and has $m$ unbounded ends, there are $m-3$ bounded edges, otherwise the number of bounded edges is $m-3-\text{ov}(\Gamma)$, where $\text{ov}(\Gamma)$ is the \textit{overvalence} of the graph. The overvalence is given by $\sum_V (\text{val}(V)-3)$, where $V$ runs over the vertices of $\Gamma$, and $\text{val}(V)$ denotes the valence of the vertex. Therefore, the set of curves having the same combinatorial type is homeomorphic to $\RR_{\geqslant 0}^{m-3-\text{ov}(\Gamma)}$, and the coordinates are the lengths of the bounded edges. If $\Gamma$ is an abstract tropical curve, we denote by $\text{Comb}(\Gamma)$ the set of curves having the same combinatorial type as $\Gamma$.\\
	
	For a given combinatorial type $\text{Comb}(\Gamma)$, the boundary of $\RR_{\geqslant 0}^{m-3-\text{ov}(\Gamma)}$ corresponds to curves for which the length of an edge is zero, and therefore corresponds to a graph having a different combinatorial type. This graph is obtained by deleting the edge with zero length and merging its extremities. We can thus glue together all the cones of the finitely many combinatorial types and obtain the \textit{moduli space $\mathcal{M}_{0,m}$ of rational tropical curves with $m$ marked points}. It is a simplicial fan of pure dimension $m-3$, and the top-dimensional cones correspond to trivalent curves. The combinatorial types of codimension $1$ are called \textit{walls}.\\
	
	Given an abstract tropical curve $\Gamma$, if we specify the slope of every unbounded end, and the position of a vertex, we can define uniquely a parametrized tropical curve $h:\Gamma\rightarrow N_\RR$. Therefore, if $\Delta\subset N$ denotes the set of slopes of the unbounded ends, the \textit{moduli space $\mathcal{M}_0(\Delta,N_\RR)$ of parametrized rational tropical curves of degree 
$\Delta$} is isomorphic to $\mathcal{M}_{0,m}\times N_\RR$ as a fan, where the $N_\RR$ factor corresponds to the position of the finite vertex adjacent to the first unbounded end.\\
	
	On this moduli space, we have a well-defined evaluation map that associates to each parametrized curve the family of moments of its unbounded ends :
	$$\begin{array}{crcl}
	\text{ev} : & \mathcal{M}_0(\Delta,N_\RR) & \longrightarrow & \RR^{m-1} \\
	 & (\Gamma,h) & \longmapsto & \mu=(\mu_i)_{2\leqslant i\leqslant m}
	\end{array}.$$
	By the tropical Menelaus theorem, the moment $\mu_1$ is equal to the opposite of the sum of the other moments, hence we do not take it into account in the map. Notice that the evaluation map is linear on every cone of $\mathcal{M}_0(\Delta,N_\RR)$. Furthermore, 
both spaces have the same dimension $m-1$. Thus, if $\Gamma$ is a trivalent curve, the restriction of $\text{ev}$ on
$\text{Comb}(\Gamma)\times N_\RR$ has a determinant well-defined up to sign when $\RR^{m-1}$ and $\text{Comb}(\Gamma)\simeq\RR_{\leqslant0}^{m-3}$ are both endowed with their canonical basis, and $N_\RR$ is endowed with a basis of $N$. 
The absolute value $m_\Gamma^\CC$ of the determinant is called the \textit{complex multiplicity} of the curve, well-known to factor into the following product over the vertices of $\Gamma$:
	$$m_\Gamma^\CC=\prod_V m_V^\CC,$$
	where $m_V^\CC=|\omega(u,v)|$ if $u$ and $v$ are the slopes of two outgoing edges of $V$. The balancing condition ensures that $m_V^\CC$ does not depend on the chosen edges. This multiplicity is the one that appears in the correspondence theorem of Mikhalkin \cite{mikhalkin2005enumerative}. Notice that the simple parametrized tropical curves are precisely the points of the cones with trivalent graph and non-zero multiplicity. We finally recall the definition of the refined Block-G\"ottsche multiplicity.
	
	\begin{defi}
	\label{refined multiplicity}
	The \textit{refined multiplicity} of a simple nodal tropical curve is
	$$m^q_\Gamma=\prod_V [m_V^\CC]_q,$$
	where $[a]_q=\frac{q^{a/2}-q^{-a/2}}{q^{1/2}-q^{-1/2}}$ is the $q$-analog of $a$.
	\end{defi}
	
	This refined multiplicity is sometimes called the Block-G\"ottsche multiplicity and intervenes in the definition of the invariant $N_\Delta^{\partial,\text{trop}}$. Notice that the multiplicity is the same for every curve inside a given combinatorial type.

	\subsection{Tropicalization}
	\label{tropicalization}

We briefly recall how to obtain an abstract tropical curve and a parametrized tropical curve from a non-archimedean parametrized curve given by a rational map $f:(C,\textbf{q})\rightarrow\mathrm{Hom}(M,\CC((t))^*)$, where $(C,\textbf{q})$ is a curve with marked points. For more details, see \cite{tyomkin2012tropical}.\\

\subsubsection{Tropicalization of a marked curve} 
 
Let $(C,\textbf{q})$ be a smooth marked curve over $\CC((t))$. Let $\mathcal{C}^{(t)}\rightarrow\text{Spec}\CC[[t]]$ be the stable model of $(C,\textbf{q})$, defined over $\CC[[t]]$. The marked points $q_i$ provide sections $\text{Spec}\CC[[t]]\rightarrow\mathcal{C}^{(t)}$. We have the special fiber $\mathcal{C}^{(0)}$, which is a stable nodal curve, meaning that each irreducible component of genus zero has at least three marked points or nodes, and each irreducible component of genus 1 has at least one marked point or a node. Let $\overline{\Gamma}$ be the dual graph in the following sense: we have one finite vertex per irreducible component of the special fiber, one infinite vertex per marked point, an infinite vertex is joined to the finite vertex of the component where the point specializes, and two finite vertices are joined by an edge if they share a node. We make $\overline{\Gamma}$ into an abstract tropical curve by declaring the length of such an edge to be $l$ if the node is locally given by $xy=t^l$ in an etale neighborhood of the node.

\begin{rem}
Intuitively, if our curve is rational, $\mathcal{C}^{(t)}$ is just $\CC P^1$ with points depending on a small complex parameter $t$ on it, \textit{i.e.} points given by locally convergent Laurent series in $\CC((t))$. If we take naively the special fiber $t=0$, some marked points may collide, \textit{i.e.} specialize on the same point, other may go to infinity, ... Taking the stable model means that we prevent that. For instance, assume a bunch of points specialize to $0$. This means that they are of given by formal series of the form $t^k x(t)$ with $k>0$. We then blow-up this point and get two copies of $\CC P^1$ sharing a node. All the points previously specializing to $0$ now specialize to at least two different points on the exceptional divisor. The length of the edge between the two copies is the smallest $k$ for all the points specializing on it. Concretely, the blow-up amounts to change the coordinate $z$ on $\CC P^1$ by $t^{-k}z$. We then repeat as long as necessary. 
\end{rem}

If the curve $(C,\textbf{q})$ is a real curve, with a real configuration of points $\textbf{q}$, the involution restricted to the special fiber induces a real structure on $\Gamma$.


\subsubsection{Tropicalization of a parametrized curve} Now assume given a rational map $f:(C,\textbf{q})\dashrightarrow\mathrm{Hom}(M,\CC((t))^*)$. There is a tropical curve $\Gamma$ associated to $(C,\textbf{q})$. The rational map $f$ extends to a rational map on the stable model $\mathcal{C}^{(t)}$ of $(C,\textbf{q})$. In order to make $\Gamma$ into a parametrized tropical curve, we define a map $h:\Gamma\rightarrow N_\RR$ in the following way:
\begin{itemize}
\item If $w\in\Gamma^0$ is a vertex dual to a component $C_w$ of $\mathcal{C}^{(0)}$, then $h(w)$ is the element of $N$ defined as follows:
$$h(w)(m)=\text{ord}_{C_w}(f^*\chi^m),$$
where $\text{ord}_{C_w}$ stands for the multiplicity of $C_w$ in the divisor of $f^*\chi^m$.
\item Then $h$ maps a bounded edge to the line segment linking its extremities.
\item If $q_i$ is a marked point, then the slope of the associated unbounded end is $\text{ord}_{q_i}(f^*\chi^m)$, where $\text{ord}_{q_i}$ stands for the multiplicity of $q_i$ in the divisor of $f^*\chi^m$.
\end{itemize}

\begin{rem}
The slope of the unbounded end associated to a given marked point $q_i$ is given both by the order of vanishing of $f^*\chi^m$ at $q_i$, and by the multiplicity of the section defined by $q_i$ in the divisor of $f^*\chi^m$ in the stable model $\mathcal{C}^{(t)}$. This is normal since the marked points provide divisors in $\mathcal{C}^{(t)}$ which are transverse to the special fiber. Concretely, in the rational case, if $y$ is a coordinate on $C$ and $f$ is given by
$$f:y\longmapsto\chi\prod_{i=1}^r(y-y(q_i))^{n_i}\in\mathrm{Hom}(M,\CC((t))^*),$$
with $\chi\in\mathrm{Hom}(M,\CC((t))^*)$, then the slope of the edge associated to the marked point $q_i$ is $n_i$.
\end{rem}

\begin{rem}
The special fiber is given by the equation $t=0$. Therefore $h(w)(m)=\text{ord}_{C_w}(f^*\chi^m)$ is the valuation in $t$ of the function evaluated at the generic point of $C_w$. Concretely, in the rational case, let $y$ be a coordinate on $C$ specializing to a coordinate on $C_w$, which is a copy of $\CC P^1$, such that no point specializes to $\infty$. Assume $f$ is given by
$$f:y\longmapsto\chi\prod_{i=1}^r(y-y(q_i))^{n_i}\in\mathrm{Hom}(M,\CC((t))^*),$$
where $\chi\in\mathrm{Hom}(M,\CC((t))^*)$. Then $h(w)(m)=\val(\chi(m))$.
\end{rem}

The fact that $(\Gamma,h)$ is indeed a parametrized tropical curve is proved in \cite{tyomkin2012tropical}. One essentially needs to check the balancing condition, and the fact that if $\gamma$ is an edge with extremities $v$ and $w$, the slope of $h(\gamma)$ lies in $N$, and the length of $h(w)-h(v)$ coincide with the length $|\gamma|$ in $\Gamma$. Finally, we can refine the tropicalization in the following way that is useful to compute quantum indices: on each $C_w$ the rational map $f$ specializes to give a parametrized complex rational curve $f_w:C_w\simeq\CC P^1\dashrightarrow\mathrm{Hom}(M,\CC^*)$. Concretely, this is the curve we would obtain by taking the naive limit of $f$ in a coordinate specializing to a coordinate of $C_w$. Therefore, we have a tropical curve and a complex curve associated to every vertex.

\begin{rem}
All our curves are taken with coefficients in $\CC((t))$, which is not algebraically closed, and has a discrete valuation. Thus, every tropicalization data has coefficients in $\ZZ$. Instead we could take the algebraic closure, which is the field of Puiseux series $\CC\{\{t\}\}=\bigcup_{k\geqslant 1}\CC((t^\frac{1}{k}))$, but as we are using only a finite number of coefficients, all belong to $\CC((t^\frac{1}{k}))$ for some $k$, and by taking $u=t^\frac{1}{k}$ we reduce it to the previous case. Therefore, we can assume that everything is defined in $\CC((t))$, up to a change of base.
\end{rem}

\subsubsection{Tropicalization of a plane curve}

We finish by describing the tropicalization of a plane curve. This tropicalization is more elementary than the tropicalization of a parametrized curve. Moreover, the tropicalization of a parametrized curve gives a parametrization of the tropicalization of its image plane curve. Let $C$ be a plane curve, defined by a polynomial $P_t\in\CC((t))[M]$ with coefficients in $\CC((t))$. We look for the points of the curve over the Puiseux series, \textit{i.e.} in $N\otimes\CC\{\{t\}\}^*$. In a basis of $M$, the polynomial is given in coordinates by
$$P_t(x,y)=\sum_{(i,j)\in P_\Delta}a_{i,j}(t)x^iy^j.$$
We assume that the coefficients in the corners of $P_\Delta$ are non-zero. Then, we have the associated tropical polynomial
$$\text{Trop}(P_t)(x,y)=\max_{(i,j)\in P_\Delta}\left(\val(a_{i,j}(t))+ix+jy\right),$$
along with a valuation map, also called \textit{tropicalization map}:
$$\text{Val}:\chi\in \mathrm{Hom}(M,\CC\{\{t\}\}^*)\longmapsto\text{val}\circ \chi\in \mathrm{Hom}(M,\RR)=N_\RR.$$
In coordinates, $\text{Val}$ is given by the coordinatewise valuation:
$$\text{Val}:(x,y)\in(\CC\{\{t\}\}^*)^2\longmapsto (\val(x),\val(y))\in\RR^2.$$
The Kapranov theorem \cite{brugalle2014bit} then ensures that the closure of the image of the vanishing locus of $P_t$ in $(\CC\{\{t\}\}^*)^2$ under the valuation map is equal to the tropical curve defined by $\text{Trop}(P_t)$.

\begin{theo}[Kapranov]
Let $C_\text{trop}$ be the tropical curve defined by $\text{Trop}(P_t)$. Then, one has
$$\overline{\text{Val}(C)}=C_\text{trop}.$$
\end{theo}

Let $\alpha_{i,j}=\val(a_{i,j}(t))$, and $a_{i,j}(t)=t^{\alpha_{i,j}}a_{i,j}^0(t)$. The function $(i,j)\mapsto\alpha_{i,j}$ induces a convex subdivision of $P_\Delta$, which is dual to $C_\text{trop}$. As in the tropicalization of a parametrized curve, one can recover complex curves, by specializing the polynomial $P_t$ to one of the polygons of the subdivision. Let $\varpi$ be one of the polygons of the subdivision of $P_\Delta$. Then, the curve associated to $\varpi$ is given by $P_\varpi(x,y)=\sum_{(i,j)\in\varpi}a_{i,j}^0(0)x^iy^j=0$, defined over $\CC$.\\

One can show that if $\varphi:C\rightarrow N\otimes\CC\{\{t\}\}^*$ is a parametrized curve tropicalizing to $h:\Gamma\rightarrow N_\RR$, then the image $h(\Gamma)$ and the tropicalization of the image $\overline{\text{Val}(\varphi(C))}$ are the same. Moreover, the local parametrized curves $f_w:C_w\dashrightarrow N\otimes \CC^*$ resulting from the tropicalization as parametrized curve, are precisely the irreducible components of the curves defined by $P_\varpi=0$. 

\section{Quantum indices of real curves}

	\subsection{The quantum index of a type $I$ real curve}

Let $\varphi:\CC C\rightarrow\CC\Delta$ be a real parametrized curve of type I, which means that $\CC C\backslash\RR C$ is disconnected, and let $S$ be one of its two connected components, which induces an orientation of $\RR C$ as its boundary, called complex orientation. The curve can assumed to be rational rational but this is not needed for the definition of the quantum index.\\

Let $\alpha_1,\dots,\alpha_r\in\RR C$ and $\beta_1,\dots,\beta_s\in S\subset\CC C\backslash\RR C$ be the parameters of the intersection points between the curve and the toric boundary of $\CC\Delta$. Let $n_j:m\mapsto\val_{\alpha_j}(\varphi^*\chi^m)$ and $n'_j:m\mapsto\val_{\beta_j}(\varphi^*\chi^m)$ be the associated weight vectors. In the rational case, the parametrization would then be
$$t\longmapsto \chi\prod_1^r(t-\alpha_i)^{n_i}\prod_1^s (t-\beta_j)^{n'_j}(t-\overline{\beta_j})^{n'_j}.$$
For each intersection points with the boundary, lying on a toric divisor associated to the primitive direction $n\in N$, its coordinate is measured using the monomial $\iota_n\omega\in M$. This coordinate is also called the moment of the point.\\

We have the following logarithmic map and the argument map defined on $N\otimes\CC^*$:
$$\begin{array}{l}
\mathrm{Log}:z^n\in N\otimes\CC^*\longmapsto n\times\log|z|\in N\otimes\RR=N_\RR,\\
\arg : z^n\in N\otimes\CC^*\longmapsto n\times\arg(z)\in N\otimes(\RR/2\pi\ZZ)=N_{2\pi}.\\
\end{array}$$
The argument is here taken mod $2\pi$ but it can also be taken mod $\pi$. We then denote $N_\pi=N\otimes(\RR/\pi\ZZ)$. In either case, each codomain is endowed with a volume form induced by the volume form $\omega$ on $N$. If coordinates on $N\otimes\CC^*$ are $(z_1,z_2)=(e^{x_1+i\theta_1},e^{x_2+i\theta_2})$, where $x_i\in\RR$ and $\theta_j\in\RR/2\pi\ZZ$, then the forms are
$$\omega_{|\bullet|}=\dd x_1\wedge\dd x_2 \text{ and }\omega_\theta=\dd\theta_1\wedge\dd\theta_2.$$
For the form $\omega_\theta$, the volume of $N_{2\pi}$ is $4\pi^2$, and the volume of $N_\pi$ is $\pi^2$. Due to the vanishing of the meromorphic form $\frac{\dd z_1}{z_1}\wedge\frac{\dd z_2}{z_2}$ on $S$, the pullbacks of $\omega_{|\bullet|}$ and $\omega_\theta$ to $S$ coincide, and therefore,
$$\boxed{ \mathcal{A}(S,\varphi)=\int_S\varphi^*\mathrm{Log}^*\omega_{|\bullet|}=\int_S\varphi^*\mathrm{arg}^*\omega_\theta, }$$
is well-defined and called the logarithmic area of $(S,\varphi)$.\\

Recall that the coordinate of each complex intersection point is given by $\varphi^*\chi^{\iota_{n'_j}\omega}|_{\beta_j}$. Let $\varepsilon_j\theta_j=\arg\varphi^*\chi^{\iota_{n'_j}\omega}|_{\beta_j}$, with $\varepsilon_j=\pm 1$ and $\theta_j\in]0;\pi[$, denote their arguments. Although the argument map is not defined at the parameters $\alpha_j$ nor $\beta_j$, it can be extended to the oriented blow-up  of $S$ at $\alpha_j$ and $\beta_j$.
\begin{itemize}
\item Under the coamoeba map to $N_{2\pi}$, each circle of the blow-up corresponding to $\beta_j$ is sent to a geodesic in the direction $n'_j$, and the value of $\iota_{n'j}\omega$ at any point of the geodesic is $\varepsilon_j\theta_j$. This geodesic is traveled a number of times equal to the integral length of $n'_j$.
\item  The half-cricles corresponding to real intersection points are sent to geodesics in the direction associated to the corresponding divisor, but traveled a number of times euqal to the integral length of $n_j$. In particular, if the integral length is odd, there is a half-way travel.
\item If the arguments are rather taken in $N_\pi$, the real intersection points now correspond to full geodesics, and complex intersection points to geodesics traveled twice.
\end{itemize}

Moreover, the complement of the geodesics is endowed with an integer function corresponding to the signed number of antecedents by $(\arg\circ\varphi)|_S$. This function is a $2$-cochain on the argument torus $N_\pi$ (or $N_{2\pi}$), and it changes by when passing on the other side of a geodesic. This last property allows to determine the function up to a shift. The knowing of this function is enough to determine the total signed area. For the coamoeba of the wole curve $\CC C$, as the total area is known to be zero, the shift can be determined. For an half-curve $S$, determining the shift can be more complicated.\\


As noticed by Mikhalkin \cite{mikhalkin2017quantum}, in case the intersection points are purely imaginary, \textit{i.e.} $\theta_j=\frac{\pi}{2}$, the logarithmic area $\mathcal{A}(S)$ is an half-integer multiple of $\pi^2$ for the following reason. Consider the quotient map
$$N_\pi\twoheadrightarrow N_\pi/\{\pm\mathrm{id}\},$$
from the argument torus mod $\pi$ to its quotient by the antipodal map. Since each map $2\arg\circ\varphi|_S:S\rightarrow N_\pi$ defines a chain in the argument torus $N_\pi$, we can push-forward to the quotient and get a chain in the quotient $N_\pi/\{\pm\mathrm{id}\}$, which is homeomorphic to a sphere. However, due to the assumption on the arguments, this chain has no boundary anymore. The boundary previously consisted of the geodesics in $N_\pi$, but as they each pass through a fixed point of the antipodal map, they cancel themselves passing to the quotient. The volume form $\omega_\theta$ also passes to the quotient, and therefore the total area is an integer multiple of the area of $N_\pi/\{\pm\mathrm{id}\}$, which is $\frac{\pi^2}{2}$.\\

In the general case where the $\theta_j$ are not necessarily $\frac{\pi}{2}$, we have the following result.

\begin{prop}
\label{prop existence quantum index}
There exists an half-integer, called the quantum index, such that
$$k\pi^2=\mathcal{A}(S)-\pi\sum_j \varepsilon_j(2\theta_j-\pi).$$
\end{prop}

\begin{proof}
We can shift the geodesics in the following way: for each point $\beta_j$ with corresponding geodesic $B_j$ in $N_{2\pi}$, add the chain whose oriented boundary consists $\overline{B_j}$ (which is $B_j$ with the opposite orientation), and $\widehat{B_j}$ which is the translate of $B_j$ so that the value by $\iota_{n'_j}\omega$ is now $\varepsilon_j\frac{\pi}{2}$. We get a new chain $\hat{S}$ whose area satisfies
$$\left\langle\omega_\theta,\hat{S}\right\rangle = \left\langle\omega_\theta,S\right\rangle + \pi\sum_j \varepsilon_j(2\theta_j-\pi).$$
The chain $\hat{S}$ is now subject to passing to the quotient by the antipodal map, after reducing mod $\pi$, and therefore $\left\langle\omega_\theta,\hat{S}\right\rangle\in\frac{\pi^2}{2}\ZZ$.
\end{proof}

We finish recalling the result from \cite{blomme2020tropical} stating that the log-area is well-behaved under the monomial maps, which are covering maps from the complex torus to itself.

\begin{lem}\cite{blomme2020tropical}\label{lemma monomial behavior}
Let $\varphi : \CC C\dashrightarrow\mathrm{Hom}(M,\CC^*)$ be a type $I$ real curve with a choice of a connected component $S\subset \CC C\backslash\RR C$, inducing a complex orientation, and let $\alpha:\mathrm{Hom}(M,\CC^*)\rightarrow \mathrm{Hom}(M',\CC^*)$ be a monomial map, associated to a morphism $A^T : M'\rightarrow M$. We consider the composition
$$\psi : \CC C \stackrel{\varphi}{\longrightarrow}\mathrm{Hom}(M,\CC^*) \stackrel{\alpha}{\longrightarrow}\mathrm{Hom}(M',\CC^*).$$
Let $\omega$ and $\omega'$ be the volume forms on respectively $N$ and $N'$, dual lattices of $M$ and $M'$, so that we have $A^*\omega'=\det A\omega$. Then, we have
$$\mathcal{A}(S,\psi)=\det A\ \mathcal{A}(S,\varphi).$$
\end{lem}

\begin{rem}
The proposition deals with the computation of log-area in the general case of a real curve. This log-area needs to be shifted in order to get the quantum index.
\end{rem}

\begin{rem}
Notice that the different notations $\mathrm{Hom}(M,\CC^*)$ and $\mathrm{Hom}(M',\CC^*)$ prevent any mistakes in the direction of the various involved maps.
\end{rem}

	\subsection{The quantum index near the tropical limit}
	
	In \cite{mikhalkin2017quantum}, Mikhalkin proved the following result, that computes the log-areas of curves in a family near the tropical limit.
	
	\begin{prop}\cite{mikhalkin2017quantum}
	Let $C^{(t)}=\left(f_t:\CC P^1\rightarrow\mathrm{Hom}(M,\CC^*)\right)$ be a family of type I real parametrized rational curves, enhanced with a family of connected components of the complex locus $S^{(t)}$, inducing complex orientations of the curves. We assume that the family tropicalizes, in the sense of \ref{tropicalization}, to a parametrized tropical curve $h:\Gamma\rightarrow N_\RR$, such that components $S^{(t)}$ specialize to components $S_w$ of $C_w$ for every vertex $w\in\fix$, thus inducing complex orientations of the curves $C_w$. Then, for $t$ large enough,
	$$\mathcal{A}(S^{(t)},f_t)=\sum_w \mathcal{A}(S_w,f_w),$$ 
where the sum is indexed over the fixed vertices of $\Gamma$.
\end{prop}

\begin{rem}
In particular, and this happens in the proof of the correspondence theorem, for one to know the log-area of curves near the tropical limit, one only needs to know the log-areas of the curves associated to the vertices of the tropical curve, and the way they are glued together along the edges. This means that the quantum index may be computed in the patchworking construction. 
\end{rem}


	\subsection{Specific computations}
\label{local computations}

In this section, we compute the log-areas of some auxiliary rational curves. This includes complex conjugated curves, and some curves of degree at most $2$. Using Lemma \ref{lemma monomial behavior}, this enables the computation of the log-area of oriented curves having up to five intersection points with the toric boundary. Some of these computations were already treated in \cite{blomme2020tropical} but we include them here for sake of completeness.

\subsubsection{Log-area of a complex curve}

We begin by proving that the log-area of a complex curve is zero. This justifies the fact that the quantum index near the tropical limit is obtained as a sum over the fixed vertices, since pairs of exchanged vertices do not contribute. The following statement is not specific to rational curves or real curves. 

\begin{lem}
Let $\varphi:\CC C\dashrightarrow N\otimes \CC^*$ be a complex parametrized curve, with $\CC C$ a smooth Riemann surface. Then
$$\mathcal{A}(\CC C,\varphi)=0.$$
\end{lem}

\begin{proof}
Let $\CC C^o$ be the open set of $\CC C$ where $\varphi$ is defined. We consider the map $\mathrm{Log}\circ\varphi:\CC C^o\rightarrow N_\RR$. This is a proper map between smooth oriented manifolds, therefore it has a well-defined degree, which corresponds both to the number of antecedents counted with signs over a generic point, and to the map $\RR =H^2_c(N_\RR)\stackrel{ (\text{Log}\circ\varphi)^*}{\longrightarrow} H^2_c(\CC C^o)=\RR$ between compactly supported cohomology groups. Since the map is not surjective, its degree is zero. Hence, if $\tilde{\omega}$ is a compactly supported $2$-form on $N_\RR$, then $\int_{\CC C^o}(\text{Log}\circ\varphi)^*\tilde{\omega}=0$. Thus, by writing $\omega$ as a (infinite) sum of compactly supported $2$-forms using partitions of unity, we get the result.
\end{proof}

\subsubsection{Log-area of a line}

We recall the computation of the quantum index of a line. This was dealt with in \cite{mikhalkin2017quantum}. A line has degree $\Delta^\mathrm{line}=\{(-1,0),(0,-1),(1,1)\}$.

\begin{lem}
The log-area of the parametrized line $t\mapsto(t,1-t)$ is $\frac{\pi^2}{2}$.
\end{lem}

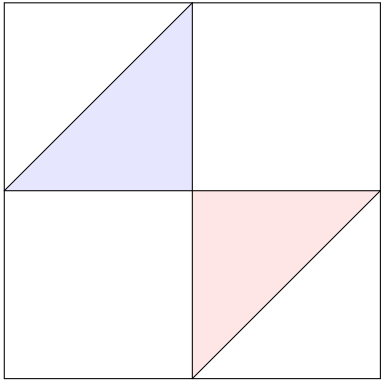
\begin{figure}
\begin{center}
\definecolor{ffqqqq}{rgb}{1.,0.,0.}
\definecolor{qqqqff}{rgb}{0.,0.,1.}
\begin{tikzpicture}[line cap=round,line join=round,>=triangle 45,x=0.5cm,y=0.5cm]
\clip(-1.,-1.) rectangle (11.,11.);
\fill[color=qqqqff,fill=qqqqff,fill opacity=0.1] (0.,5.) -- (5.,5.) -- (5.,10.) -- cycle;
\fill[color=ffqqqq,fill=ffqqqq,fill opacity=0.1] (5.,5.) -- (10.,5.) -- (5.,0.) -- cycle;

\draw (0.,0.)-- (10.,0.);
\draw (10.,0.)-- (10.,10.);
\draw (10.,10.)-- (0.,10.);
\draw (0.,10.)-- (0.,0.);
\draw (0.,5.)-- (10.,5.);
\draw (5.,10.)-- (5.,0.);
\draw (0.,5.)-- (5.,10.);
\draw (5.,0.)-- (10.,5.);
\end{tikzpicture}

\caption{Co-amoeba of a line.}
\label{coamoeba line}
\end{center}
\end{figure}

\begin{proof}
This computation can be done by hand using the logarithmic map, or using the argument map. The coamoeba is depicted on Figure \ref{coamoeba line}, an half-curve corresponds to one of the two triangles. Therefore, the signed area corresponds to the area of one of the triangles.
\end{proof}

\subsubsection{Log-area of a parabola}

We now consider rational curves of degree $\Delta^{\mathrm{par}}=\{(-1,1),(1,1),(0,-1)^2\}$. We assume that the two intersection points with the toric divisor associated to $(0,-1)$ are complex conjugated. Choosing a coordinate on the curve such that these complex points are $\pm i$ and the intersection point with the toric divisor associated to $(1,1)$ is $\infty$. There are two such coordinates which differ by their orientation. Thus, the orientation fixes uniquely the coordinate. Up to a multiplicative translation, the oriented curve has a parametrization
$$\varphi:t\in\CC P^1\longmapsto\left( t-c,\frac{t^2+1}{t-c}\right),$$
where $c\in\RR$ is the coordinate of the last intersection point with the toric boundary. The reversing of the orientation leads to the change of $c$ by $-c$.

\begin{lem}
The log-area of the parabola $\varphi:t\in\CC P^1\longmapsto\left( t-c,\frac{t^2+1}{t-c} \right)$ is $2\pi\arctan c$.
\end{lem}

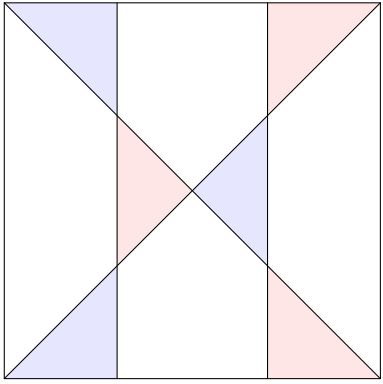
\begin{figure}
\begin{center}
\definecolor{ffqqqq}{rgb}{1.,0.,0.}
\definecolor{qqqqff}{rgb}{0.,0.,1.}
\begin{tikzpicture}[line cap=round,line join=round,>=triangle 45,x=0.5cm,y=0.5cm]
\clip(-1.,-1.) rectangle (11.,11.);
\fill[color=qqqqff,fill=qqqqff,fill opacity=0.1] (0.,10.) -- (3.,10.) -- (3.,7.) -- cycle;
\fill[color=qqqqff,fill=qqqqff,fill opacity=0.1] (0.,0.) -- (3.,3.) -- (3.,0.) -- cycle;
\fill[color=ffqqqq,fill=ffqqqq,fill opacity=0.1] (3.,3.) -- (3.,7.) -- (5.,5.) -- cycle;
\fill[color=qqqqff,fill=qqqqff,fill opacity=0.1] (5.,5.) -- (7.,7.) -- (7.,3.) -- cycle;
\fill[color=ffqqqq,fill=ffqqqq,fill opacity=0.1] (7.,10.) -- (7.,7.) -- (10.,10.) -- cycle;
\fill[color=ffqqqq,fill=ffqqqq,fill opacity=0.1] (7.,3.) -- (7.,0.) -- (10.,0.) -- cycle;
\draw (0.,0.)-- (10.,0.);
\draw (10.,0.)-- (10.,10.);
\draw (10.,10.)-- (0.,10.);
\draw (0.,10.)-- (0.,0.);
\draw (0.,0.)-- (10.,10.);
\draw (0.,10.)-- (10.,0.);
\draw (7.,0.)-- (7.,10.);
\draw (3.,0.)-- (3.,10.);
\end{tikzpicture}

\caption{Co-amoeba of a parabola with order map: $-1$ for red (triangles with left vertical side), $+1$ for blue (triangles with right vertical side).}
\label{coamoeba parabola}
\end{center}
\end{figure}

\begin{proof}
The coamoeba along with its order map is depicted on Figure \ref{coamoeba parabola}. The order map has value $1$ on the blue triangles and $-1$ on the red ones. The abscissa of the two vertical geodesics are both opposite arguments of the complex intersection points with the boundary. The one with parameter $i$ is
$$\arg(i-c) =\mathrm{arcot}(-c)\in]0;\pi[.$$
This is the co-amoeba of the whole curve. However, as $z$ is a coorinate on the curve, we can restrict to the co-amoeba of the half-curve parametrized by $\HH$ if we restrict to the triangles in the right half of the square. Therefore, the log-area is equal to
\begin{align*}
\mathcal{A}(\HH,\varphi) &= \mathrm{arcot}(-c)^2 - (\pi-\mathrm{arcot}(-c))^2 \\
	& = 2\pi \mathrm{arcot}(-c) - \pi^2 \\
	& = 2\pi\arctan c.\\
\end{align*}
\end{proof}

\begin{rem}
Such curves have an equation of the form $w=az+b+\frac{c}{z}$, which becomes $zw=w'=az^2+bz+c$ after a change of toric coordinates. This is why we call them parabolas.
\end{rem}

\subsubsection{Log-area of a tangent ellipse}
\label{subsubsection log-area of a tangent ellipse}

We now consider an ellipse tangent to one of the coordinate axis. An ellipse means that this is a curve of degree $2$ in $\CC P^2$, \textit{i.e.} a conic. We assume that its intersection points with the coordinate axis $\{w=0\}$ and $\{z=0\}$ are complex, and that it has a unique intersection points with the infinite axis. In the standard denomination, a conic tangent to the infinite axis is called a parabola, but we call it an tangent ellipse to distinguish it from the previous parabola case studied.\\

Same as in the parabola base, we choose a suitable coordinate on the oriented curve such that the infinite point has coordinate $\infty$. The other condition to fix this coordinate is that both parameters corresponding to complex intersection points have modulus $1$. This is possible by taking as $0$ the center of the circle containing all four points, and a suitable scaling. Let $\theta,\varphi\in ]0;\pi[$ be the arguments of these parameters. The parametrization is up to a multiplicative translation
$$\psi:t\in\CC P^1\longmapsto\left( t^2-2t\cos\varphi+1,t^2-2t\cos\theta+1\right).$$

Moreover, up to a permutation if the coordinate axis, one can assume that $\theta<\varphi$, and up to a change of the orientation of the curve, which can be done by changing $t$ by $-t$ in the parametrization, we can assume that $\theta+\varphi<\pi$. Indeed, the change of $t$ by $-t$ is the curve where $\theta$ and $\varphi$ have been change by $\pi-\theta$ and $\pi-\varphi$. If $\theta+\varphi>\pi$, then $(\pi-\theta)+(\pi-\varphi)<\pi$. This also reverses their order.

\begin{figure}
\begin{center}
\begin{tabular}{cc} 
\definecolor{ududff}{rgb}{0.30196078431372547,0.30196078431372547,1.}
\definecolor{xdxdff}{rgb}{0.49019607843137253,0.49019607843137253,1.}
\begin{tikzpicture}[line cap=round,line join=round,>=triangle 45,x=0.5cm,y=0.5cm]
\clip(-1.,-1.) rectangle (11.,11.);
\draw [line width=2.pt] (0.,0.)-- (10.,0.);
\draw [line width=2.pt] (10.,0.)-- (10.,10.);
\draw [line width=2.pt] (10.,10.)-- (0.,10.);
\draw [line width=2.pt] (0.,10.)-- (0.,0.);
\draw [line width=1.2pt] (1.,10.)-- (1.,0.);
\draw [line width=1.2pt] (1.,4.904805811625541) -- (0.8810072645319262,5.);
\draw [line width=1.2pt] (1.,4.904805811625541) -- (1.1189927354680747,5.);
\draw [line width=1.2pt] (10.,8.)-- (0.,8.);
\draw [line width=1.2pt] (4.90480581162554,8.) -- (5.,8.118992735468074);
\draw [line width=1.2pt] (4.90480581162554,8.) -- (5.,7.881007264531926);
\draw [line width=1.2pt] (0.,0.)-- (5.,5.);
\draw [line width=1.2pt] (2.6346249122582575,2.6346249122582592) -- (2.651453026290539,2.4831718859677174);
\draw [line width=1.2pt] (2.6346249122582575,2.6346249122582592) -- (2.483171885967715,2.651453026290542);
\draw [line width=1.2pt] (2.5,2.5) -- (2.5168281140322804,2.348546973709457);
\draw [line width=1.2pt] (2.5,2.5) -- (2.3485469737094564,2.516828114032281);
\draw [line width=1.2pt] (5.,5.)-- (10.,10.);
\draw [line width=1.2pt] (7.634624912258259,7.634624912258259) -- (7.6514530262905405,7.4831718859677165);
\draw [line width=1.2pt] (7.634624912258259,7.634624912258259) -- (7.4831718859677165,7.651453026290541);
\draw [line width=1.2pt] (7.5,7.5) -- (7.516828114032282,7.348546973709458);
\draw [line width=1.2pt] (7.5,7.5) -- (7.348546973709458,7.516828114032282);
\begin{scriptsize}
\draw [fill=red] (0.,0.) circle (2.5pt);
\draw [fill=red] (10.,0.) circle (2.5pt);
\draw [fill=red] (10.,10.) circle (2.5pt);
\draw (9.2,8.6) node {$1$};
\draw [fill=red] (0.,10.) circle (2.5pt);
\draw (0.5089140343942222,9.091821214176585) node {$1$};
\draw [fill=red] (5.,10.) circle (2.5pt);
\draw (4.7609211151200705,9.075955516114176) node {$0$};
\draw [fill=red] (10.,5.) circle (2.5pt);
\draw [fill=red] (5.,5.) circle (2.5pt);
\draw [fill=red] (0.,5.) circle (2.5pt);
\draw (0.5247797324566321,4.538365870264948) node {$0$};
\draw (2.999828630192574,5.934547299757018) node {$-1$};
\draw (6.569610694234798,2.967661762086369) node {$0$};
\draw [fill=red] (5.,0.) circle (2.5pt);

\fill[color=blue,fill=blue,fill opacity=0.1] (0.,10.) -- (1.,10.) -- (1.,8.) -- (0.,8.) -- cycle;
\fill[color=blue,fill=blue,fill opacity=0.1] (0.,0.) -- (1.,0.) -- (1.,1.) -- cycle;
\fill[color=blue,fill=blue,fill opacity=0.1] (8.,8.) -- (10.,8.) -- (10.,10.) -- cycle;

\fill[color=red,fill=red,fill opacity=0.1] (1.,1.) -- (8.,8.) -- (1.,8.) -- cycle;

\end{scriptsize}
\end{tikzpicture}
&
\begin{tikzpicture}[line cap=round,line join=round,>=triangle 45,x=0.5cm,y=0.5cm]
\clip(-3.5,-3.5) rectangle (8.5,8.5);
\draw [line width=2.pt] (0.,0.)-- (5.,0.);
\draw [line width=2.pt] (5.,0.)-- (5.,5.);
\draw [line width=2.pt] (5.,5.)-- (0.,5.);
\draw [line width=2.pt] (0.,5.)-- (0.,0.);
\draw [line width=1.2pt] (0.,0.)-- (5.,5.);
\draw [line width=1.2pt] (2.6194190024394985,2.6194190024394994) -- (2.634346377744435,2.485072624695063);
\draw [line width=1.2pt] (2.6194190024394985,2.6194190024394994) -- (2.485072624695062,2.634346377744437);
\draw [line width=1.2pt] (2.5,2.5) -- (2.514927375304937,2.3656536222555635);
\draw [line width=1.2pt] (2.5,2.5) -- (2.3656536222555635,2.514927375304938);
\draw [line width=1.2pt] (1.,5.)-- (1.,0.);
\draw [line width=1.2pt] (1.,2.4155580135724977) -- (0.894447516965622,2.5);
\draw [line width=1.2pt] (1.,2.4155580135724977) -- (1.1055524830343786,2.5);
\draw [line width=2.pt] (0.,0.)-- (5.,0.);
\draw [line width=1.2pt] (5.,3.)-- (0.,3.);
\draw [line width=1.2pt] (2.4155580135724977,3.) -- (2.5,3.1055524830343777);
\draw [line width=1.2pt] (2.4155580135724977,3.) -- (2.5,2.8944475169656214);
\begin{scriptsize}
\draw [fill=red] (0.,0.) circle (2.0pt);
\draw [fill=red] (5.,0.) circle (2.5pt);
\draw [fill=red] (5.,5.) circle (2.5pt);
\draw [fill=red] (0.,5.) circle (2.5pt);
\end{scriptsize}
\end{tikzpicture}
\\
$(a)$ & $(b)$ \\
\end{tabular}
\caption{\label{coamoeba tangent ellipse} Coemoeba of an half ellipse tangent to the infinite axis. On $(a)$ the coamoeba in $N_{2\pi}$, on $(b)$ in $N_\pi$}
\end{center}
\end{figure}
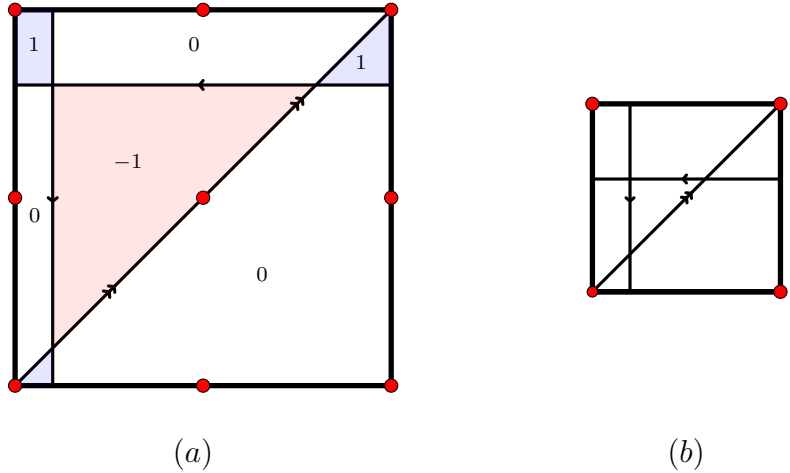

\begin{lem}
The log-area of the oriented curve
$$\psi:t\in\CC P^1\longmapsto\left( t^2-2t\cos\varphi+1,t^2-2t\cos\theta+1\right),$$
is
$$\mathcal{A}(\HH,\psi)=2\pi(\varphi-\theta)=4\pi\arctan\left(\frac{r'-r}{s'+s}\right),$$
where $r+is$ and $r'+is'$ are the coordinates of $e^{i\theta}$ and $e^{i\varphi}$ in any coordinate of the curve keeping the last intersection point with coordinate $\infty$.
\end{lem}

\begin{proof}
We use the coamoeba of the half-curve to determine the logarithmic area. The coamoeba is depicted on Figure \ref{coamoeba tangent ellipse}. The fact that it is a chain allows us to determine the degree map up to a shift by  some $k\in\ZZ$. The degree is then as depicted on \ref{coamoeba tangent ellipse} $(a)$ up to $k$. We then have to show that $k=0$.\\

First, let compute the position of the geodesics.
\begin{itemize}[label=-]
\item The geodesic with slope $1$ is the geodesic corresponding to the tangent point with the infinite axis, and passes through $(0,0)$ since the curve is located in the quadrant $\RR_{>0}^2$.
\item The abscissa of the vertical geodesic is the argument of the intersection point with the axis $\{w=0\}$:
$$e^{2i\theta}-2e^{i\theta}\cos\varphi + 1  = 2e^{i\theta}(\cos\theta-\cos\varphi).$$
As $\cos\theta-\cos\varphi>0$, the argument is $\theta$.
\item The height of the horizontal geodesic is the argument of the intersection point with the coordinate axis $\{z=0\}$, which has itself coordinate
$$e^{2i\varphi}-2e^{i\varphi}\cos\theta + 1  = 2e^{i\varphi}(\cos\varphi-\cos\theta).$$
This time, the argument is $\varphi+\pi$.
\end{itemize}

For each generic point in $N_{2\pi}$, the absolute value of the degree is a lower bound to the number of antecedents by the coamoeba map. For a point in $N_\pi$, we can also use the degree, but we obtain a finer lower bound by adding the absolute value of the degrees for each of its four antecedents by the covering map $N_{2\pi}\rightarrow N_\pi$. We now provide an upper bound to the number of antecedents in $N_\pi$ to get a bound on $k$. Let $(\alpha,\beta)$ be a generic point in $N_\pi$. The antecedents of $(\alpha,\beta)$ by the coamoeba map are the points of the half-curve lying in $e^{i\alpha}\RR^*\times e^{i\beta}\RR^*$. This set is the fixed point set of the twisted conjugation
$$\mathrm{conj}_{\alpha,\beta}(z,w)=(e^{2i\alpha}\overline{z},e^{2i\beta}\overline{w}),$$
therefore, they are a subset of the fixed points of the whole curve $C$ by this conjugation. The fixed point coincide with the intersection $C\cap\mathrm{conj}_{\alpha,\beta}(C)$. These are two conics, thus they have at most $4$ intersection points. Therefore, the number of antecedents is at most $4$.\\

We get the following set of inequality by looking at each of the three regions of the complement of the geodesics in $N_\pi$.
$$\left\{ \begin{array}{l}
3|k-1|+|k|\leqslant 4 \text{ for the central triangle}, \\
3|k|+|k-1|\leqslant 4 \text{for the hexagone}, \\
3|k|+|k+1| \leqslant 4 \text{for the splitted triangle}.\\
\end{array}\right.$$
The two first inequalities imply $k=0$ or $1$. The last inequality forces $k=0$. In fact, the degree in $N_\pi$ is bounded by $4$, and any change in the degree in $N_{2\pi}$ changes the degree in $N_\pi$ by $4$, that is why there is only a few allowed values.\\

We now only have to compute the signed area, which is equal to the difference between the areas of the triangles. The red triangle has side length $\pi+\varphi-\theta$ and sign $-1$, while the blue one has side length $\theta+\pi-\varphi$, and sign $+1$
\begin{align*}
\mathcal{A} & = \frac{1}{2}(\pi+\varphi-\theta)^2-\frac{1}{2}(\pi+\theta-\varphi)^2 \\
	& = 2\pi(\varphi-\theta).\\
\end{align*}

Without the assumption $\theta<\varphi$ and $\theta+\varphi<\pi$, the result is still valid:
\begin{itemize}[label=-]
\item If $\theta>\varphi$ and $\varphi+\theta<\pi$, the change of the axis reduces to the first case by changing the role of $\theta$ and $\varphi$, but also changes the sign of the volume form, so the formula gives $-2\pi(\theta-\varphi)=2\pi(\varphi-\theta)$, and the same formula holds.
\item If $\theta>\varphi$ and $\theta+\varphi>\pi$, replacing $(\theta,\varphi)$ by $(\pi-\theta,\pi-\varphi)$ reduces to the first case, but the orientation of the curve has to be changed. Therefore, the log-area becomes $-2\pi(\pi-\varphi-(\pi-\theta))=2\pi(\varphi-\theta)$.
\item The last case is a combination of both previous cases. If $\theta+\varphi>\pi$ and $\theta<\varphi$, the replacement of $(\theta,\varphi)$ by $(\pi-\varphi,\pi-\theta)$ reduces to the first case, and the formula gives $2\pi(\pi-\theta-(\pi-\varphi))=2\pi(\varphi-\theta)$. The formula still holds. There is no additional sign since the curve changes its orientation when changing to $\pi-\bullet$, and also the volume form when changing the axis.
\end{itemize}

We have proven that the formula is valid for any $(\theta,\varphi)$; Last, we relate the value of $\theta$ and $\varphi$ to the affixes of the parameters under any coordinate of the curve sending the tangent point to $\infty$. The calculation emphasizes the drawing on Figure \ref{tangent ellipse change of coordinate}. The goal is to exprime the angle $\varphi-\theta$ in terms of $r,r',s$ and $s'$. The drawing is obtained as follows:
\begin{itemize}[label=-]
\item We start with the complex points $B:r+is$ and $C:r'+is'$ in $\CC$, with $s,s'>0$ and assuming $r<r'$.
\item The points $A$ and $B$ are their projections on the real axis.
\item The point $O$ correspond to the $0$ of the coordinate for which $B$ and $C$ are $e^{i\varphi}$ and $e^{i\theta}$. It lies on the mediator $(IO)$. Notice that $O$ does not necessarily lie on the segment $[AD]$.
\end{itemize}

Then, we proceed as follows:
\begin{itemize}[label=-]
\item Triangles $OBA$ and $OBI$ are right triangles, therefore $A,B,I$ and $O$ are cocyclic. Thus, $\widehat{OBI}=\widehat{OAI}$.
\item The triangles $OBI$ and $DAA'$ are right triangle whose angle have the same value. Thus, they are similar. Hence, $\frac{\varphi-\theta}{2}=\widehat{IOB}=\widehat{AA'D}$.
\item Finally, $\tan\widehat{AA'D}=\frac{r'-r}{s+s'}$.
\end{itemize}

We have proven that $\mathcal{A}=2\pi(\varphi-\theta)=4\pi\arctan\left(\frac{r'-r}{s+s'}\right)$.

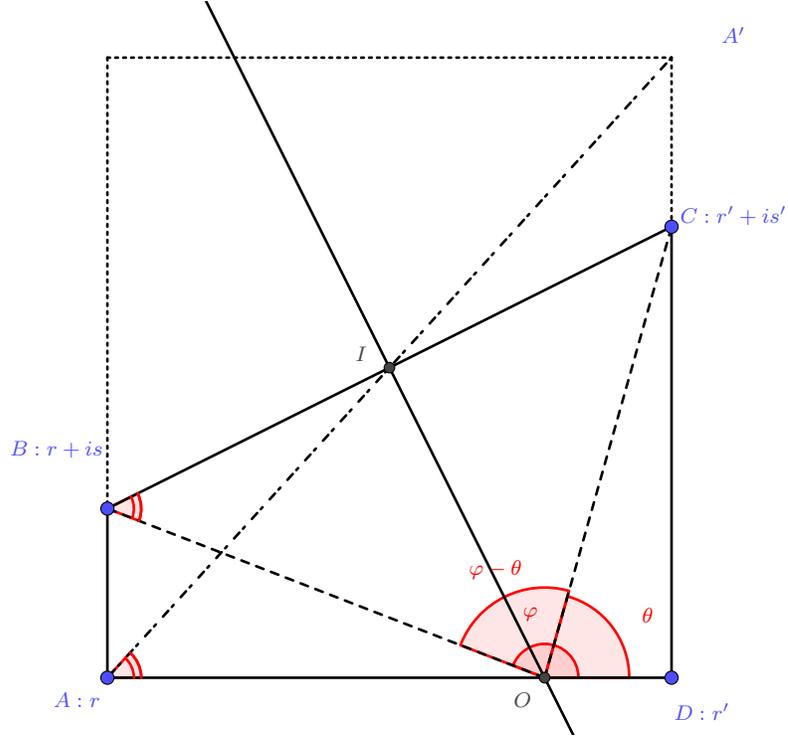
\begin{figure}
\begin{center}
\definecolor{qqwuqq}{rgb}{0.,0.39215686274509803,0.}
\definecolor{uuuuuu}{rgb}{0.26666666666666666,0.26666666666666666,0.26666666666666666}
\definecolor{ududff}{rgb}{0.30196078431372547,0.30196078431372547,1.}
\begin{tikzpicture}[line cap=round,line join=round,>=triangle 45,x=0.75cm,y=0.75cm]
\clip(-1.0,0.) rectangle (13.,13.);
\draw [shift={(8.75,1.)},line width=1.pt,color=red,fill=red,fill opacity=0.10000000149011612] (0,0) -- (0.:1.5) arc (0.:74.29136217098426:1.5) -- cycle;
\draw [shift={(8.75,1.)},line width=1.pt,color=red,fill=red,fill opacity=0.10000000149011612] (0,0) -- (0.:0.6) arc (0.:158.83874018317172:0.6) -- cycle;
\draw [shift={(8.75,1.)},line width=1.pt,color=red,fill=red,fill opacity=0.10000000149011612] (0,0) -- (74.29136217098426:1.6) arc (74.29136217098426:158.83874018317172:1.6) -- cycle;
\draw [shift={(1.,4.)},line width=1.pt,color=red,fill=red,fill opacity=0.10000000149011612] (0,0) -- (-21.161259816828277:0.6) arc (-21.161259816828277:26.56505117707799:0.6) -- cycle;
\draw [shift={(1.,1.)},line width=1.pt,color=red,fill=red,fill opacity=0.10000000149011612] (0,0) -- (0.:0.6) arc (0.:47.72631099390627:0.6) -- cycle;
\draw [line width=1.pt] (1.,1.)-- (1.,4.);
\draw [line width=1.pt] (1.,4.)-- (11.,9.);
\draw [line width=1.pt] (11.,9.)-- (11.,1.);
\draw [line width=1.pt] (11.,1.)-- (1.,1.);
\draw [line width=1.pt,domain=0.:13.] plot(\x,{(-92.5--10.*\x)/-5.});
\draw [line width=1.pt,dash pattern=on 3pt off 3pt] (8.75,1.)-- (1.,4.);
\draw [line width=1.pt,dash pattern=on 3pt off 3pt] (8.75,1.)-- (11.,9.);
\draw [line width=1.pt,dotted] (1.,12.)-- (1.,4.);
\draw [line width=1.pt,dotted] (1.,12.)-- (11.,12.);
\draw [line width=1.pt,dotted] (11.,12.)-- (11.,9.);
\draw [line width=1.pt,dash pattern=on 1pt off 2pt on 3pt off 4pt] (11.,12.)-- (1.,1.);
\draw [shift={(1.,4.)},line width=1.pt,color=red] (-21.161259816828277:0.6) arc (-21.161259816828277:26.56505117707799:0.6);
\draw [shift={(1.,4.)},line width=1.pt,color=red] (-21.161259816828277:0.47) arc (-21.161259816828277:26.56505117707799:0.47);
\draw [shift={(1.,1.)},line width=1.pt,color=red] (0.:0.6) arc (0.:47.72631099390627:0.6);
\draw [shift={(1.,1.)},line width=1.pt,color=red] (0.:0.47) arc (0.:47.72631099390627:0.47);
\begin{scriptsize}
\draw [fill=ududff] (1.,1.) circle (2.5pt);
\draw[color=ududff] (0.46,0.61) node {$A:r$};
\draw [fill=ududff] (1.,4.) circle (2.5pt);
\draw[color=ududff] (0.1,5.05) node {$B:r+is$};
\draw [fill=ududff] (11.,9.) circle (2.5pt);
\draw[color=ududff] (12.1,9.2) node {$C:r'+is'$};
\draw[color=ududff] (12.1,12.4) node {$A'$};
\draw [fill=ududff] (11.,1.) circle (2.5pt);
\draw[color=ududff] (11.54,0.4) node {$D:r'$};
\draw [fill=uuuuuu] (8.75,1.) circle (2.0pt);
\draw[color=uuuuuu] (8.36,0.61) node {$O$};
\draw [fill=uuuuuu] (6.,6.5) circle (2.0pt);
\draw[color=uuuuuu] (5.5,6.75) node {$I$};
\draw[color=red] (10.58,2.11) node {$\theta$};
\draw[color=red] (8.5,2.13) node {$\varphi$};
\draw[color=red] (7.88,2.91) node {$\varphi-\theta$};
\end{scriptsize}
\end{tikzpicture}
\caption{\label{tangent ellipse change of coordinate} Depiction of the change of coordinate in the tangent ellipse case}
\end{center}
\end{figure}

\end{proof}

\begin{rem}
The computation is also technically doable by the logarithmic side. We have to integrate over the Poincar\'e half-plane $\RR\times ]0;\infty[$ the jacobian of the logarithmic map. the jacobian is a rational function in the two coordinates of $\HH$ whose denominator is already factorized. The integral over $\RR$ can be dealt with using the residue formula, and the remaining integral over $]0;\infty[$ dealt with using standard methods. However, the computation passes through several decompositions and can be considered painful.
\end{rem}

\subsection{Log-area of a real rational curve}

Using Lemma \ref{lemma monomial behavior} and the previous computations, it is possible to compute the logarithmic area of any real oriented curve which has
\begin{itemize}[label=-]
\item three real punctures,
\item two real punctures and two complex ones,
\item one real and two pairs of complex ones.
\end{itemize}

This is already enough to compute the logarithmic area of a family of curves close to the tropical limit  provided that the tropical limit curve is generic. We now show that this also allows us to compute the logarithmic area of any oriented real rational curve provided that there is at least one real boundary point and we are given a parametrization of the oriented rational curve and the parameters of the intersection points with the toric boundary, \textit{i.e.} the factorization of both rational functions on the rational curve is known.\\

Let us be given a parametrized oriented real rational curve
$$\varphi:t\in\CC P^1\longmapsto \chi\prod_1^r (t-\alpha_i)^{n_i}\prod_1^s (t^2-2t\re\beta_j +|\beta_j|^2)^{n'_j}\in N\otimes\CC^*.$$

\begin{theo}
\label{log-area rational curve}
If $r\geqslant 1$, assume that $\alpha_r=\infty$, the logarithmic area is given by
$$\begin{array}{rcl}
\mathcal{A}(\HH,\varphi)= & & \sum_{i<i'}\omega(n_i,n_{i'})\frac{\pi^2}{2}\mathrm{sgn}(\alpha_{i'}-\alpha_i) \\
& + & \sum_{j<j'}\omega(n'_j,n'_{j'})2\pi\arctan\left( \frac{\alpha_i-\Re\beta_j}{\Im\beta_j} \right) \\
	& + & \sum_{i,j}\omega(n_i,n'_j)4\pi\arctan\left( \frac{\Re\beta_{j'}-\Re\beta_j}{\Im\beta_{j'}+\Im\beta_j} \right).\\
\end{array}$$
\end{theo}

\begin{proof}
We use a higher dimensional setting of Lemma \ref{lemma monomial behavior}. First, assume that $r\geqslant 1$. If we denote by $(e_1,\dots,e_{r-1},f_1,\dots,f_s)$ the canonical base of $\ZZ^{r+s-1}$, there is unique linear map $A:\ZZ^{r+s-1}\rightarrow N$ such that
$$\left\{ \begin{array}{l}
A(e_i)=n_i, \\
A(f_j)=f_j.\\
\end{array}\right. $$
If we denote by $\alpha:(\CC^*)^{r+s-1}\rightarrow N\otimes\CC^*$ the associated monomial map, the parametrized curve factors as $\varphi=\alpha\circ\psi$, where $\psi$ is
$$\psi:t\in\CC P^1\longmapsto \chi'\prod_1^r (t-\alpha_i)^{e_i}\prod_1^s (t^2-2t\re\beta_j +|\beta_j|^2)^{f_j}\in(\CC^*)^{r+s-1},$$
where $\chi'$ is chosen such that $\alpha(\chi')=\chi$. Using the fact that $\mathrm{Log}\circ\alpha=A\circ\mathrm{Log}$, we relate $\mathcal{A}(\HH,\varphi)$ to the log-area of the pullback of the volume form $\omega$ by $A$:
$$\mathcal{A}(\HH,\varphi)=\int_\HH \psi^*\mathrm{Log}^*A^*\omega.$$
The $2$-form $A^*\omega\in(\Lambda^2\ZZ^{r+s-1})^*$ decomposes as follows:
$$A^*\omega = \sum_{i<i'}\omega(n_i,n_{i'})e_i^*\wedge e^*_{i'} + \sum_{j<j'}\omega(n'_j,n'_{j'})f_j^*\wedge f^*_{j'} + \sum_{i,j}\omega(n_i,n'_j)e_i^*\wedge f^*_j.$$
Therefore, we only need to compute the logarithmic area of $(\HH,\psi)$ for the canonical forms $e_i^*\wedge e_{i'}^*$, $f_j^*\wedge f_{j'}^*$ and $e_i^*\wedge f_j^*$. However, we notice that these forms are the pullback of the canonical volume form on $\ZZ^2$ by the projections
$$\left\{ \begin{array}{l}
p_{ii'}:x\in\ZZ^{r+s-1}\longmapsto (x_i,x_{i'}) , \\
p_{jj'}:x\in\ZZ^{r+s-1}\longmapsto (x'_j,x'_{j'}) , \\
p_{ij}:x\in\ZZ^{r+s-1}\longmapsto (x_i,x'_j). \\
\end{array}\right.$$
Using the same trick, the computation of the log-area for these $2$-forms are equal to the log-area of some plane curves which have previously been studied. Therefore, we are left with computing the log-area of a plane curve whose degree is one of the degrees for which the log-area has been computed: it is either a line, a parabola, or a tangent ellipse. In each case, we perform a change of variable in order to get back to the canonical choice of coordinate used in the computation of those specific cases. In order to simplify the formula, we have assumed that $\alpha_r=\infty$. Therefore, the log-area is finally
$$\begin{array}{rcl}
\mathcal{A}(\HH,\varphi)= & & \sum_{i<i'}\omega(n_i,n_{i'})\frac{\pi^2}{2}\varepsilon_{ii'} \\
& + & \sum_{j<j'}\omega(n'_j,n'_{j'})2\pi\arctan\left( \frac{\alpha_i-\Re\beta_j}{\Im\beta_j} \right) \\
	& + & \sum_{i,j}\omega(n_i,n'_j)4\pi\arctan\left( \frac{\Re\beta_{j'}-\Re\beta_j}{\Im\beta_{j'}+\Im\beta_j} \right) ,\\
\end{array}$$
where $\varepsilon_{ii'}$ is $\pm 1$ according to whether $\alpha_i$, $\alpha_{i'}$ and $\alpha_r$ are in the cyclic order induced by the orientation of the curve, \textit{i.e.} $\alpha_i<\alpha_{i'}$.\\
\end{proof}

\begin{rem}
If all the intersection points are real, we recover Mikhalkin's result from \cite{mikhalkin2017quantum} for toric type I curves, stating thet the quantum index, here equal to the log-area up to a $\pi^2$ factor, only depends on the cyclic order in which the oriented curve meets the divisors, here embodied in the collection of signs $\mathrm{sgn}(\alpha_{i'}-\alpha_i)$.
\end{rem}

\section{Tropical enumerative problem and refined curve counting}

	\label{tropical enumerative inv}

Let $\Delta\subset N$ be a family of $m$ primitive lattice vectors, with total sum $0$. As described in the introduction, there is an associated lattice polygon $P_\Delta$ having $m$ lattice points on its boundary. The toric surface obtained from $\Delta$ is denoted by $\CC\Delta$. Let $E_1,\dots, E_p$ denote the sides of the polygon $P_\Delta$ and let $n_1,\dots,n_p\in N$ be their normal primitive vectors. Let $s=(s_1,\dots,s_p)$ where $s_i\leqslant\frac{l(E_i)}{2}$ for each $i$ be tuple of positive integers. Let $r_i=l(E_i)-2s_i$, so that we have $\sum_1^p r_i +2s_i=m$. Let
$$\Delta(s)=\{n_1^{r_1},(2n_1)^{s_1},\dots,n_p^{r_p},(2n_p)^{s_p} \}.$$
The size of $s$ is denoted by $|s|$:
$$|s|=\sum_1^p s_i.$$

\subsection{Tropical Problem}

The tropical curves of degree $\Delta(s)$ have $m-|s|$ unbounded ends, and therefore the moduli space $\mathcal{M}_0(\Delta(s), N_\RR)$ of parametrized rational tropical curves of degree $\Delta(s)$ in $N_\RR\simeq\RR^2$ has dimension $m-|s|-1$. We have the evaluation map:
$$\text{ev} : \mathcal{M}_0(\Delta(s), N_\RR)\longrightarrow \RR^{m-|s|-1},$$
that associates to a parametrized tropical curve the moments of every unbounded end but the first. Recall that the moment of the first unbounded end is equal to minus the sum of the other moments because of the tropical Menelaus theorem. Let $\mu\in\RR^{m-|s|-1}$ be a generic family of moments. We look for parametrized rational tropical curves $\Gamma$ of degree $\Delta(s)$ such that $\text{ev}(\Gamma)=\mu$.\\

Due to genericity, as noticed in \cite{blomme2019caporaso}, every parametrized rational tropical curve $\Gamma$ such that $\text{ev}(\Gamma)=\mu$ is a simple nodal tropical curve and thus has a well-defined refined multiplicity. We then set
$$N_{\Delta(s)}^{\partial,\text{trop}}(\mu)=\sum_{\text{ev}(\Gamma)=\mu}m_\Gamma^q\in\ZZ[q^{\pm\frac{1}{2}}].$$

\begin{theo}\cite{blomme2019caporaso}
\label{tropinvariance}
The value of $N_{\Delta(s)}^{\partial,\text{trop}}(\mu)$ is independent of $\mu$ provided that it is generic.
\end{theo}

\begin{rem}
The unbounded edges of weight $2$ are destined to be split, and thus correspond to transverse intersection points between the curve and the toric boundary. Higher order unbounded ends would correspond to tangencies with the boundary for which we do not yet have an associated classical invariant to relate to.
\end{rem}

\begin{rem}
Theorem \ref{tropinvariance} can be seen as a particular case of the theorem about refined broccoli invariants proven by L.~G\"ottsche and F.~Schroeter in \cite{gottsche2019refined}. However, the theorem from \cite{gottsche2019refined} about refined broccoli invariants has more general assumptions, while the context of Theorem \ref{tropinvariance} is more specific. That is why the proof provided in \cite{blomme2019caporaso} is easier.
\end{rem}

\subsection{Classical problem}
\label{classical problem}

Keeping previous notations, let $\mathcal{P}$ be a configuration of $m$ points on the toric boundary $\partial\CC\Delta$ such that:
\begin{itemize}[label=-]
\item each toric divisor associated to a side $E_i$ of $P_\Delta$ contains exactly $r_i$ real points and $s_i$ pairs of complex conjugated points,
\item the configuration satisfies the Menelaus condition.
\end{itemize} 

Let $\mathcal{S}(\mathcal{P})$ be the set of oriented real rational curves such that for every $p\in\mathcal{P}$, the curve passes through $p$ or $-p$. Such a curve is said to \textit{pass through the symmetric configuration} $\mathcal{P}$. As the curves are oriented, each real curve is counted twice: once with each of its orientations. Notice that if a curve passes through one of the points of a pair of non-real points, it also passes through its conjugate since the curve is real. We denote by $\mathcal{S}_k(\mathcal{P})$ the subset of $\mathcal{S}(\mathcal{P})$ formed by oriented curves with quantum index $k$.\\

Let $\varphi:\CC P^1\rightarrow\CC\Delta$ be an oriented real parametrized rational curve, denoted by $(S,\varphi)$. The logarithmic Gauss map sends a point $p\in\RR P^1$ to the tangent direction to $\text{Log}\varphi(\RR P^1)$ inside $N_\RR$. We get a map
$$\gamma:\RR P^1\rightarrow \PP^1(N_\RR).$$
The first space $\RR P^1$ is oriented since the curve is oriented, while $\PP^1(N_\RR)$ is oriented by $\omega$. The degree of this map is denoted by $\text{Rot}_\text{Log}(S,\varphi)\in\ZZ$. If the curve has transverse intersections with the divisors, it has the same parity as the number of boundary points $m$. We then set
$$\sigma(S,\varphi)=(-1)^\frac{m-\text{Rot}_\text{Log}(S,\varphi)}{2}\in\{\pm 1\}.$$
Now let
$$R_{\Delta,k}(\mathcal{P)}=\sum_{(S,\varphi)\in\mathcal{S}_k(\mathcal{P})}\sigma(S,\varphi),$$
and
$$R_{\Delta}(\mathcal{P)}=\frac{1}{4}\sum_{k}R_{\Delta,k}(\mathcal{P)}q^k\in\ZZ [q^{\pm\frac{1}{2}}].$$
The coefficient $\frac{1}{4}$ is here to account for the deck transformation: if $\{f(x,y)=0\}$ is a curve in $\mathcal{S}(\mathcal{P})$, then $\{f(x,-y)=0\}$, $\{f(-x,y)=0\}$, $\{f(-x,-y)=0\}$ are in $\mathcal{S}(\mathcal{P})$ too.

\begin{rem}
The shift by $m$ to the logarithmic rotation number is only to keep track of its residue mod $2$, while we are interested in its residue mod $4$, one could also choose another convention. For instance the logarithmic rotation number of a maximal curve.
\end{rem}

\begin{theo}(Mikhalkin \cite{mikhalkin2017quantum})
As long as $r=\sum_1^p r_i\geqslant 1$, the value of $R_{\Delta}(\mathcal{P)}$ is independent of the configuration $\mathcal{P}$ as long as it is generic. It only depends on $\Delta$ and $s$.
\end{theo}

The obtained polynomial, independent of $\mathcal{P}$, is denoted by $R_{\Delta,s}$.

\begin{rem}
Although this theorem is not stated in these terms in \cite{mikhalkin2017quantum} because the only case considered is the one of purely imaginary points where the quantum index coincides with the log-area, the proof does not use this specific assumption, and thus applies also in this setting.
\end{rem}

\subsection{Statement of the result}

The refined tropical count and the refined classical count find their relation through the following theorem, proven in section \ref{section statement of result and proof}. In what follows, we use the positive and negative $q$-analogs: for $a\in\RR$, let
	$$\plus{a}=q^{a}+q^{-a} \text{ and }\minus{a}=q^a-q^{-a}.$$

\begin{theo}
\label{theorem paper}
One has
$$R_{\Delta,s}=2^{|s|}\frac{\minus{\frac{1}{2}}^{m-2-|s|}}{\minus{1}^{|s|}}N^{\partial,\text{trop}}_{\Delta(s)} =2^{|s|}\frac{\minus{\frac{1}{2}}^{m-2-2|s|}}{\plus{\frac{1}{2}}^{|s|}}N^{\partial,\text{trop}}_{\Delta(s)}  $$
\end{theo}

In \cite{gottsche2019refined} L. G\"ottsche and F. Schroeter proposed a refined way to count so-called \textit{refined broccoli curves} having fixed ends, and passing through a fixed configuration of "real and complex" points. In the case where there are only marked ends and no marked points, this count coincides with the count of plane tropical curves passing through the configuration with usual Block-G\"ottsche multiplicities from \cite{gottsche2014refined} up to a multiplication by a constant term depending on the degree and easily computed. More precisely, provided there are no marked points, the refined broccoli multiplicity is just the refined multiplicity from \cite{gottsche2014refined} enhanced by a product over the ends of weight higher than $2$, coinciding with this aforementioned constant term. In our case, since the only multiple edges are marked and of weight $2$ (only one real unbounded end is unmarked and of weight one), this factor is $\frac{q+q^{-1}}{q^{1/2}+q^{-1/2}}$ for each of the $|s|$ ends. If we denote by $BG_{\Delta(s)}(q)$ the refined invariant obtained in \cite{gottsche2019refined}, then we have the relation
$$R_{\Delta,s}(q) = \frac{\minus{\frac{1}{2}}^{m-2-2|s|}}{\plus{1}^{|s|}} BG_{\Delta(s)}(q).$$

\section{Realization and correspondence theorem in the real case}


The proof of Theorem \ref{theorem paper} uses a correspondence theorem between real classical curves and tropical curves. Such a theorem was already proven in \cite{blomme2020tropical} by refining to the case of real curves the realization theorem of Tyomkin \cite{tyomkin2017enumeration}. The proof of this theorem followed the same steps as in \cite{tyomkin2017enumeration} and presented similar calculations. Its proof proceeds as follows: find a suitable set of coordinates on the space of family of curves tropicalizing to a given tropical curve, and then prove that the jacobian of the evaluation map is invertible, so that the Hensel's lemma can be applied. The first subsections are dedicated to recall the notations from \cite{blomme2020tropical} which are specific to the correspondence theorem. The last subsection proves that the correspondence theorem can be applied in our setting showing that the jacobian is indeed invertible. \\

	\subsection{Notations}

Let $(\Gamma,\sigma)$ be a real rational abstract  tropical curve with set of ends $I$ of size $m$. The quotient map is denoted by $\pi:\Gamma\rightarrow\Gamma/\sigma$. The action of $\sigma$ on $I$ induces the following decomposition:
$$I=\{ x_1,\dots,x_r,z_1^\pm,\dots,z_s^\pm\},$$
where $\sigma(x_i)=x_i$ and $\sigma(z_i^\pm)=z_i^\mp$. If $\gamma\in\Gamma^1$ is a bounded edge of $\gamma$ (same for $\Gamma/\sigma$), let $\tg$ and $\hg$ be the tail and the head of $\gamma$. Notice that $\hg\notin\fix$ if and only if $\gamma\notin\fix$. Assume that $r\geqslant 1$, and orient the edges of both $\Gamma$ and $\Gamma/\sigma$ away from $x_r$, which makes them rooted trees, thus inducing a partial order $\prec$ on the curve $\Gamma$. We also endow the set $I/\sigma$ with a total order.\\
\begin{itemize}[label=-]
\item If $w\in\Gamma^0$ is a vertex of $\Gamma$, let $I_w^\infty$ be the set of ends of $\Gamma$ which are greater than $w$ for the order $\prec$. We take a similar notation $(I/\sigma)^\infty_w$ for $w\in(\Gamma/\sigma)^0$. Notice that if $w\in\fix$, then $I^\infty_w$ is stable by $\sigma$, and if $w\notin\fix$, then at most one element of each pair $\{z_j^\pm\}$ belongs to $I^\infty_w$.\\
\item If $\gs\in(\Gamma/\sigma)^1$, let $\iota(\gamma^\sigma)=\min(I/\sigma)^\infty_\hgs$.
\item The order on $I/\sigma$ along with this map $\iota$ induces an order on the edges of $\Gamma/\sigma$ having the same tail. We thus can speak about the smallest and biggest edge leaving a vertex $\pi(w)$ of $\Gamma/\sigma$.
	\item If $w\in\fix$, then we have three cases for the lift of an edge $\gs$ such that $\tgs=\pi(w)$:
		\begin{itemize}[label=$\star$]
		\item $(RR)$ the edge $\gs=\{\gamma\}$ lifts to a fixed edge $\gamma$ of $\Gamma$, and $\iota(\gs)=\{x_j\}$ is a real marking. We then set $\iota(\gamma)=x_j$.
		\item $(RC)$ the edge $\gs=\{\gamma\}$ lifts to a fixed edge $\gamma$ of $\Gamma$ but $\iota(\gs)=\{z_j^\pm\}$ is a complex marking. It means that the curve $\Gamma$ splits at some point on the path from $w$ to $\iota(\gs)$, but not right away in $\gamma$.
		\item $(CC)$ the edge $\gs=\{\gamma^+,\gamma^-\}$ lifts to a pair of exchanged edges in $\Gamma$. If $\iota(\gs)=\{z_j^\pm\}$, with $\tg\prec z_j^+$.
		\end{itemize}
	\item If $w\notin\fix$, for every edge $\gamma$ such that $\tg=w$ we have a unique complex end $\iota(\gamma)\in\iota(\gs)=\{z_j^\pm\}$ accessible by $w$.
	\end{itemize}
	
	We denote by $v_\RR$ (resp. $v_\CC$) the number of fixed vertices (resp. pairs of exchanged vertices), and by $e_\RR$ (resp. $e_\CC$) the number of fixed bounded edges (resp. pairs of exchanged bounded edges).\\

	\subsection{Space of rational curves with given tropicalization} 

Let $\Gamma$ be an abstract tropical curve with $m$ ends. Let $(C^{(t)},x_1,\dots,x_r,z_1^\pm,\dots,z_s^\pm)$ be a stable family of real smooth rational curve with a real configuration of marked points, tropicalizing on $\Gamma$, the dual graph of $C^{(0)}$.\\

\begin{rem}
By taking a coordinate, the marked curve $(C^{(t)},\textbf{x},\textbf{z}^\pm)$ can be seen as $\PP^1\left(\CC((t))\right)\simeq\CC((t))\cup\{\infty\}$, the projective line over the field of Laurent series, along with $r+2s$ Laurent series which are the marked points, taken up to a change of coordinate in $GL_2\big(\RR((t))\big)$. The first $r$ Laurent series are in $\RR((t))$, and the last $s$ are taken in $\CC((t))\backslash\RR((t))$ along with their conjugate.
\end{rem}

We associate to each vertex $w\in\Gamma^0$ a coordinate $y_w$ on $C^{(t)}$,such that $y_{\sigma(w)}=\overline{y_w\circ\sigma}$, and $y_w$ specializes to a coordinate on the irreducible component of $C^{(0)}$ associated to $w$.

\begin{itemize}[label=-]
\item If $w\in\fix$, let $\gamma^\sigma_a$ and $\gamma^\sigma_b$ be the smallest and biggest edges emanating from $\pi(w)$ in $\Gamma/\sigma$. We make a disjunction according to the type $(RR),(RC),(CC)$ of each edge:
	\begin{itemize}[label=$\star$]
	\item[$(RR/RR)$] Edges $\gamma^\sigma_a$ and $\gamma^\sigma_b$ lift to edges $\gamma_a$ and $\gamma_b$ of $\Gamma$, which have well-defined real marking $x_a=\iota(\gamma_a)$ and $x_b=\iota(\gamma_b)$. Then we take $y_w$ such that $y_w(x_r)=\infty$, $y_w(x_a)=0$, $y_w(x_b)=1$. Moreover, $\overline{y_w\circ\sigma}=y_w$.
	\item[$(RR/RC)$] If $\gamma^\sigma_a$ and $\gamma^\sigma_b$ lift up to edges $\gamma_a,\gamma_b\in\fix$, and we have $\iota(\gamma_a)=x_a$, $\iota(\gamma_b)=z_b^\pm$. Then $\re z_b^\pm$ is a well-defined real Laurent series whose specialization on the component associated to $w$ is different from the one of $x_a$. Then we can take $y_w$ such that $y_w(x_r)=\infty$, $y_w(x_a)=0$, $y_w(\re z_b^\pm)=1$.
	\item[$(RC/RR)$] We do the same with $\re z_a^\pm$ and $x_b$.
	\item[$(RC/RC)$] If $\gamma^\sigma_a$ and $\gamma^\sigma_b$ are both of type $(RC)$ we do the same with $\re z_a^\pm$ and $\re z_b^\pm$.
	\item[$(CC/-)$] If $\gamma^\sigma_a$ is of type $(CC)$, then $\gamma^\sigma_a$ lifts to a pair of exchanged edges $\{\gamma_a^\pm\}$ both emanating from $w$. They both have a well-defined $\iota(\gamma_a^\pm)=z_a^\pm$. Then we take $y_w$ such that $y_w(x_r)=\infty$, $y_w(z_a^\pm)=\pm i$, which also is a real coordinate.
	\item[$(-/CC)$] If $\gamma^\sigma_a$ is of type $(RR)$ or $(RC)$ and $\gamma^\sigma_b$ is of type $(CC)$, we do the same with $z_b^\pm$.
	\end{itemize}
\item If $w\notin\fix$, then $I_w^\infty$ consists only of complex markings, all edges emanating from $w$ are of type $(CC)$ and we have a well-defined $\iota(\gamma)$ for each of them. Let $z_a^\varepsilon$ and $z_b^\eta$ be the smallest and biggest elements in $I_w^\infty$. We take $y_w$ such that $y_w(x_r)=\infty$, $y_w(z_a^\varepsilon)=0$, $y_w(z_b^\eta)=1$. This choice ensures that $y_{\sigma(w)}=\overline{y_w\circ\sigma}$.
\end{itemize}

The functions $y_w$ are all coordinates on $C$ sending $x_r$ to $\infty$, therefore we can pass from one to another by a real affine function which we now describe. The proof is straightforward by checking that both coordinates coincide in three points.

\begin{prop}
Let $\gamma\in\Gamma^1$ be a bounded edge.
\begin{itemize}[label=-]
\item If $\gamma\notin\fix$, let $z_a^\varepsilon$ and $z_b^\eta$ be the smallest and biggest elements in $I^\infty_\hg$, then
	$$y_\hg=\frac{y_\tg - y_\tg(z_a^\varepsilon)}{y_\tg(z_b^\eta) - y_\tg(z_a^\varepsilon)} \text{ and }|\gamma|=\val(y_\tg(z_b^\eta) - y_\tg(z_a^\varepsilon)).$$
\item If $\gamma\in\fix$, we make a disjunction according to the type of $\hg$:
	\begin{itemize}[label=$\star$]
	\item[$(RR/RR)$] $y_\hg=\frac{y_\tg - y_\tg(x_a)}{y_\tg(x_b) - y_\tg(x_a)} \text{ and }|\gamma|=\val(y_\tg(x_b) - y_\tg(x_a)).$
	\item[$(RR/RC)$] $y_\hg=\frac{y_\tg - y_\tg(x_a)}{\re y_\tg(z_b^\pm) - y_\tg(x_a)} \text{ and }|\gamma|=\val(\re y_\tg(z_b^\pm) - y_\tg(x_a)).$
	\item[$(RC/RR)$] $y_\hg=\frac{y_\tg - \re y_\tg(x_a)}{y_\tg(x_b) - \re y_\tg(z_a^\pm)} \text{ and }|\gamma|=\val(y_\tg(x_b) - \re y_\tg(z_a^\pm)).$
	\item[$(RC/RC)$] $y_\hg=\frac{y_\tg - \re y_\tg(z_a^\pm)}{\re y_\tg(z_b^\pm) - \re y_\tg(z_a^\pm)} \text{ and }|\gamma|=\val(\re y_\tg(z_b^\pm) - \re y_\tg(z_a^\pm)).$
	\item[$(CC/-)$] $y_\hg=\frac{y_\tg-\re y_\tg(z_a^\pm)}{\im y_\tg(z_a^+)} \text{ and } |\gamma|=\val\left(\im y_\tg(z_a^+)\right).$
	\item[$(-/CC)$] same with $a$ switched by $b$.
	\end{itemize}
\end{itemize}
\end{prop}

For every edge we now define $\alpha_\gamma\in\CC[[t]]^\times$ and $\beta_\gamma\in\CC[[t]]$ which will be the coordinates on the space of real marked curves tropicalizing on $\Gamma$. Once again, the  definition goes through the distinction of the type of edges emanating from $\hg$. Let $\gamma\in\Gamma^1_b$ be a bounded edge:
	\begin{align*}
	(RR/RR)\ & \alpha_\gamma=t^{-|\gamma|}\left(y_\tg(x_b) - y_\tg(x_a)\right), \\
	(RR/RC)\  & \alpha_\gamma=t^{-|\gamma|}\left(\re y_\tg(z_b^\pm) - y_\tg(x_a)\right), \\
	(RC/RR)\  & \alpha_\gamma=t^{-|\gamma|}\left(y_\tg(x_b) - \re y_\tg(z_a^\pm)\right), \\
	(RC/RC)\  & \alpha_\gamma=t^{-|\gamma|}\left(\re y_\tg(z_b^\pm) - \re y_\tg(z_a^\pm)\right), \\
	(CC/-)\  & \alpha_\gamma=t^{-|\gamma|}\im y_\tg(z_a^+), \\
	(-/CC)\  & \alpha_\gamma=t^{-|\gamma|}\im y_\tg(z_b^+). \\
	\end{align*}
	Let $\gamma\in\Gamma^1$ be a non-necessarily bounded edge:
	\begin{itemize}[label=$\star$]
	\item If $\gamma\notin\fix$ is of type $(CC)$ then $\beta_\gamma=y_\tg(z_{\iota(\gamma)}^\varepsilon)$ where $z_{\iota(\gamma)}^\varepsilon$ is the lift of $z_{\iota(\gamma)}$ accessible by $\gamma$.
	\item If $\gamma\in\fix$ is of type $(RR)$ then $\beta_\gamma=y_\tg(x_{\iota(\gamma)})$.
	\item If $\gamma\in\fix$ is of type $(RC)$ then $\beta_\gamma=\re y_\tg(z_{\iota(\gamma)}^\pm)$.
	\end{itemize}
	
	We now can define the function
	$$\Psi_\gamma(y)=\beta_\gamma + t^{|\gamma|}\alpha_\gamma y,$$
	which allows an easy description of the relations between the $y_w$.
	
	\begin{prop}\label{propcalcul}
	The Laurent series $\alpha_\gamma$ and $\beta_\gamma$ satisfy the following properties.
	\begin{enumerate}[label=(\roman*)]
	\item If $\gamma$ is an edge, one has $\alpha_\gamma\in\CC[[t]]^\times$. Moreover, if $\gamma\neq\gamma'$ are two different edges with the same tail $\tg=\mathfrak{t}(\gamma')$, then $\beta_\gamma-\beta_{\gamma'}\in\CC[[t]]^\times$.
	\item They are real: $\alpha_{\sigma(\gamma)}=\overline{\alpha_\gamma}$, $\beta_{\sigma(\gamma)}=\overline{\beta_\gamma}$. In particular $\alpha_\gamma,\beta_\gamma\in\RR[[t]]$ if $\gamma\in\fix$.
	\item For each edge $\gamma$, one has $y_\tg=\Psi_\gamma(y_\hg)$. In particular, for any marked point $q$, one has
	$$y_\tg-y_\tg(q)=t^{|\gamma|}\alpha_\gamma\left(y_\hg-y_\hg(q)\right).$$
	Moreover $\Psi_\gamma$ is real in the sense that $\Psi_{\sigma(\gamma)}(y)=\overline{\Psi_\gamma(\overline{y})}$.
	\item Let $w,w'$ be two vertices, and $\gamma_1,\dots,\gamma_d$ the geodesic path between from $w$ to $w'$. Let $\varepsilon_i=\pm 1$ according to the orientation of $\gamma_i$ in $\Gamma$, and the orientation in the geodesic path agree or disagree. Then
	$$y_w=\Psi_{\gamma_1}^{\varepsilon_1}\circ\cdots\circ\Psi_{\gamma_d}^{\varepsilon_d}(y_{w'}),$$
	and in particular for any marked point $q$:
	$$y_w-y_w(q)=t^{\sum_1^d \varepsilon_i|\gamma_i|}\left(\prod_1^d \alpha_{\gamma_i}^{\varepsilon_i}\right)\left(y_{w'}-y_{w'}(q)\right).$$
	\item Let $q$ be a marked point associated with an unbounded end $e\in\Gamma_\infty^1$, $w$ be a vertex, and $\gamma_1,\dots,\gamma_d,e$ be the geodesic path from $w$ to $q$. Then
	$$y_w(q)=\Psi_{\gamma_1}^{\varepsilon_1}\circ\cdots\circ\Psi_{\gamma_d}^{\varepsilon_d}(\beta_e),$$
	and in particular for every marked point $q$ associated with the unbounded end $e$, if $v_r$ is the vertex adjacent to the unbounded end associated to $x_r$,
	$$y_{v_r}(q)=\beta_{\gamma_1}+t^{|\gamma_1|}\alpha_{\gamma_1}\left( \cdots (\beta_{\gamma_d}+t^{|\gamma_d|}\alpha_{\gamma_d}\beta_e) \right).$$
	\item For every marked point $q_i$ and vertex $w\in\Gamma^0$, $\val(y_w(q))\geqslant 0\Leftrightarrow$ $q$ is accessible by $w$.
	\item For every edge $\gamma\in\Gamma^1$ and every marked point $q\in I^\infty_\tg\backslash I_\hg^\infty$, we have $\val(y_\tg(q)-\beta_\gamma)=0$.
	\end{enumerate}
	\end{prop}
	
	\begin{proof}
	It suffices to check every statement:
	\begin{itemize}[label=-]
	\item $(i)$ and $(ii)$ follow from the definition of $\alpha$ and $\beta$,
	\item $(iii)$ comes from the definition of $\Psi_\gamma$ and from the fact that $\alpha$ and $\beta$ are real,
	\item $(iv)$ and $(v)$ are just iterations from $(iii)$,
	\item $(vi)$ comes from $(v)$ and,$(vii)$ follows from $(i)$ and $(v)$.
	\end{itemize}	 
	\end{proof}
	
	\begin{rem}
	Formally speaking, this proposition is the direct translation of Proposition $4.3$ in \cite{tyomkin2017enumeration}, in the setting of curves with a real structure. Although the formulas seem quite repulsive at the first look, the meaning of each object must be clear. 
The formal series $\alpha_\gamma$ and $\beta_\gamma$ allow one to recover the coordinates of the marked points, following the formula $(v)$ of Proposition \ref{propcalcul}. 
The formal series $\alpha_\gamma$ are the \textit{"phase length"} of the edge $\gamma$, in contrast to $t^{|\gamma|}$ which could be called \textit{"valuation length"}, while the formal series $\beta_\gamma$ are the directions one needs to follow at each $w$ in order to get to the points of $I^\infty_\hg$. In other terms, $\alpha$ and $\beta$ provide the necessary coefficients to find the coordinates of the marked points. One could say that the abstract tropical curve $\Gamma$ only remembers the valuation information, while the formal series $\alpha$ and $\beta$ encode the phase information. 
	\end{rem}
	
	Let $v=v_r$ be the vertex adjacent to the end $x_r$, the tuple
	$$(y_v(x_1),\dots,y_v(x_{r-1}),y_v(z_1^+),\dots,y_v(z_s^+))\in\RR((t))^{r-1}\times\CC((t))^s$$
	provides a system of coordinates on the moduli space of real rational marked curves, that we can restrict on the moduli space of curves tropicalizing on $\Gamma$. Notice that the choice of $y_v$ fixes the value of some members of the tuple. The definition of $\alpha,\beta$ along with a quick induction ensures that they can be written in terms of $\big(y_v(q)\big)_q$. Conversely, the formula from Proposition \ref{propcalcul}$(v)$ allows to recover $\big(y_v(q)\big)_q$ from $\alpha$ and $\beta$. Therefore, they also provide a system of coordinates. Moreover,   Proposition \ref{propcalcul} describes the set of possible values of $\alpha$, $\beta$, since the formula from Proposition \ref{propcalcul}$(v)$ gives the values of the points to choose on $\PP^1\left(\CC((t))\right)$, in order to make it into a marked curve with the right tropicalization and the right formal series $\alpha$, $\beta$.\\
	
	We denote by $\mathcal{A}$ the space $\left(\RR[[t]]^\times\right)^{e_\RR} \times \left(\CC[[t]]^\times\right)^{e_\CC}$ of possible values of $\alpha$. We denote by $\mathcal{B}$ the space of possible values of $\beta$ satisfying the conditions of Proposition \ref{propcalcul}.

	\subsection{Space of morphisms with given tropicalization} 
	
	Let $h:\Gamma\rightarrow N_\RR$ be a real rational parametrized tropical curve of degree $\Delta$. In this subsection we give an explicit description of morphisms $f:(C,\textbf{x},\textbf{z}^\pm)\rightarrow\mathrm{Hom}(M,\CC((t))^*)$ of degree $\Delta$ tropicalizing to it, and for which $(C,\textbf{x},\textbf{z}^\pm)$ is a smooth connected marked rational curve.\\
	
	Let $f:C\rightarrow\mathrm{Hom}(M,\CC((t))^*)$ be a real morphism that tropicalizes to $h$. By assumption, in the coordinate $y_w$, the morphism $f$ takes the following form:
	$$f(y_w)=t^{h(w)}\chi_w\prod_{i\in I^\infty_w} (y_w-y_w(q_i))^{n_i} \prod_{i\notin I^\infty_w}\left( \frac{y_w}{y_w(q_i)}-1\right)^{n_i} \in\mathrm{Hom}(M,\CC((t))^*).$$
	
\begin{rem}	
Notice that $t^{h(w)}$ denotes the morphism $m\mapsto t^{h(w)(m)}\in\CC((t))^*$. The choice of normalization in the product ensures that $\chi_w:M\rightarrow\CC((t))^*$ has value in $\CC[[t]]^\times$. Finally, both products are indexed by the ends of $\Gamma$, and although the writing does not emphasize this aspect, there are real ends and complex ends. Furthermore, there is a constant term $(-1)$ in the second product, corresponding to the root $x_r$, for which $y_w(x_r)=\infty$ for any $w$.
\end{rem}

We now relate the expression of $f$ in two different vertices. We assume that they are connected by an edge $\gamma$. Let
$$ \boxed{ \phi_\gamma=\prod_{i\in I^\infty_\tg\backslash I^\infty_\hg}\left( y_\tg(q_i)-\beta_\gamma\right)^{n_i} \prod_{i\notin I^\infty_\tg}\left( 1-\frac{\beta_\gamma}{y_\tg(q_i)} \right)^{n_i}\in\mathrm{Hom}(M,\CC((t))^*. }$$
Notice that $\phi_\gamma$ depends only on the value of $\alpha,\beta$, and is thus a function $\phi_\gamma(\alpha,\beta)$.

\begin{prop}\label{morphism}\cite{blomme2020tropical}
We have the following properties of $\chi_w$ and $\phi_\gamma$:
\begin{enumerate}[label=(\roman*)]
\item The $\chi_w$ are real: $\chi_{\sigma(w)}=\overline{\chi_w}$. In particular, if $w\in\fix$, $\chi_w$ takes values in $\RR[[t]]^\times$.
\item The co-characters $\phi_\gamma$ are real: $\phi_{\sigma(\gamma)}=\overline{\phi_\gamma}$. In particular, if $\gamma\in\fix$ is a fixed edge, the co-character $\phi_\gamma$ takes values in $\RR[[t]]^\times$.
\item For any edge $\gamma$, let $n_\gamma$ denote the slope of $h$ on $\gamma$. One has the following \emph{"transfer equation"}:
$$ \boxed {\phi_\gamma\cdot\frac{\chi_\tg}{\chi_\hg}\cdot\alpha_\gamma^{n_\gamma}=1\in\mathrm{Hom}(M,\CC((t))^*). }$$
\end{enumerate}
\end{prop}

Conversely, if we are given a real family of $\chi_w:M\rightarrow\CC[[t]]^\times$ such that Proposition \ref{morphism} holds for any $\gamma$, then the maps defined by the formulas 
$$f(y_w)=t^{h(w)}\chi_w\prod_{i\in I^\infty_w} (y_w-y_w(q_i))^{n_i} \prod_{i\notin I^\infty_w}\left( \frac{y_w}{y_w(q_i)}-1\right)^{n_i}$$
agree and define a real morphism $f$ tropicalizing to $h:\Gamma\rightarrow N_\RR$.\\

If $G$ is an abelian group, let $N_G=N\otimes G$. Let $\mathcal{X}=N_{\RR[[t]]^\times}^{v_\RR}\times N_{\CC[[t]]^\times}^{v_\CC}$ be the space where the tuple $\chi$ is chosen. Then, the space of morphisms $f:(C,\textbf{x},\textbf{z}^\pm)\rightarrow N\otimes\CC((t))^*$ tropicalizing to $h:\Gamma\rightarrow N_\RR $ is the subset of $\mathcal{X}\times\mathcal{A}\times\mathcal{B}$ given by the following equations:
$$\forall\gamma\in \Gamma^1_b\ :\ \phi_\gamma(\alpha,\beta)\cdot\frac{\chi_\tg}{\chi_\hg}\cdot\alpha_\gamma^{n_\gamma}=1\in\mathrm{Hom}(M,\CC((t))^*).$$
The tuples $\alpha,\beta$ deal with the tropicalization of the curve, and the tuple $\chi$ with the tropicalization of the morphism.
	
	\subsection{Evaluation map} 

We now use the previous description to write down the conditions that a curve having fixed moments must satisfy. For each $n_j\in\Delta$, let $m_j=\iota_{n_j}\omega$ be the monomial used to measure the moment of the corresponding end $q_j$.\\

Let $q_j$ be a real or complex marked point, and $v_j$ be the adjacent vertex of the associated unbounded end $e_j$. In the coordinate $y_{v_j}$, the expression of the moment of the marked point $q_j$ takes the following form:
$$f^*\chi^{m_j}|_{q_j} = t^{h(v_j)(m_j)}\chi_{v_j}(m_j)\prod_{i\in I^\infty_{v_j}}\left(\beta_{e_j}-y_{v_j}(q_i)\right)^{\omega(n_j,n_i)}\prod_{i\notin I^\infty_{v_j}}\left( \frac{\beta_{e_j}}{y_{v_j}(q_i)}-1\right)^{\omega(n_j,n_i)}.$$
We then put 
$$\boxed{ \varphi_j=\prod_{i\in I^\infty_{v_j}}\left(\beta_{e_j}-y_{v_j}(q_i)\right)^{\omega(n_j,n_i)}\prod_{i\notin I^\infty_{v_j}}\left( \frac{\beta_{e_j}}{y_{v_j}(q_i)}-1\right)^{\omega(n_j,n_i)}\in\CC[[t]]^\times, }$$
which, according to Proposition \ref{propcalcul}, is an invertible formal series only depending on $\alpha,\beta$. The series is invertible since the only terms of the product having positive valuation are taken with a zero exponent.

\begin{prop}
The formal series $\varphi_j$ are real: $\varphi_{\sigma(j)}=\overline{\varphi_j}$ for every end $e_j$, and in particular, if $x_j$ is a real marked point, then $\varphi_j\in\RR[[t]]^\times$. Moreover, $\varphi$ is a function of $\alpha,\beta$, \textit{i.e.} it does not depend on $\chi$.
\end{prop}

\begin{proof}
It follows from the fact that all the quantities that intervene in the definition of $\varphi_j$ are real. The second part is obvious.
\end{proof}



Thus, with this new notation, the evaluation map takes the following form:

$$f^*\chi^{m_j}|_{q_j} = t^{h(v_j)(m_j)}\chi_{v_j}(m_j)\varphi_j.$$

	\subsection{Correspondence theorem}
	\label{correspondence theorem section}
	Now we look at the following map: 
	$$\boxed{ \begin{array}{rccl}
	\Theta : & \mathcal{X}\times\mathcal{A}\times\mathcal{B} &\longrightarrow & \left(N_{\RR[[t]]^\times}^{e_\RR}\times N_{\CC[[t]]^\times}^{e_\CC}\right)\times\RR[[t]]^{\times r-1}\times\CC[[t]]^{\times s} \\
		& (\chi,\alpha,\beta) & \longmapsto & \left( \left(\phi_\gamma\cdot\frac{\chi_\tg}{\chi_\hg}\cdot\alpha_\gamma^{n_\gamma}\right)_{\gamma\in\Gamma^1_b},\left(\chi_{v_j}(m_j)\varphi_j \right)_j \right) \\
	\end{array}. }$$
	This map is the same as in \cite{tyomkin2017enumeration} but for a curve endowed with a non-trivial real involution. That is why we take every real vertex or real edge with real coefficients, and only one of every pair of complex vertices or complex edges with complex coefficients.\\

	\begin{lem}\cite{blomme2020tropical}
	The dimensions of the source and target spaces are the same.
	\end{lem}
		
	Let $\zeta\in\RR((t))^{*r-1}\times \CC((t))^{*s}$ and $(\mu_j)=(\val\zeta_j)$ be families of complex and tropical moments, both assumed to be generic. We denote by $\zeta_j^\Gamma=\zeta_j t^{-\mu_j}\in\CC[[t]]^\times$ the complex moments normalized to have a zero valuation. The next proposition is straightforward to check.

\begin{prop}
Any family of classical curves passing through the symmetric configuration $\mathcal{P}$ tropicalizes on a real parametrized tropical curve $h:\Gamma\rightarrow N_\RR$ satisfying $\mathrm{ev}(\Gamma)=\mu$, and specializes to a tuple $(\chi,\alpha,\beta)\in\mathcal{X}\times\mathcal{A}\times\mathcal{B}$, satisfying $\Theta(\chi,\alpha,\beta)=(1,\zeta^\Gamma)$.
\end{prop}

\begin{rem}
Moreover, the plane tropical curve image $h(\Gamma)$ has a unique parametrization as a parametrized tropical curve of degree $\Delta(s)$. Notice that the space $\mathcal{X}\times\mathcal{A}\times\mathcal{B}$ depends on the choice of the parametrized tropical curve $(\Gamma,h)$.\\
\end{rem}
	
	Conversely, for each real parametrized tropical curve $(\Gamma,h)$ with $\mathrm{ev}(\Gamma)=\mu$, we need to find the classical curves tropicalizing on $(\Gamma,h)$ and passing through the symmetric configuration $\mathcal{P}$. Such a curve corresponds to a point in the moduli space $\mathcal{X}\times\mathcal{A}\times\mathcal{B}$. Finding the curves passing through the symmetric configuration $\mathcal{P}$ and tropicalizing on $\Gamma$ thus amounts to solve for $(\chi,\alpha,\beta)$ the equation $\Theta(\chi,\alpha,\beta)=(1,\zeta^\Gamma)$, for any possible sign of $\zeta^\Gamma$. We say that $(\chi_0,\alpha_0,\beta_0)$ is a first order solution if $\Theta(\chi_0,\alpha_0,\beta_0)=(1,\zeta^\Gamma)$ mod $t$.\\
	
	Let be given a parametrized tropical curve $h_0:\Gamma_0\rightarrow N_\RR$ of degree $\Delta(s)$ with $\text{ev}(\Gamma_0)=\mu$.  Notice that, as $\mu$ is generic, $C_\text{trop}=h(\Gamma)_0$ is a nodal curve. The real rational curves with image $C_\text{trop}$ are described in Lemma \ref{realstruc}. Assume that $(\Gamma,h)$ has no flat vertex. We prove in the next section that it is a necessary condition for $(\Gamma,h)$ to have first order solutions. We then have the following theorem, that allows us to lift first order solutions to true solutions, provided that the Jacobian is invertible.
	
	\begin{theo}
	\label{realization theorem}
	For each real parametrized tropical curve $(\Gamma,h)$ of degree $\Delta$, obtained from a parametrized tropical curve of degree $\Delta(s)$ passing through $\mu$, without any flat vertex, and each first order solution $\Theta(\chi_0,\alpha_0,\beta_0)=(1,\zeta^\Gamma)$:
	\begin{itemize}
	\item the Jacobian of $\Theta$ at $(\chi_0,\alpha_0,\beta_0)$ is invertible at first order,
	\item there is a unique lift of $(\chi_0,\alpha_0,\beta_0)$ to a true solution $(\chi,\alpha,\beta)$ in $\mathcal{X}\times\mathcal{A}\times\mathcal{B}$.
	\end{itemize}
	\end{theo}
	
	\begin{rem}
	Most of the proof consists in showing that the Jacobian is indeed invertible for all the considered first order solutions.
	\end{rem}
		
	\begin{proof}
	Let $h:\Gamma\rightarrow N_\RR$ be one of the above real parametrized tropical curves, meaning that $h(\Gamma)$ is a plane tropical curve such that its parametrization $h_0:\Gamma_0\rightarrow N_\RR$ as a curve of degree $\Delta(s)$ satisfies $\text{ev}(\Gamma_0)=\mu$, and that $(\Gamma,h)$ has no flat vertex. Let $(\chi_0,\alpha_0,\beta_0)$ be a first order solution. Using Hensel's lemma, the invertibleness of the Jacobian at $(\chi_0,\alpha_0,\beta_0)$ allows one to lift to a unique true solution. Therefore, the second statement follows from the first, which we now show by induction. For simplicity, and because of the multiplicative nature of the map $\Theta$, we use logarithmic coordinates for every variable except $\beta$. It means that we look at $\log\Theta$, depending on the new variables $(\log\chi,\log\alpha,\beta)$. Each time, the logarithm is taken coordinate by coordinate. Hence, if $\gamma$ is an edge, $\log\alpha_\gamma$ and $\beta_\gamma$ are scalars, while if $w$ is a real vertex, $\log\chi_w:M\rightarrow\RR$ is an element of $N_\RR$.\\
	
To compute the Jacobian	relative to coordinates $(\log\chi,\log\alpha,\beta)$ at $t=0$, we can first put $t=0$. Thus, we get for every bounded edge $\gamma\in\Gamma^1_b$:
$$\phi_\gamma|_{t=0}=\prod_{ \substack{ \gamma'\neq\gamma \\ \mathfrak{t}(\gamma')=\tg }} (\beta_{\gamma'}|_{t=0}-\beta_\gamma|_{t=0})^{n_{\gamma'}}\in N\otimes\RR^*,$$
and for every unbounded end $e_j$:
$$\varphi_j|_{t=0}=\pm \prod_{ \substack{ \gamma'\neq e_j \\ \mathfrak{t}(\gamma')=v_j }} (\beta_{e_j}|_{t=0}-\beta_{\gamma'}|_{t=0})^{\omega(n_j,n_{\gamma'})}\in\CC^*.$$
Notice that at the first order, $\phi_\gamma$ depends only on $\beta$ and not $\alpha$. For convenience of notation, all the following computations are taken at the first order, and we drop $"|_{t=0}"$ out of the notation.\\
	
For a bounded edge $\gamma\in\Gamma^1_b$, let $N_\gamma=\phi_\gamma\cdot\frac{\chi_\tg}{\chi_\hg}\cdot\alpha_\gamma^{n_\gamma}\in N_{\KK[[t]]^\times}$, where $\KK=\RR$ ou $\CC$ according to whether or not $\gamma\in\fix$. Similarly, let $X_j=\chi_{v_j}(m_j)\varphi_j\in\RR[[t]]^\times$ for a real end $x_j$, and $Z_j=\chi_{v_j}(m_j)\varphi_j\in\CC[[t]]^\times$ for a complex end $z_j^\pm$. The variables $N_\gamma,X_j$ and $Z_j$ index the lines of the Jacobian matrix $\parfrac{\log\Theta}{(\log\chi,\log\alpha,\beta)}$. Thus, we have
$$\log N_\gamma=\sum_{ \substack{ \gamma'\neq\gamma \\ \mathfrak{t}(\gamma')=\tg }}  n_{\gamma'}\log(\beta_{\gamma'}-\beta_\gamma) \ +\ \log\chi_\tg-\log\chi_\hg+n_\gamma\log\alpha_\gamma \in\mathrm{Hom}(M,\RR),$$

$$\log X_j=\sum_{ \substack{ \gamma'\neq e_j \\ \mathfrak{t}(\gamma')=v_j }} \omega(n_j,n_{\gamma'}) \log(\beta_{e_j}-\beta_{\gamma'}) \ +\ \log\chi_{v_j}(m_j)\in\RR,$$

$$\log Z_j=\sum_{ \substack{ \gamma'\neq e_j \\ \mathfrak{t}(\gamma')=v_j }} \omega(n_j,n_{\gamma'}) \log(\beta_{e_j}-\beta_{\gamma'}) \ +\ \log\chi_{v_j}(m_j)\in\CC.$$

Notice that the vertices are either trivalent, quadrivalent including a pair of exchanged edges, or pentavalent including two pairs of exchanged edges. This means that the sum over $\gamma'$ is one constant term if $\tg$ is trivalent, contains one or two terms depending on a real parameter $\beta$ if $\tg$ is quadrivalent, and two or three terms depending on a complex parameter $\beta$ if it is pentavalent, as depicted below. The following lemmata are used to prove by induction that the Jacobian is invertible.

	\begin{lem}
	Let $V$ be a real trivalent vertex adjacent to two real unbounded ends.
	\begin{itemize}[label=-]
	\item If $V$ is the only vertex, the Jacobian is invertible,
	\item If not, the invertibleness of the Jacobian reduces to the invertibleness of the Jacobian for the curve where $V$ is removed and the unique bounded edge adjacent to $V$ is replaced with an unbounded ends of the same direction.
	\end{itemize}
	\end{lem}
	
	\begin{proof}
	Let $\gamma$ be the edge with $\hg=V$. We assume that $\gamma$ is a bounded edge. The matrix has the following form

$$\text{Jac}\Theta=\begin{array}{lcccccc}
                                        & \multicolumn{1}{l}{} & \multicolumn{1}{l}{} & \multicolumn{1}{l}{}   & \multicolumn{2}{c}{\chi_V}            & \multicolumn{1}{c}{\alpha_\gamma}                 \\ \cline{5-7} 
                                        &                      &                      & \multicolumn{1}{c|}{}  & 0     & \multicolumn{1}{c|}{0}     & \multicolumn{1}{c|}{0}                    \\
                                        &                      & \ast                  & \multicolumn{1}{c|}{}  & \vdots & \multicolumn{1}{c|}{\vdots} & \multicolumn{1}{c|}{\vdots}                \\
                                        &                      &                      & \multicolumn{1}{c|}{}  & 0     & \multicolumn{1}{c|}{0}     & \multicolumn{1}{c|}{0}                    \\ \cline{2-7} 
\multicolumn{1}{r|}{\multirow{2}{*}{$N_\gamma$}} &        \ast              &                      \cdots & \multicolumn{1}{c|}{\ast}  & -1    & \multicolumn{1}{c|}{0}     & \multicolumn{1}{c|}{\multirow{2}{*}{$n_\gamma$}} \\
\multicolumn{1}{r|}{}                   &        \ast              &                     \cdots & \multicolumn{1}{c|}{\ast}  & 0     & \multicolumn{1}{c|}{-1}    & \multicolumn{1}{c|}{}                     \\ \cline{2-7} 
\multicolumn{1}{r|}{X_0}                  & 0                    & \cdots                & \multicolumn{1}{c|}{0} & \multicolumn{2}{c|}{m_0}             & \multicolumn{1}{c|}{0}                    \\ \cline{2-7} 
\multicolumn{1}{r|}{X_1}                  & 0                    & \cdots                & \multicolumn{1}{c|}{0} & \multicolumn{2}{c|}{m_1}             & \multicolumn{1}{c|}{0}                    \\ \cline{2-7} 
\end{array} \ ,$$
where $n_\gamma=n_0+n_1$. Indeed, at first order we have
$$\left\{ \begin{array}{l}
\log N_\gamma=\sum_{ \substack{ \gamma'\neq\gamma \\ \mathfrak{t}(\gamma')=\tg }}  n_{\gamma'}\log(\beta_{\gamma'}-\beta_\gamma) \ +\ \log\chi_\tg-\log\chi_V+n_\gamma\log\alpha_\gamma \in\mathrm{Hom}(M,\RR) ,\\
\log X_0=\log\chi_V(m_0) ,\\
\log X_1=\log\chi_V(m_1) .\\
\end{array}\right.$$
Therefore we have the following partial derivatives:
$$\boxed{ \parfrac{\log N_\gamma}{\log\chi_V}=-I_2 , \ \parfrac{\log N_\gamma}{\log\alpha_\gamma}=n_\gamma ,} \text{ and } \boxed{ \parfrac{\log X_0}{\log\chi_V}=m_0 , \ \parfrac{\log X_1}{\log\chi_V}=m_1 .}$$
By developing with respect to the last two rows, since $(m_0,m_1)$ are free, we are left with the following determinant:
$$\begin{array}{lcccc}
                                        & \multicolumn{1}{l}{} & \multicolumn{1}{l}{} & \multicolumn{1}{l}{}     & \multicolumn{1}{l}{\alpha_\gamma}                 \\ \cline{5-5} 
                                        &                      &                      & \multicolumn{1}{c|}{}    & \multicolumn{1}{c|}{0}                    \\
                                        &                      & \ast                  & \multicolumn{1}{c|}{}    & \multicolumn{1}{c|}{\vdots}                \\
                                        &                      &                      & \multicolumn{1}{c|}{}    & \multicolumn{1}{c|}{0}                    \\ \cline{2-5} 
\multicolumn{1}{r|}{\multirow{2}{*}{$N_\gamma$}} & \ast                  & \cdots                & \multicolumn{1}{c|}{\ast} & \multicolumn{1}{c|}{\multirow{2}{*}{$n_\gamma$}} \\
\multicolumn{1}{r|}{}                   & \ast                  & \cdots                & \multicolumn{1}{c|}{\ast} & \multicolumn{1}{c|}{}                     \\ \cline{2-5} 
\end{array} \ .$$
If $\gamma$ was an unbounded end, we would be left with the empty matrix and we would have proven invertibleness. Otherwise, the last two rows correspond to a copy of $N_\RR$, and are thus given by two elements of $M_\RR$, the dual of $N_\RR$. Up to a change of basis, one can assume that one of these elements of $M_\RR$ is $\omega(n_\gamma,-)$, which takes $0$ value on $n_\gamma$. Thus, by making a development with respect to the column, noticing that $\omega(n_\gamma,N_\gamma)$ is precisely the line corresponding to the unbounded end in the curve where $V$ is removed and $\gamma$ replaced by an unbounded end, we are reduced to the invertibleness of the corresponding matrix.	
	\end{proof}

	\begin{lem}
	Let $V$ be a complex trivalent vertex adjacent to two complex unbounded ends. The invertibleness of the Jacobian reduces to the invertibleness of the Jacobian for the curve where $V$ and $\sigma(V)$ are removed and the unique bounded edges adjacent to $V$ and $\sigma(V)$ are replaced with unbounded ends of the same direction.
	\end{lem}
	
	\begin{proof}
Let $\gamma$ be the bounded edge with $\hg=V$. The matrix has the following form

	$$\mathrm{Jac}\Theta=\begin{array}{ccccccccccc}
                                                                                   &                     &                       &                       &                                          & \multicolumn{2}{c}{|\chi_V|}             & \multicolumn{2}{c}{\arg\chi_V} & \multicolumn{2}{c}{\alpha_\gamma}                                           \\ \cline{6-11} 
                                                                                   &                     & \ast                  &                       & \multicolumn{1}{c|}{}                    & \multicolumn{2}{c}{0}                     & \multicolumn{2}{c|}{0}                                         & 0                                         & \multicolumn{1}{c|}{0}                                         \\
                                                                                   &                     &                       &                       & \multicolumn{1}{c|}{}                    & \multicolumn{2}{c}{\vdots} & \multicolumn{2}{c|}{\vdots}                     & \vdots                     & \multicolumn{1}{c|}{\vdots}                     \\
                                                                                   &                     &                       &                       & \multicolumn{1}{c|}{}                    & \multicolumn{2}{c}{0}                     & \multicolumn{2}{c|}{0}                                         & 0                                         & \multicolumn{1}{c|}{0}                                         \\ \cline{2-11} 
\multicolumn{1}{c|}{\multirow{2}{*}{$|N_\gamma|$}}                   & \ast & \cdots & \cdots & \multicolumn{1}{c|}{\ast} & -1                   & 0                   & 0                   & \multicolumn{1}{c|}{0}                  & \multirow{2}{*}{$n_\gamma$} & \multicolumn{1}{c|}{\multirow{2}{*}{$0$}}                        \\
\multicolumn{1}{c|}{}                                                              & \ast & \cdots & \cdots & \multicolumn{1}{c|}{\ast} & 0                    & -1                  & 0                   & \multicolumn{1}{c|}{0}                  &                                           & \multicolumn{1}{c|}{}                                          \\
\multicolumn{1}{c|}{\multirow{2}{*}{$\arg N_\gamma$}} & \ast & \cdots & \cdots & \multicolumn{1}{c|}{\ast} & 0                    & 0                   & -1                  & \multicolumn{1}{c|}{0}                  & \multirow{2}{*}{$0$}                        & \multicolumn{1}{c|}{\multirow{2}{*}{$n_\gamma$}} \\
\multicolumn{1}{c|}{}                                                              & \ast & \cdots & \cdots & \multicolumn{1}{c|}{\ast} & 0                    & 0                   & 0                   & \multicolumn{1}{c|}{-1}                 &                                           & \multicolumn{1}{c|}{}                                          \\ \cline{2-11} 
\multicolumn{1}{c|}{\multirow{2}{*}{$Z_0$}}                                         & 0                   & \cdots & \cdots & \multicolumn{1}{c|}{0}                   & \multicolumn{2}{c}{m_0}                  & \multicolumn{2}{c|}{0}                                         & 0                                         & \multicolumn{1}{c|}{0}                                         \\
\multicolumn{1}{c|}{}                                                              & 0                   & \cdots & \cdots & \multicolumn{1}{c|}{0}                   & \multicolumn{2}{c}{0}                     & \multicolumn{2}{c|}{m_0}                                      & 0                                         & \multicolumn{1}{c|}{0}                                         \\ \cline{2-11} 
\multicolumn{1}{c|}{\multirow{2}{*}{$Z_1$}}                                         & 0                   & \cdots & \cdots & \multicolumn{1}{c|}{0}                   & \multicolumn{2}{c}{m_1}                  & \multicolumn{2}{c|}{0}                                         & 0                                         & \multicolumn{1}{c|}{0}                                         \\
\multicolumn{1}{c|}{}                                                              & 0                   & \cdots & \cdots & \multicolumn{1}{c|}{0}                   & \multicolumn{2}{c}{0}                     & \multicolumn{2}{c|}{m_1}                                      & 0                                         & \multicolumn{1}{c|}{0}                                         \\ \cline{2-11} 
\end{array}$$
Indeed, the coordinates are now complex: $\alpha_\gamma$ is a complex scalar, $\chi_V$ has complex values, $N_\gamma\in N\otimes\CC^*$ and $Z_0,Z_1\in\CC^*$. To get real coordinate, we split them into real and imaginary parts. But as we take logarithmic coordinate, the real/imaginary decomposition emphasizes the modulus/argument decomposition of the coordinates. One has
$$\left\{\begin{array}{l}
\log\alpha_\gamma=\log|\alpha_\gamma|+i\arg\alpha_\gamma ,\\
\log\chi_V=\log|\chi_V|+i\arg\chi_V.\\
\end{array}\right.$$
Therefore, we also have
\begin{align*}
\log N_\gamma & =\sum_{ \substack{ \gamma'\neq\gamma \\ \mathfrak{t}(\gamma')=\tg }}  n_{\gamma'}\log(\beta_{\gamma'}-\beta_\gamma) \ +\ \log\chi_\tg-\log|\chi_V|-i\arg\chi_V + n_\gamma\log|\alpha_\gamma|+in_\gamma\arg\alpha_\gamma ,\\
\log Z_j & =\log|\chi_V|(m_j)+i\arg\chi_V(m_j) +c.\\
\end{align*}
Hence, we get the announced Jacobian matrix. Then, the computation is entirely similar, except it is done twice: develop with respect to the last four rows, do a change of basis in the new last four rows and develop with respect to the last two columns, thus reducing to the Jacobian of the announced curve.	
	\end{proof}

	\begin{lem}
	Let $V$ be a real quadrivalent vertex adjacent to real unbounded end and two exchanged complex unbounded ends.
	\begin{itemize}[label=-]
	\item If $V$ is the only vertex, the Jacobian is invertible,
	\item If not, the invertibleness of the Jacobian reduces to the invertibleness of the Jacobian for the curve where $V$ is removed and the unique bounded edge adjacent to $V$ is replaced with an unbounded edge of the same direction.
	\end{itemize}
	\end{lem}
	
	\begin{proof}
	If the vertex $V$ is adjacent to two complex unbounded ends (directed by $n_0$), and to a real unbounded end (directed by $n_1$), let $\beta_1$ be the $\beta$ coordinate associated to the real end. Let $\gamma$ be the edge with $\hg=V$, thus directed by $2n_0+n_1$. The determinant takes the following form:
$$\text{Jac}\Theta=\begin{array}{lccccccc}
                                        &     &       &                          & \multicolumn{2}{c}{\chi_V}            & \alpha_\gamma                                   & \beta_1                       \\ \cline{5-8} 
                                        &     &       & \multicolumn{1}{c|}{}    & 0     & \multicolumn{1}{c|}{0}     & \multicolumn{1}{c|}{0}                  & \multicolumn{1}{c|}{0}     \\
                                        &     & \ast   & \multicolumn{1}{c|}{}    & \vdots & \multicolumn{1}{c|}{\vdots} & \multicolumn{1}{c|}{\vdots}              & \multicolumn{1}{c|}{\vdots} \\
                                        &     &       & \multicolumn{1}{c|}{}    & 0     & \multicolumn{1}{c|}{0}     & \multicolumn{1}{c|}{0}                  & \multicolumn{1}{c|}{0}     \\ \cline{2-8} 
\multicolumn{1}{l|}{\multirow{2}{*}{$N_\gamma$}} & \ast & \cdots & \multicolumn{1}{c|}{\ast} & -1    & \multicolumn{1}{c|}{0}     & \multicolumn{1}{c|}{\multirow{2}{*}{$n_\gamma$}} & \multicolumn{1}{c|}{0}     \\
\multicolumn{1}{l|}{}                   & \ast & \cdots & \multicolumn{1}{c|}{\ast} & 0     & \multicolumn{1}{c|}{-1}    & \multicolumn{1}{c|}{}                   & \multicolumn{1}{c|}{0}     \\ \cline{2-8} 
\multicolumn{1}{l|}{\multirow{2}{*}{$Z_0$}} & 0   & \cdots & \multicolumn{1}{c|}{0}   & \multicolumn{2}{c|}{m_0}             & \multicolumn{1}{c|}{0}                  & \multicolumn{1}{c|}{\beta_1}  \\
\multicolumn{1}{l|}{}                   & 0   & \cdots & \multicolumn{1}{c|}{0}     & \multicolumn{2}{c|}{0}     &  \multicolumn{1}{c|}{0}                  & \multicolumn{1}{c|}{1}     \\ \cline{2-8} 
\multicolumn{1}{l|}{X_1}                  & 0   & \cdots & \multicolumn{1}{c|}{0}   & \multicolumn{2}{c|}{m_1}             & \multicolumn{1}{c|}{0}                  & \multicolumn{1}{c|}{-2\beta_1}  \\ \cline{2-8} 
\end{array}.
$$
Indeed, we have
$$\left\{ \begin{array}{rl}
\log N_\gamma &=\cdots + \log\chi_\tg -\log\chi_V+n_\gamma\log\alpha_\gamma,\\
\log Z_0 &= \omega(n_0,n_1)\log(i-\beta_1)+\log\chi_V(m_0),\\
\log X_1 &= \omega(n_1,n_0)\log(\beta_1-i)+\omega(n_1,n_0)\log(\beta_1+i)+\log\chi_V(m_1)\\
& = \omega(n_1,n_0)\log(\beta_1^2+1)+\log\chi_V(m_1).\\
\end{array}\right.$$
Hence, we recover the partial derivatives $\parfrac{\log N_\gamma}{\log\chi_V}$, $\parfrac{\log N_\gamma}{\log\alpha_\gamma}$, $\parfrac{\log Z_0}{\log\chi_V}$ and $\parfrac{\log X_1}{\log\chi_V}$, but we also get the new partial derivatives:
$$ \boxed{ \parfrac{\log X_1}{\beta_1}=\omega(n_1,n_0)\frac{2\beta_1}{\beta_1^2+1} } \text{ and } \boxed{ \parfrac{\log Z_0}{\beta_1}= \omega(n_0,n_1)\frac{-1}{i-\beta_1} = \frac{\omega(n_0,n_1)}{\beta_1^2+1}(\beta_1+i).}$$
Notice that we dropped out the non-zero constant factor $\frac{\omega(n_0,n_1)}{\beta_1^2+1}$ in the last column. Then, one can make a development with respect to the penultimate row, and then the second resulting last rows. We get the empty determinant if $\gamma$ is an unbounded end, thus proving invertibleness. Otherwise, as before, we make a change of basis in the rows associated to $N_\gamma$ and develop with respect to the last column, thus reducing to the associated determinant.	
	\end{proof}

	\begin{lem}
	Let $V$ be a real pentavalent vertex adjacent to two real unbounded ends.
	\begin{itemize}[label=-]
	\item If $V$ is the only vertex, the Jacobian is invertible,
	\item If not, the invertibleness of the Jacobian reduces to the invertibleness of the Jacobian for the curve where $V$ is removed and the unique bounded edge adjacent to $V$ is replaced with an unbounded edge of the same direction.
	\end{itemize}
	\end{lem}
	
	\begin{proof}
	Let the adjacent unbounded ends be labeled $0$ and $1$, and let $\gamma$ be the unique edge with $\hg=V$. In the coordinate $y_V$, the ends associated to the label $0$ are $\pm i$, and the ends associated to $1$ are $\beta$ and $\overline{\beta}$. The evaluation matrix takes the following form:
	
$$\mathrm{Jac}\Theta = \begin{array}{cccccccccc}
                                                               &                     &                       &                       &                                          & \multicolumn{2}{c}{\chi_V} & \alpha                                          & \multicolumn{2}{c}{\beta}                                                                          \\ \cline{6-10} 
                                                               &                     & \ast   &                       & \multicolumn{1}{c|}{}                    & \multicolumn{2}{c|}{0}                     & \multicolumn{1}{c|}{0}                                         & \multicolumn{2}{c|}{0}                                                                                            \\
                                                               &                     &                       &                       & \multicolumn{1}{c|}{}                    & \multicolumn{2}{c|}{\vdots} & \multicolumn{1}{c|}{\vdots}                     & \multicolumn{2}{c|}{\vdots}                                                                        \\
                                                               &                     &                       &                       & \multicolumn{1}{c|}{}                    & \multicolumn{2}{c|}{0}                     & \multicolumn{1}{c|}{0}                                         & \multicolumn{2}{c|}{0}                                                                                            \\ \cline{2-10} 
\multicolumn{1}{c|}{\multirow{2}{*}{$N_\gamma$}} & \ast & \cdots & \cdots & \multicolumn{1}{c|}{\ast} & -1        & \multicolumn{1}{c|}{0}         & \multicolumn{1}{c|}{\multirow{2}{*}{$n_\gamma$}} & 0                                             & \multicolumn{1}{c|}{0}                                            \\
\multicolumn{1}{c|}{}                                          & \ast & \cdots & \cdots & \multicolumn{1}{c|}{\ast} & 0         & \multicolumn{1}{c|}{-1}        & \multicolumn{1}{c|}{}                                          & 0                                             & \multicolumn{1}{c|}{0}                                            \\ \cline{2-10} 
\multicolumn{1}{c|}{\multirow{2}{*}{$Z_0$}}                     & 0                   & \cdots & \cdots & \multicolumn{1}{c|}{0}                   & \multicolumn{2}{c|}{m_0}                  & \multicolumn{1}{c|}{0}                                         & \multicolumn{2}{c|}{\multirow{2}{*}{$\parfrac{\log Z_0}{\beta}$}} \\
\multicolumn{1}{c|}{}                                          & 0                   & \cdots & \cdots & \multicolumn{1}{c|}{0}                   & \multicolumn{2}{c|}{0}                     & \multicolumn{1}{c|}{0}                                         & \multicolumn{2}{c|}{}                                                                                             \\ \cline{2-10} 
\multicolumn{1}{c|}{\multirow{2}{*}{$Z_1$}}                     & 0                   & \cdots & \cdots & \multicolumn{1}{c|}{0}                   & \multicolumn{2}{c|}{m_1}                  & \multicolumn{1}{c|}{0}                                         & \multicolumn{2}{c|}{\multirow{2}{*}{$\parfrac{\log Z_1}{\beta}$}} \\
\multicolumn{1}{c|}{}                                          & 0                   & \cdots & \cdots & \multicolumn{1}{c|}{0}                   & \multicolumn{2}{c|}{0}                     & \multicolumn{1}{c|}{0}                                         & \multicolumn{2}{c|}{}                                                                                             \\ \cline{2-10} 
\end{array}$$

where $\parfrac{\log Z_0}{\beta}$ and $\parfrac{\log Z_1}{\beta}$ are some matrices in $\mathcal{M}_2(\RR)$. Indeed, we still have the previously computed partial derivative, but we have the new partial derivatives depending on $\beta$ to compute. We have the following computations. Since
$$\begin{array}{rl}
\log Z_1 & =\omega(n_1,n_0)\log(\beta-i)+\omega(n_1,n_0)\log(\beta+i)+\cdots \\
 & =\omega(n_1,n_0)\log(\beta^2+1) + \cdots,\\
 \end{array}$$
which is function of the complex variable $\beta$, we have
$$\boxed{ \parfrac{\log Z_1}{\beta} = \omega(n_1,n_0)\frac{2\beta}{\beta^2+1}\in\CC\subset\mathcal{M}_2(\RR) , }$$
where the scalar complex represents the multiplication by itself from $\CC$ to $\CC$:
$$r+is\in\CC\mapsto \begin{pmatrix}
r & -s \\
s & r \\
\end{pmatrix} \in\mathcal{M}_2(\RR).$$
Now,
$$\log Z_0=\omega(n_0,n_1)\log(i-\beta)+\omega(n_0,n_1)\log(i-\overline{\beta})+\cdots,$$
which depends on $\beta$ but also $\overline{\beta}$. Same as before, the derivative of the first term is $\omega(n_0,n_1)\frac{1}{\beta-i}$, the second being precomposed by the conjugation, whose matrix is $J=\begin{pmatrix}
1 & 0 \\
0 & -1 \\
\end{pmatrix},$ its derivative is
$$\omega(n_0,n_1)\frac{1}{\overline{\beta}-i}\circ J\in\mathcal{M}_2(\RR),$$
where $\circ$ is the product of matrices. The total partial derivative is
$$\parfrac{\log Z_0}{\beta} = \omega(n_0,n_1)\left(  \frac{1}{\beta-i} + \frac{1}{\overline{\beta}-i}\circ J \right) \in\mathcal{M}_2(\RR).$$
We now make our usual development in the determinant. Provided that the determinant consisting of the second rows of $\parfrac{\log Z_0}{\beta}$ and $\parfrac{\log Z_1}{\beta}$ is non-zero, we can make the development with respect the the last and second to last rows. Then, as in all previous cases, we can make a development with respect to the last two rows since $(m_0,m_1)$ is free, and either get the empty determinant, or do the change of basis in the rows corresponding to $N_\gamma$ and reduce to a smaller determinant. Thus, it only suffices to prove the invertibleness of the $2\times 2$ determinant consisting of the second row os both $2\times 2$ matrices. Identifying the second row with $\CC$ by the following map
$$r+is\in\CC\longmapsto (s\ r)\in\mathcal{M}_{1,2}(\RR),$$
the second row of $\parfrac{\log Z_1}{\beta}$ is $\frac{2\beta}{\beta^2+1}$, while the second row of $\parfrac{\log Z_0}{\beta}$ is $\frac{1}{\beta-i}-\frac{1}{\beta+i}$. Indeed,
$$\frac{1}{\overline{\beta}-i}\circ J=J\circ\frac{1}{\beta+i},$$
and multiplication by $J$ on the left just changes the sign of the second row. Thus, we get $\frac{2i}{\beta^2+1}$. Therefore, up to the multiplication by $\frac{2}{\beta^2+1}$, we have to show that the family $(\beta,i)$ is free over $\RR$, which is the case if $\beta$ is not purely imaginary. If $\beta$ was purely imaginary, the moments of the complex ends would be real, and this is excluded by genericity. See section \ref{section statement of result and proof} for more details on why $\beta$ cannot be purely imaginary.
	
	\end{proof}

The previous lemmata prove that the Jacobian is always invertible. Therefore, this finishes the proof of the correspondence theorem.

\end{proof}

\begin{figure}
\begin{center}
\begin{tabular}{cc}
\begin{tikzpicture}[line cap=round,line join=round,>=triangle 45,x=0.5cm,y=0.5cm]
\clip(-1.,-1.) rectangle (7.,8.);
\draw (3.,0.)-- (3.,4.);
\draw (3.,4.)-- (0.,7.);
\draw (3.,4.)-- (6.,7.);
\begin{scriptsize}
\draw[color=black] (4.5,3.744529205178832) node {$V:\ \chi_V$};
\draw[color=black] (4.5,2.0531972812887695) node {$\gamma:\ \alpha_\gamma$};
\draw[color=black] (-0.24508937151643118,7.280950500585326) node {$0$};
\draw[color=black] (1.6838676987052943,5.99497912043751) node {$n_0$};
\draw[color=black] (6.4,7.6) node {$1$};
\draw[color=black] (5.1,5.170280083168802) node {$n_1$};
\end{scriptsize}
\end{tikzpicture}
&
\begin{tikzpicture}[line cap=round,line join=round,>=triangle 45,x=0.5cm,y=0.5cm]
\clip(-1.,-1.) rectangle (7.,8.);
\draw (3.,0.)-- (3.,4.);
\draw (3.,4.)-- (0.,7.);
\draw (3.,3.7)-- (-0.15,6.85);
\draw (3.,4.)-- (6.,7.);
\begin{scriptsize}
\draw[color=black] (4.5,3.744529205178832) node {$V:\ \chi_V$};
\draw[color=black] (4.5,2.0531972812887695) node {$\gamma:\ \alpha_\gamma$};
\draw[color=black] (-0.24508937151643118,7.280950500585326) node {$0$};
\draw[color=black] (1.6838676987052943,5.99497912043751) node {$n_0$};
\draw[color=black] (6.4,7.6) node {$1$};
\draw[color=black] (5.4,5.170280083168802) node {$n_1:\ \beta_1$};
\end{scriptsize}
\end{tikzpicture}
\\
$(a)$ & $(b)$ \\
\end{tabular}
\label{induction vertex}
\caption{Vertices adjacent to two real ends $(a)$ or a real end and two complex ends $(b)$.}
\end{center}
\end{figure}
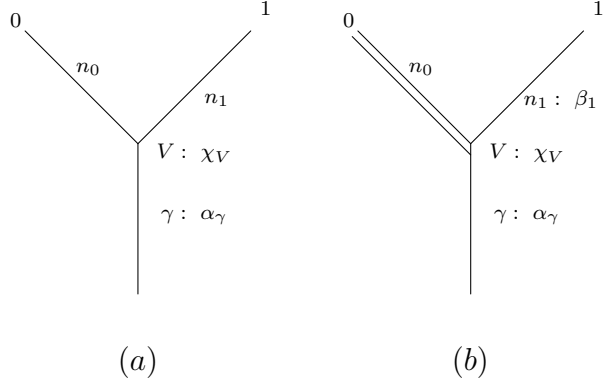

	\section{Proof of Theorem \ref{theorem paper} }
	\label{section statement of result and proof}
	
	Now that we have a correspondence theorem that allows one to lift every first order solution (provided the coordinates $\beta$ are not purely imaginary), one only needs to count the first order solutions for any rational tropical curve solution to the tropical enumerative problem. To do this, we first solve the enumerative problems associated to the vertices of a real rational tropical curve. Then, we use these local resolutions to solve inductively the system $\Theta=(1,\zeta^\Gamma)$ at first order. There are three local enumerative problems: one for trivalent vertices, one for quadrivalent vertices, and one for pentavalent vertices. Each time, we do the refined count of real oriented curves having fixed moments, by previously doing the same count but over the curves with fixed primitive moments, which explains that the computation is rather long.
	
	\subsection{Local enumerative problems}

	\subsubsection{Line problem: trivalent vertex}
	
We consider curves of degree
$$\Delta(n_1,n_2,n_3)=\{ n_1,n_2,n_3 \},$$
where $n_1+n_2+n_3=0$, and to which is assigned to a triangle of lattice area $m_\Delta=|\det(n_1,n_2)|$. A curve of this degree is the image of a line under a monomial map followed by a multiplicative translation. It has a parametrization of the form
$$\varphi:t\longmapsto \chi \times t^{n_1}(t-1)^{n_2}.$$
Let $\CC E_j$ be the toric divisor associated to the vector $n_j$. Let $l_j$ be the integral length of $n_j$. We here recall the resolution of two classic enumerative problems:
\begin{prob}
Let take one point on $\CC E_1$ and one on $\CC E_2$.
\begin{itemize}[label=-]
\item How many complex curves of degree $\Delta(n_1,n_2,n_3)$ pass through  the two fixed points on the toric divisors $\CC E_1$ and $\CC E_3$ ?
\item What if the points are real and the curve asked to be real ?
\end{itemize}
\end{prob}

The first proposition is proved in \cite{mikhalkin2005enumerative}.

\begin{prop}(Mikhalkin \cite{mikhalkin2005enumerative})
Let $p_1\in\CC E_1$ an $p_2\in\CC E_2$ be two complex points. There are $\frac{m_\Delta}{l_1l_2}$ curves of degree $\Delta(n_1,n_2,n_3)$ passing through $p_1$ and $p_2$.
\end{prop}

We now turn our focus to a slight variation on the same problem: we look for curves having a fixed moment. The difference is that the moment is not measured with the primitive vector $\frac{n_j}{l_j}$ but with the non-primitive vector $n_j$. Geometrically, this means that we count curve passing by several pairs of points at the same time.

\begin{prop}
\label{prop trivalent complex vertex}
Let $\mu_1,\mu_2\in\CC^*$. There are $m_\Delta$ curves of degree $\Delta(n_1,n_2,n_3)$ having moments $\mu_1$ and $\mu_2$. Moreover, the curves are equally distributed among the possible intersection points with $\CC E_3$.
\end{prop}

\begin{proof}
The problem amounts to solve
$$\left\{\begin{array}{l}
\chi(m_1)=\mu_1 \\
\chi(m_2)=\mu_2 \\
\end{array}\right. ,$$
Taking the logarithm solves uniquely for the modulus. Taking the argument leads to a system for $\arg\chi:N\rightarrow \RR/2\pi\ZZ$, knowing it on the lattice $\ZZ m_1+\ZZ m_2$, which is a sublattice of index $m_\Delta$. This leads to $m_\Delta=|\det(n_1,n_2)|$ solutions: let $(e_1,e_2)$ be a basis of $N$ such that $(d_1e_1,d_2e_2)$ is a basis of $m_1\ZZ+m_2\ZZ$, with $d_1d_2=m_\Delta$. We have $d_1$ possible values for $\chi(e_1)$ and $d_2$ for $\chi(e_2)$.\\

 The value of $m_3=-(m_1+m_2)$ must satisfy $\chi(m_3)=\frac{1}{\mu_1\mu_2}$. By choosing $e_1=\frac{m_3}{l_3}$, it thus can take $l_3$ values. Thus there are each time $\frac{m_\Delta}{l_3}$ for each value of $\chi\left(\frac{m_3}{l_3}\right)$.
\end{proof}

Finally, when the points are real and the curve is real, we have the following proposition.

\begin{prop}
\label{prop trivalent real vertex}
Let $p_1\in\RR E_1$ and $p_2\in\RR E_2$ be two points with respective coordinates $\mu_1$ and $\mu_2$. There are four real curves passing through $\pm p_1$ an $\pm p_2$. This leads to $8$ oriented real curves, and their refined count is
$$4(q^{m_\Delta/2}-q^{-m_\Delta/2})=4\minus{\frac{m_\Delta}{2}}.$$
\end{prop}

\begin{proof}
We solve the system
$$\left\{\begin{array}{l}
\chi\left(\frac{m_1}{l_1}\right)=\pm\mu_1 \\
\chi\left(\frac{m_2}{l_2}\right)=\pm\mu_2 \\
\end{array}\right. ,$$
for $\chi:N\rightarrow\RR^*$. We solve uniquely for the absolute value, and any choice of sign leads to a solution. Therefore, there are $4$ solutions, leading to $8$ oriented real curves when accounting for both orientations of each of them. their logarithmic rotation number and logarithmic area are computed with Lemma \ref{lemma monomial behavior}.
\end{proof}

	\subsubsection{Parabola problem: quadrivalent vertex}
	
\paragraph{Parametrization and sign}

	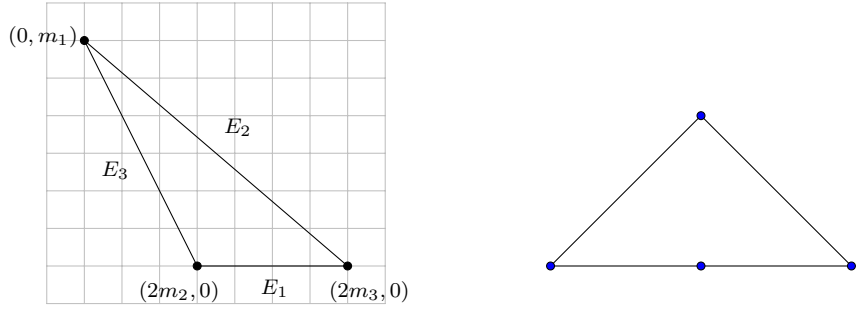
\begin{figure}
\begin{center}

\begin{tabular}{cc}
\definecolor{qqqqff}{rgb}{0.,0.,1.}
\definecolor{xdxdff}{rgb}{0.49019607843137253,0.49019607843137253,1.}
\definecolor{cqcqcq}{rgb}{0.7529411764705882,0.7529411764705882,0.7529411764705882}
\begin{tikzpicture}[line cap=round,line join=round,>=triangle 45,x=0.5cm,y=0.5cm]
\draw [color=cqcqcq,, xstep=0.5cm,ystep=0.5cm] (-1.,-1.) grid (8.,7.);
\clip(-2.,-2.) rectangle (9.,8.);
\draw (0.,6.)-- (3.,0.);
\draw (3.,0.)-- (7.,0.);
\draw (7.,0.)-- (0.,6.);
\begin{scriptsize}
\draw [fill=black] (0.,6.) circle (1.5pt);
\draw[color=black] (-1.1,6.1) node {$(0,m_1)$};
\draw [fill=black] (3.,0.) circle (1.5pt);
\draw[color=black] (2.52,-0.68) node {$(2m_2,0)$};
\draw[color=black] (0.82,2.54) node {$E_3$};
\draw [fill=black] (7.,0.) circle (1.5pt);
\draw[color=black] (7.58,-0.68) node {$(2m_3,0)$};
\draw[color=black] (5.06,-0.6) node {$E_1$};
\draw[color=black] (4.08,3.68) node {$E_2$};
\end{scriptsize}
\end{tikzpicture}

&

\definecolor{qqqqff}{rgb}{0.,0.,1.}
\definecolor{xdxdff}{rgb}{0.49019607843137253,0.49019607843137253,1.}
\begin{tikzpicture}[line cap=round,line join=round,>=triangle 45,x=2.0cm,y=2.0cm]
\clip(8.5,-0.5) rectangle (11.5,1.5);
\draw (9.,0.)-- (10.,1.);
\draw (10.,1.)-- (11.,0.);
\draw (11.,0.)-- (10.,0.);
\draw (10.,0.)-- (9.,0.);
\begin{scriptsize}
\draw [fill=qqqqff] (9.,0.) circle (1.5pt);
\draw [fill=qqqqff] (10.,1.) circle (1.5pt);
\draw [fill=qqqqff] (11.,0.) circle (1.5pt);
\draw [fill=qqqqff] (10.,0.) circle (1.5pt);
\end{scriptsize}
\end{tikzpicture}
\\
\end{tabular}

\caption{\label{figure polygon parabola} The Polygon $P_\Delta^\mathrm{par}(m_1,m_2,m_3)$ and the parabola triangle used for its parametrization}
\end{center}
\end{figure}
	
	We consider a real oriented curve of degree
$$\dpar=\{ (0,m_3-m_2)^2,(m_1,2m_2),(-m_1,-2m_3)\},$$
which is associated to a triangle $P_\Delta^\mathrm{par}(m_1,m_2,m_3)$ of lattice area $m_\Delta$ depicted on Figure \ref{figure polygon parabola}. The sides of the triangle are denoted $E_1$, $E_2$ and $E_3$. The associated toric divisors in the associated toric surface $\CC\Delta(m_1,m_2,m_3)$ are denoted by $\CC E_j$. We denote by $l_j$ the integral length of the corresponding vectors in $\dpar$. Thus, we have $l_1=m_3-m_2$ and $l_3=\mathrm{gcd}(m_1,2m_2)$. A curve having such a degree has a parametrization of the form
$$t\longmapsto\left( a(t-c)^{m_1},b(t-c)^{2m_2}(t^2+1)^{m_3-m_2}\right),$$
where $a,b\in\RR^*$ and $c\in\RR$ are real parameters. The coordinate of parametrization is chosen such that the intersection points with $\CC E_1$ have coordinate $\pm i$ and the intersection point with $\CC E_2$ has coordinate $\infty$. There are two such coordinates. The choice of $i$ determines an orientation of the curve and fixes the coordinate. Such a curve is the image of a curve of degree $\Delta^\mathrm{par}=\{(1,1),(-1,1),(0,-1)^2\}$ having parametrization
$$t\longmapsto\left( t-c ,\frac{t^2+1}{t-c}\right),$$
by a monomial map followed by a multiplicative translation. The matrix of the monomial map between the lattices of co-characters and characters are respectively
$$A=\begin{pmatrix}
m_1 & 0 \\
m_2+m_3 & m_3-m_2 \\
\end{pmatrix}:N_0\rightarrow N \text{ and }A^T=\begin{pmatrix}
m_1 & m_2+m_3 \\
0 & m_3-m_2 \\
\end{pmatrix}:M\rightarrow M_0.$$

\begin{prop}
The curves having parametrization
$$t\longmapsto\left( a(t-c)^{m_1},b(t-c)^{2m_2}(t^2+1)^{m_3-m_2}\right),$$
have logarithmic rotation index $0$ and log-area
$$\pi\frac{m_\Delta}{2}(2\mathrm{arccot}(-c)-\pi).$$
\end{prop}

\begin{proof}
The log-area is computed using Lemma \ref{lemma monomial behavior} and the computation of the log-area of a parabola. As the monomial map acts as the linear map $A$ on the logarithmic side, the logarithmic rotation number of the image curve is equal to the one of the curve in the domain multiplied by the sign of the determinant of the monomial map. But as a parabola has logarithmic rotation index $0$, it stays $0$.
\end{proof}

\paragraph{Enumerative problems}	

We consider several similar enumerative problems dealing with parabolas. First, we look for real oriented curves of degree $\dpar$ having fixed intersection points with the toric boundary of $\CC\dpar$. It means their primitive moments are fixed. (see Definition \ref{definition primitive moment}) Let $\mu_2\in\RR^*$ be the coordinate of a point on $\CC E_3$, and $\mu_1 e^{\pm i\pi\theta}\in\CC^*$, with $\mu_1>0$ and $0<\theta< 1$ be the coordinate of a pair of complex conjugated points on $\CC E_1$. The coordinate on the toric divisor $\CC E_1$ is the monomial $z$, and the coordinate on $\CC E_3$ is the monomial $z^{2m_2/l_3}w^{-m_1/l_3}$. These are in fact primitive lattice vectors directing the side of the triangle $P_\Delta^\mathrm{par}(m_1,m_2,m_3)$ from Figure \ref{figure polygon parabola}.

\begin{prob}
\label{problem parabola prim}
How many oriented real rational curves of degree $\dpar$ pass through the pair of conjugated points on $\CC E_1$ and one of the opposite real points on $\CC E_3$ ? Their count is refined by their log-area and counted with sign.
\end{prob}

The enumerative problem amounts to solving the following system:
$$\left\{ \begin{array}{l}
a(i-c)^{m_1}= \mu_1 e^{\pm i\pi\theta} \\
a^{\frac{2m_2}{l_3}}b^{-\frac{m_1}{l_3}}(c^2+1)^{-\frac{m_1}{l_3}(m_3-m_2)}=\pm\mu_2 .\\
\end{array}\right.$$
Recall $l_1=m_3-m_2$ is the integral length of both vectors associated to the side $E_1$. The height relative to the corresponding side is $m_1$, now rather denoted by $h_1$, so that $h_1\times 2l_1=m_\Delta$, which is the multiplicity of the dual vertex. We have two systems, according to whether $i$ is sent on $\mu_1 e^{i\pi\theta}$ or its conjugate $\mu_1 e^{-i\pi\theta}$. Each system solves itself as follows: first take the argument mod $\pi$ of the first equation. This determines the possible values of $c$. Then, the first equation solves for $a$, and the second for $b$. Doing so, we get the following proposition.

\begin{prop}
\label{prop parabola problem}
For $0<\theta<1$, the refined count of real oriented curves solution to the enumerative problem \ref{problem parabola prim} is:
$$2(q^{l_1(2\theta-1)}+q^{-l_1(2\theta-1)})\frac{q^{\frac{m_\Delta}{2}} - q^{-\frac{m_\Delta}{2}} }{q^{l_1}-q^{-l_1}} = 2\plus{l_1(2\theta-1)} \frac{ \minus{\frac{m_\Delta}{2} } }{ \minus{l_1} } .$$
\end{prop}

\begin{proof}
One needs to carry out the solving explained before the statement of Proposition \ref{prop parabola problem}.
\begin{itemize}[label=-]
\item The system with $e^{i\pi\theta}$ solves itself as follows. The first equation solves for $c$ by taking the argument of the equation mod $\pi$. We take mod $\pi$ rather than $2\pi$ since the sign of $a$ is unknown. It then solves uniquely for $a$, including its sign. Last, the second equation solves for $b$ with a unique solution if $\frac{m_1}{l_3}$ is odd, and $0$ or $2$ solutions according to the sign of $\mu_2$ if it is even. Let us carry this computation. The argument mod $\pi$ of the first equation gives us
$$h_1\arg(i-c)\equiv \pi\theta\ [\pi].$$
This is equivalent to
$$\mathrm{arccot}(-c)\equiv \frac{\theta+k}{h_1}\pi \ [\pi],\ \text{ with }0\leqslant k<h_1$$
As $\mathrm{arccot}$ has image $]0;\pi[$, for each $k$ we get a unique solution $c_k=-\cot\left(\frac{\theta+k}{h_1}\pi\right)$ provided that $0<\theta+k<h_1$. Notice that in total, we get $2h_1$ oriented curves for the solving with $e^{i\pi\theta}$. The log-area of the oriented curve is given by $\frac{m_\Delta}{2}(2\pi\mathrm{arccot}(-c)-\pi^2)$, which is equal to
$$ l_1(2(\theta+k)-h_1)\pi^2, \text{ for }0\leqslant k<h_1. $$
The count of real oriented curves passing through the fixed points and refined by the value of the log-area is thus
$$2\sum_{k=0}^{h_1-1} q^{l_1(2\theta+2k-h_1)} = 2q^{l_1(2\theta-1)}\frac{q^{\frac{m_\Delta}{2}} - q^{-\frac{m_\Delta}{2}} }{q^{l_1}-q^{-l_1}}=2q^{l_1(2\theta-1)}\frac{ \minus{\frac{m_\Delta}{2}} }{ \minus{l_1} }.$$

\item For the system with $e^{-i\pi\theta}$, the solving is identical, by replacing $\theta$ with $-\theta$. Except a change in the possible values  for $k$, since we must still have $0< -\theta+k<h_1$. As a matter of fact, this is equivalent to replacing $\theta$ by $1-\theta$ in the previous formula. Adding both contribution, and considering both signs of $\mu_1$, we get the desired refined count.
\end{itemize}
\end{proof}

We can now slightly modify the problem by also allowing to pass through the opposite pair of complex points.

\begin{prob}
\label{problem parabola prim opposite}
How many real rational curves of degree $\dpar$ pass through one of the opposite pairs of complex points on $\CC E_1$, and one of the opposite real points on $\CC E_2$ ?
\end{prob}

\begin{coro}
For $0<\theta<1$, the refined count of real oriented curves passing through the symmetric configuration is equal to
$$4\plus{l_1(2\theta-1)}\frac{ \minus{\frac{m_\Delta}{2} } }{ \minus{l_1} } .$$
\end{coro}

\begin{proof}
Just add the counts for the values $\theta$ and $1-\theta$, which happen to be equal.
\end{proof}

Finally, we solve the final problem, where we impose the moment of the curve rather than its primitive moment, which would mean imposing the position of its intersection point.

\begin{prob}
\label{problem parabola full}
Let $\mu_1=e^{\pm i\pi\theta}\in\CC^*$ and $\mu_2\in\RR^*$. What is the refined count of real oriented curves of degree $\dpar$ having moments $\pm\mu_1$ on $\CC E_1$ and $\pm\mu_2$ on $\CC E_2$ ?
\end{prob}

The solving is a direct corollary of the solving of the previous problem by taking the $l_1$th-root of the first equation.

\begin{coro}
The refined number of oriented real rational curves of degree $\dpar$ with moments $\pm e^{\pm i\pi\theta}\in\CC^*$ on $\CC E_1$ and $\pm\mu_2\in\RR^*$ on $\CC E_2$ is
$$4\plus{2\theta-1}\frac{ \minus{\frac{m_\Delta}{2} } }{ \minus{1} } .$$
\end{coro}

\begin{proof}
Let $l=l_1$. The moment on $\CC E_1$ is known. Therefore, the primitive moment is its $l$th-root. So it can take the following values:
$$\left\{ \begin{array}{l}
\pm\frac{\theta+k}{l}\pi \text{ for }0\leqslant k<l, \\
\pm\frac{1-\theta+k}{l}\pi \text{ for }0\leqslant k<l. \\
\end{array}\right.$$
We have $4l$ of these values: $l$ roots for each of the four arguments $\pm\theta\pi$ and $(1\pm\theta)\pi$. Half of them belong to $]0;\pi[$ and they are obtained by taking $+$ signs in front of the fractions. These values can be grouped into packs of $4$: if $\varphi\pi$ is a possible primitive moment, so are $-\varphi\pi, (1-\varphi)\pi$ and $(1+\varphi)\pi$. In the choices of parametrization of the values by $k$, this association is as follows:
$$1-\frac{\theta+k}{l}=\frac{1-\theta+l-1-k}{l},$$
thus, $\pm\frac{\theta+k}{l}$ corresponds to $\pm\frac{1-\theta+k'}{l}$, where $k+k'=l-1$. So we can take $\frac{\theta+k}{l}$ to be a positive representative of each group of four. The refined count is known to be
$$4\plus{l\left(2\frac{\theta+k}{l}-1\right)}\frac{ \minus{\frac{m_\Delta}{2} } }{ \minus{l} } .$$
To conclude, we just need to add these contributions, and notice that:
\begin{align*}
\sum_0^{l-1} \plus{2\theta+2k-l} & = q^{2\theta-l}\sum_0^{l-1} q^{2k} + q^{l-2\theta}\sum_0^{l-1} q^{-2k} \\
	&=q^{2\theta-l}\frac{q^{2l}-1}{q^2-1}+q^{l-2\theta}\frac{1-q^{-2l}}{1-q^{-2}} \\
	& = (q^{2\theta-1}+q^{1-2\theta})\frac{q^l-q^{-l}}{q-q^{-1}}\\
	& = \plus{2\theta-1}\frac{\minus{l}}{\minus{1}}.\\
\end{align*}
\end{proof}

\begin{rem}
Notice that if the vectors are primitive, \textit{i.e.} $l_1=1$, by making $\theta\rightarrow\frac{1}{2}$, and divide by $4$, we get the value of the refined invariant announced in \ref{theorem paper}:
$$R_{\dpar} = 2\frac{q^{\frac{m_\Delta}{2}} - q^{-\frac{m_\Delta}{2}} }{q-q^{-1}} = 2\frac{\minus{\frac{1}{2}}}{\minus{1}}N^{\partial,\mathrm{trop}}_{\dpar}.$$
\end{rem}

	\subsubsection{Tangent ellipse problem: pentavalent vertex}
	
\paragraph{Parametrization and sign}

	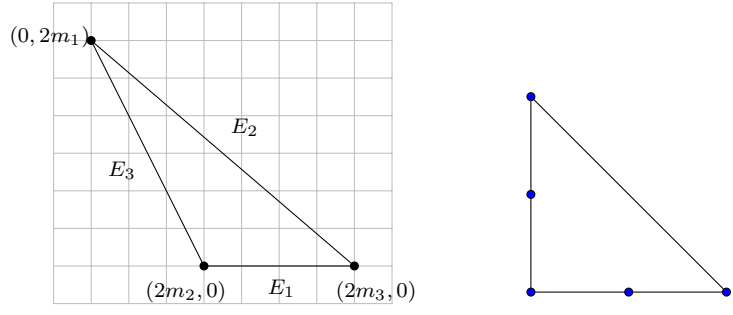
\begin{figure}
\begin{center}

\begin{tabular}{cc}

\definecolor{qqqqff}{rgb}{0.,0.,1.}
\definecolor{xdxdff}{rgb}{0.49019607843137253,0.49019607843137253,1.}
\definecolor{cqcqcq}{rgb}{0.7529411764705882,0.7529411764705882,0.7529411764705882}
\begin{tikzpicture}[line cap=round,line join=round,>=triangle 45,x=0.5cm,y=0.5cm]
\draw [color=cqcqcq,, xstep=0.5cm,ystep=0.5cm] (-1.,-1.) grid (8.,7.);
\clip(-2.2,-2.) rectangle (9.,8.);
\draw (0.,6.)-- (3.,0.);
\draw (3.,0.)-- (7.,0.);
\draw (7.,0.)-- (0.,6.);
\begin{scriptsize}
\draw [fill=black] (0.,6.) circle (1.5pt);
\draw[color=black] (-1.1,6.1) node {$(0,2m_1)$};
\draw [fill=black] (3.,0.) circle (1.5pt);
\draw[color=black] (2.52,-0.68) node {$(2m_2,0)$};
\draw[color=black] (0.82,2.54) node {$E_3$};
\draw [fill=black] (7.,0.) circle (1.5pt);
\draw[color=black] (7.58,-0.68) node {$(2m_3,0)$};
\draw[color=black] (5.06,-0.6) node {$E_1$};
\draw[color=black] (4.08,3.68) node {$E_2$};
\end{scriptsize}
\end{tikzpicture}

&

\definecolor{qqqqff}{rgb}{0.,0.,1.}
\definecolor{xdxdff}{rgb}{0.49019607843137253,0.49019607843137253,1.}
\begin{tikzpicture}[line cap=round,line join=round,>=triangle 45,x=1.3cm,y=1.3cm]
\clip(-0.5,-0.5) rectangle (2.5,2.5);
\draw (0.,0.)-- (0.,1.);
\draw (0.,1.)-- (0.,2.);
\draw (0.,0.)-- (1.,0.);
\draw (1.,0.)-- (2.,0.);
\draw (2.,0.)-- (0.,2.);
\begin{scriptsize}
\draw [fill=qqqqff] (0.,0.) circle (1.5pt);
\draw [fill=qqqqff] (1.,0.) circle (1.5pt);
\draw [fill=qqqqff] (2.,0.) circle (1.5pt);
\draw [fill=qqqqff] (0.,1.) circle (1.5pt);
\draw [fill=qqqqff] (0.,2.) circle (1.5pt);
\end{scriptsize}
\end{tikzpicture}
\\
\end{tabular}

\caption{\label{figure polygon conic} The Polygon $P_\Delta^\mathrm{con}(m_1,m_2,m_3)$ and the conic triangle used for its parametrization}
\end{center}
\end{figure}
	
	We consider real oriented curves of degree
$$\dcon=\{ (0,m_3-m_2)^2,(m_1,m_2)^2,(-2m_1,-2m_3)\},$$
which is dual to the triangle $P_\Delta^\mathrm{con}(m_1,m_2,m_3)$ depicted on Figure \ref{figure polygon conic}. We label the sides $E_1$, $E_2$ and $E_3$, and name the toric divisors in the toric surface accordingly. A curve having such a degree has a parametrization of the form
$$t\longmapsto\left( a(t^2-2t\cos (x_3\pi)+1)^{m_1},b(t^2-2t\cos(x_3\pi)+1)^{m_2}(t^2-2t\cos(x_1\pi)+1)^{m_3-m_2}\right),$$
where $a,b\in\RR^*$ are real parameters, and $0<x_1,x_3<1$ are the normalized arguments of the parameters for the complex intersection points. The coordinate is chosen as in subsubsection \ref{subsubsection log-area of a tangent ellipse}: so that the intersection point with $\CC E_2$ has coordinate $\infty$, and the intersection points with $\CC E_1$ have coordinate of modulus $1$. The choice of the orientation fixes one among the two coordinates satisfying this assumption. In this setting, $e^{i\pi x_1}$ corresponds to the coordinate of the intersection points with $\CC E_1$ and similarly for $\CC E_3$. The normalized arguments $x_1$ and $x_3$ cannot be equal, otherwise the curve is a double cover of a curve of smaller degree $\frac{1}{2}\dcon$.\\

A curve of degree $\dcon$ is the image of a conic in the projective plane parametrized as follows
$$t\longmapsto \left(t^2-2t\cos(x_3\pi)+1,t^2-2t\cos(x_1\pi)+1)\right),$$
under a monomial map followed by a translation. The monomial map between the lattices of co-characters and characters have respective matrices
$$A=\begin{pmatrix}
m_1 & 0 \\
m_2 & m_3-m_2 \\
\end{pmatrix}:N_0\rightarrow N, \text{ and } A^T=\begin{pmatrix}
m_1 & m_2 \\
0 & m_3-m_2 \\
\end{pmatrix}:M\rightarrow M_0.$$
The monomial map is a ramified cover of degree $m_1(m_3-m_2)=\frac{m_\Delta}{4}$, which is ramified over the toric boundary. The degree of the covering map restricted to a $\CC E_j$ divisor is equal to the total degree divided by the integral length of the corresponding size in the associated triangle. Let $l_j$ be the integral length of the vectors in $\dcon$ associated to the divisor, and $h_j$ its relative height, so that $h_j l_j=m_1(m_3-m_2)$. In our case, $l_3=\mathrm{gcd}(m_1,m_2)$, and $l_1=m_3-m_2$.\\

The coordinate on $\CC E_j$ is one of the two primitive monomials directing the side $E_j$. For instance, the coordinate on $\CC E_3$ is $z^{m_2/l_3}w^{-m_1/l_3}$. Its pull-back by the monomial map is $w^{-\frac{m_1(m_3-m_2)}{l_3}}=w^{-h_3}$. The degree of the cover restricted to the toric divisor $\CC E_j$ is $h_j$.\\

Just as in the parabola case, the logarithmic rotation number of a curve and its image by a monomial map are equal up to sign. Then, the following proposition computes the logarithmic rotation number of curves of degree $\dcon$.

\begin{lem}
The logarithmic rotation number of
$$t\longmapsto \left(t^2-2t\cos(x_3\pi)+1,t^2-2t\cos(x_1\pi)+1)\right),$$
 is $\pm 1$, equal to the sign of $x_3-x_1$.
\end{lem}

\begin{proof}
The logarithmic map thus takes the following form:
$$t\longmapsto \left( \log(t^2-2t\cos(x_3\pi) +1),\log(t^2-2t\cos(x_1\pi) +1) \right).$$
We can now express the logarithmic Gauss map
$$\gamma:t\longmapsto \left[ \frac{t-\cos(x_3\pi)}{t^2-2t\cos(x_3\pi) +1} : \frac{t-\cos(x_1\pi)}{t^2-2t\cos(x_1\pi) +1} \right] \in\RR P^1.$$
To compute its degree, we compute the signed number of antecedents over a fixed direction in $\RR P^1$. We choose the direction to be $[0:1]$. The unique solution is $t=\cos(x_3\pi)$. We need to know if the signed contribution is $+1$ or $-1$. To do this, we have to notice the orientation of $\RR P^1$ given by the orientation of $\RR^2$ is also given by the canonical orientation of $\RR$ via the coordinate $-\frac{u}{v}$, where $[u:v]\in\RR P^1$. Therefore, the sign is given by the sign of the second coordinate $\frac{\cos(x_3\pi)-\cos(x_1\pi)}{\cos^2(x_3\pi)-2\cos(x_3\pi)\cos(x_1\pi) +1}$. Finally, the logarithmic rotation number is $+1$ if $\cos(x_3\pi)<\cos(x_1\pi)$ and $-1$ otherwise.
\end{proof}

\paragraph{Enumerative problems}

We now consider the same enumerative problems as \ref{problem parabola prim}, \ref{problem parabola prim opposite} and \ref{problem parabola full} but for the degree $\dcon$. First, we look for real oriented curves of degree $\dcon$ having fixed intersection with the toric boundary of $\CC\dcon$. Up to the real torus action, we can assume that the coordinates of the chosen intersection points are of modulus $1$, given by $e^{\pm i\theta_1\pi}$ on $\CC E_1$ and $e^{\pm i\theta_3\pi}$ on $\CC E_3$. We choose $0<\theta_j< 1$.

\begin{prob}
\label{problem ellipse prim}
How many oriented real rational curves of degree $\dcon$ and passing through the fixed pairs of complex conjugated points on $\CC E_1$ and $\CC E_3$ ?
\end{prob}

Using the given parametrization of such a curve, the problem amounts to the following system:
$$\left\{ \begin{array}{l}
a(\cos(x_1\pi)-\cos(x_3\pi))^{h_1}e^{i\pi h_1 x_1}=e^{\pm i\pi\theta_1}, \\
a^{-\frac{m_2}{l_3}}b^\frac{m_1}{l_3}(\cos(x_3\pi)-\cos(x_1\pi))^{h_3} e^{i\pi h_3 x_3}=e^{\pm i\pi\theta_3}.\\
\end{array}\right. $$
The two equations are obtained by evaluating the primitive monomials $z$ and $z^{-\frac{m_2}{l_3}}w^\frac{m_1}{l_3}$ respectively at $e^{i\pi x_1}$ and $e^{i\pi x_3}$ to obtain the coordinates of the intersection points with the toric boundary. Notice that $h_1=m_1$ and $h_3=\frac{m_1(m_3-m_2)}{l_3}$ are the degree of the monomial ramified cover restricted to the toric divisors.\\

The choice of the signs in front of the $\theta_j$ corresponds to the image of $e^{i\pi x_1}$ and $e^{i\pi x_3}$ be either point in each of both pairs of complex conjugated points. In fact, counting real oriented curves is equivalent to counting real holomorphic disks with boundary in the real part of the complex manifold. The disk only passes through one point of each pair of complex conjugated points. \\ 

The system may be solved as follows: take the argument mod $\pi$ and find $x_1$ and $x_3$, then take the modulus to find the modulus of $a$ and $b$, and the argument mod $2\pi$ to find their sign. The first equation fixes the sign of $a$, but the second equation might fail to fix the sign of $b$ if the exponent $\frac{m_1}{l_3}$ is even.

\begin{prop}
\label{prop ellipse problem}
For $0<\theta_1,\theta_3<1$, the refined contribution is:
\begin{itemize}
\item If $\frac{m_1}{l_3}$ is odd:
$$\sum_{ \substack{ 0\leqslant k_1 <h_1 \\ 0\leqslant k_3 < h_3 }} \sigma_{k_1,k_3} q^{2l_3\left( (1-)\theta_3 + k_3 \right) - 2l_1\left( (1-)\theta_1 + k_1 \right)},$$
where the sum has $4h_1h_3$ terms accounting for the values $(1-)\theta$, taken equal to $\theta$ and $1-\theta$. The sign $\sigma_{k_1,k_3}$ is the sign of the exponent.
\item If $\frac{m_1}{l_3}$ is even:
$$2\sum_{ \substack{ 0\leqslant k_1 <h_1 \\ 0\leqslant k_3 < h_3 \\ k_1+k_3\equiv\epsilon [2] }} \sigma_{k_1,k_3} q^{2l_3\left( (1-)\theta_3 + k_3 \right) - 2l_1\left( (1-)\theta_1 + k_1 \right)},$$
where the sum has $2h_1h_3$ terms accounting for the values $(1-)\theta$ and the congruence condition for $k_1,k_3$: $\epsilon=0$ for $(\theta_1,\theta_3)$ and $(1-\theta_1,1-\theta_3)$, and $\epsilon=1$ for $(1-\theta_1,\theta_3)$ and $(\theta_1,1-\theta_3)$. The sign $\sigma_{k_1,k_3}$ is still the sign of the exponent.
\end{itemize}
\end{prop}

\begin{rem}
The computation of these sum depends in a complicated way of the value of $(l_1,l_3)$. It can be computed in the case $l_1=l_3=1$ and this is in fact the only case that we need. This is done later in this subsection.
\end{rem}

\begin{proof}
Let us carry out the resolution as proposed before the proposition.
\begin{itemize}[label=-]
\item First, take the image by the logarithm of modulus to get the following system:
$$\left\{ \begin{array}{l}
\log|a|+h_1\log|\cos(x_3\pi)-\cos(x_1\pi)|=0, \\
-\frac{m_2}{l_3}\log|a|+\frac{m_1}{l_3}\log|b|+h_3\log|\cos(x_1\pi)-\cos(x_1\pi)| =0.\\
\end{array}\right. $$
This allows to express the modulus of $a$ and $b$ in term of $|\cos(x_3\pi)-\cos(x_1\pi)|$, which is known if $x_1,x_3$ are known.
\item To find $x_1,x_3$ and the signs of $a$ and $b$, we take the argument mod $\pi$ of both equations:
$$\left\{ \begin{array}{l}
h_1 x_1\pi\equiv \pm\theta_1\pi \ [\pi], \\
h_3 x_3\pi \equiv \pm\theta_3\pi \ [\pi]. \\
\end{array}\right. $$
Therefore, we get
$$\left\{ \begin{array}{l}
x_1=\frac{(1-)\theta_1+k_1}{h_1}, \ 0\leqslant k_1 <h_1, \\
x_3=\frac{(1-)\theta_3+k_1}{h_3}, \ 0\leqslant k_3 <h_3.  \\
\end{array}\right. $$
The $(1-)\theta$ denotes the value $\theta$ or $1-\theta$. This convention replaces the $\pm\theta$ to get the same interval for $k_1$ in both cases. This solves for $x_1,x_3$. We get $4h_1h_3$ solutions. Then, going back to the principal system, we find the sign of $a$ in the first equation, and the sign of $b$ in the second equation according to the parity of $\frac{m_1}{l_3}$. The solving for $b$ gives $1$ solution if $\frac{m_1}{l_3}$ is odd, or $0$ or $2$ solution according to sign if it is even. In each case, we get $4h_1h_3$ solutions, for which we need to compute the contribution.
\begin{itemize}[label=$\bullet$]
\item Assume $\frac{m_1}{l_3}$ is odd. Then, there is a unique solution for the signs of $a$ and $b$ and this for each value of $(k_1,k_3)$. Moreover, the log-area is equal to 
\begin{align*}
\mathcal{A} &= \det A \times 2\pi(x_3\pi-x_1\pi) \\
	& = 2 h_j l_j \left( \frac{(1-)\theta_3 + k_3}{h_3}-\frac{(1-)\theta_1+k_1}{h_1} \right)\pi^2 \\
	& = 2l_3\left( (1-)\theta_3 + k_3 \right)\pi^2 - 2l_1\left( (1-)\theta_1 + k_1 \right)\pi^2,\\
\end{align*}
and the logarithmic rotation number is equal to its sign. Therefore, the total refined count is equal to
$$\sum_{ \substack{ 0\leqslant k_1 <h_1 \\ 0\leqslant k_3 < h_3 }} \sigma_{k_1,k_3} q^{2l_3\left( (1-)\theta_3 + k_3 \right) - 2l_1\left( (1-)\theta_1 + k_1 \right)},$$
where the sum has $4h_1h_3$ terms accounting for the values $(1-)\theta$, and the sign $\sigma_{k_1,k_3}$ is the sign of the exponent.
\item Now, assume that $\frac{m_1}{l_3}$ is even. Then $h_1=m_1$ is also even, as well as $h_3=\frac{m_1}{l_3}(m_3-m_2)$. However, as $l_3=\mathrm{gcd}(m_1,m_2)$, $\frac{m_2}{l_3}$ is odd. Taking care of the parity of these integers and taking $\arg(a)\equiv \varepsilon\pi\ [2]$, the argument mod $2\pi$ of the principal system gives us
$$\left\{ \begin{array}{l}
\varepsilon\ +\ (1-)\theta_1+k_1\ \equiv \ \pm\theta_1 \ [2],\\
\varepsilon\ +\ (1-)\theta_3+k_3\ \equiv \ \pm\theta_3 \ [2].\\
\end{array}\right.$$
Both equation give the value of $\varepsilon$, and these have to coincide. This is the case if and only if
$$\begin{array}{l}
k_1+k_3 \equiv 0 \ [2] \text{ for the pairs }(\theta_1,\theta_3) \text{ and }(-\theta_1,-\theta_3),\\
k_1+k_3 \equiv 1 \ [2] \text{ for the pairs }(-\theta_1,\theta_3) \text{ and }(\theta_1,-\theta_3).\\
\end{array}$$
The total signed contribution is then
$$2\sum_{ \substack{ 0\leqslant k_1 <h_1 \\ 0\leqslant k_3 < h_3 \\ k_1+k_3\equiv\bullet [2] }} \sigma_{k_1,k_3} q^{2l_3\left( (1-)\theta_3 + k_3 \right) - 2l_1\left( (1-)\theta_1 + k_1 \right)},$$
where the sum has $2h_1h_3$ terms accounting for the values $(1-)\theta$ and the congruence condition for $k_1,k_3$. Once again, the sign $\sigma_{k_1,k_3}$ is the sign of the exponent.
\end{itemize}
\end{itemize}

\end{proof}

Now, we get to the first variant of the enumerative problem and allow the points the complex points to take the opposite value.

\begin{prob}
How many oriented real rational curves of degree $\dcon$ have primitive moments $\pm e^{i\pi\theta_1}$ on $\CC E_1$ and $\pm e^{i\pi\theta_3}$ on $\CC E_3$ ?
\end{prob}

As in the parabola case, the solving consists in adding the contribution for the different possibilities.

\begin{coro}
The refined number of oriented real rational curves of degree $\dcon$ having primitive moments $\pm e^{i\pi\theta_1}$ on $\CC E_1$ and $\pm e^{i\pi\theta_3}$ on $\CC E_3$ is
$$4\sum_{ \substack{ 0\leqslant k_1 <h_1 \\ 0\leqslant k_3 < h_3 }} \sigma_{k_1,k_3} q^{2l_3\left( (1-)\theta_3 + k_3 \right) - 2l_1\left( (1-)\theta_1 + k_1 \right)}  .$$
\end{coro}
	
\begin{proof}
One just needs to replace consider the couples $(1-\theta_1,\theta_3)$, $(\theta_1,1-\theta_3)$ and $(1-\theta_1,1-\theta_3)$ in place of $(\theta_1,\theta_3)$ in Proposition \ref{prop ellipse problem}.
\begin{itemize}
\item If $\frac{m_1}{l_3}$ is odd, the contribution is the same for each of the couples. Thus, we get the announced contribution.
\item If $\frac{m_1}{l_3}$ is even, let us write carefully the contributions. The contribution for $(\theta_1,\theta_3)$ given by Proposition \ref{prop ellipse problem} is
\begin{align*}
 & 2\sum_{ \substack{ 0\leqslant k_1 <h_1 \\ 0\leqslant k_3 < h_3 \\ k_1+k_3\equiv0 [2] }} \sigma q^{2l_3\left( \theta_3 + k_3 \right) - 2l_1\left( \theta_1 + k_1 \right)} + \sigma q^{2l_3\left( 1-\theta_3 + k_3 \right) - 2l_1\left( 1-\theta_1 + k_1 \right)} \\
+ & 
2\sum_{ \substack{ 0\leqslant k_1 <h_1 \\ 0\leqslant k_3 < h_3 \\ k_1+k_3\equiv 1 [2] }} \sigma q^{2l_3\left( 1- \theta_3 + k_3 \right) - 2l_1\left( \theta_1 + k_1 \right)} + \sigma q^{2l_3\left( \theta_3 + k_3 \right) - 2l_1\left( 1-\theta_1 + k_1 \right)}. \\
\end{align*}
The contribution is the same when replacing $(\theta_1,\theta_3)$ by $(1-\theta_1,1-\theta_3)$. However, when replacing by $(\theta_1,1-\theta_3)$ and $(1-\theta_1,\theta_3)$, the congruences on $k_1,k_3$ are exchanged. Then, adding all contribution, we get
\begin{align*}
 & 4\sum_{ \substack{ 0\leqslant k_1 <h_1 \\ 0\leqslant k_3 < h_3  }} \sigma q^{2l_3\left( \theta_3 + k_3 \right) - 2l_1\left( \theta_1 + k_1 \right)} + \sigma q^{2l_3\left( 1-\theta_3 + k_3 \right) - 2l_1\left( 1-\theta_1 + k_1 \right)} \\
+ & 
4\sum_{ \substack{ 0\leqslant k_1 <h_1 \\ 0\leqslant k_3 < h_3 }} \sigma q^{2l_3\left( 1- \theta_3 + k_3 \right) - 2l_1\left( \theta_1 + k_1 \right)} + \sigma q^{2l_3\left( \theta_3 + k_3 \right) - 2l_1\left( 1-\theta_1 + k_1 \right)}, \\
\end{align*}
which is the announced result.
\end{itemize}
\end{proof}
	
Finally, we can turn our focus on the last problem, which is to make the refined count of curves having fixed moments rather than fixed primitive moments.

\begin{prob}
Let $e^{i\pi\theta_1},e^{i\pi\theta_3}\in\CC^*$. How many oriented real rational curves of degree $\dcon$ have moments $\pm e^{\pm i\pi\theta_1}$ on $\CC E_1$ and $\pm e^{\pm i\pi\theta_1}$ on $\CC E_3$ ?
\end{prob}

The proof goes by the same trick as in the parabola case for solving problem \ref{problem parabola full}: we extract $l_1$ and $l_3$ roots to find the possible values of the respective primitive moments.

\begin{coro}
\label{coro ellipse problem}
Let $e^{i\pi\theta_1},e^{i\pi\theta_3}\in\CC^*$. The refined number of oriented real rational curves of degree $\dcon$ have moments $\pm e^{\pm i\pi\theta_1}$ on $\CC E_1$ and $\pm e^{\pm i\pi\theta_1}$ on $\CC E_3$ is
$$ \plus{2\theta_3-1}\plus{2\theta_1-1}\frac{\minus{\frac{m_\Delta}{2}}}{\minus{1}\minus{1}}-\frac{m_\Delta}{2\minus{1}} \left\{ \begin{array}{ll}
\plus{2\theta_3-1}\plus{2\theta_1} &  \text{ if }\theta_1<\theta_3,\ \theta_1+\theta_3<1,\\
\plus{2\theta_3-1}\plus{2(1-\theta_1)} & \text{ if }\theta_1>\theta_3,\ \theta_1+\theta_3>1,\\
\plus{2\theta_1-1}\plus{2\theta_3} & \text{ if }\theta_1>\theta_3,\ \theta_1+\theta_3<1,\\
\plus{2\theta_1-1}\plus{2(1-\theta_3)} & \text{ if }\theta_1<\theta_3,\ \theta_1+\theta_3>1,\\
 \end{array}.\right. \\
 .$$
\end{coro}

\begin{proof}
As for \ref{problem parabola full}, the primitive moment on $\CC E_1$ can be
$$\left\{ \begin{array}{l}
\pm\frac{\theta_1+p_1}{l_1}\pi \text{ for }0\leqslant p_1<l_1, \\
\pm\frac{1-\theta_1+p_1}{l_1}\pi \text{ for }0\leqslant p_1<l_1. \\
\end{array}\right.$$
Same with $1$ replaced by $3$ for $\CC E_3$. They are grouped by four: pairs of opposite pairs of conjugated points. The representatives of the arguments lying in $]0;\pi[$ can be taken to be $\frac{\theta_j+p_j}{l_j}$ for $j=1$ or $3$. The other possible choices would be to replace $\theta_j$ by $1-\theta_j$ for some of them. Thus, we only have to add the contributions for these arguments, and using \ref{prop ellipse problem}, we get
$$\sum_{\substack{0\leqslant p_1 <l_1 \\ 0\leqslant p_3<l_3}}
4\sum_{ \substack{ 0\leqslant k_1 <h_1 \\ 0\leqslant k_3 < h_3 }} \sigma q^{2l_3\left( (1-)\frac{\theta_3+p_3}{l_3} + k_3 \right) - 2l_1\left( (1-)\frac{\theta_1+p_1}{l_1} + k_1 \right)}  ,$$
which we now compute explicitly. First, notice that the exponent has value
\begin{align*}
 & 2l_3\left( (1-)\frac{\theta_3+p_3}{l_3} + k_3 \right) - 2l_1\left( (1-)\frac{\theta_1+p_1}{l_1} + k_1 \right) \\
 = & 2\left( (l_3-)\theta_3 + p_3 +k_3l_3 \right) - 2\left( (l_1-)\theta_1 + p_1+k_1l_1 \right).\\
\end{align*}
where $(l-)\theta$ denotes $\theta$ or $l-\theta$. In case the value is $l-\theta$, we can make the change of index $p_3'=l_3-1-p_3$, and we are left with $1-\theta$ instead, and it goes through the same set of values. Moreover, notice that the index $p_j+k_jl_j$ goes through the integer interval $[\![ 0;l_jh_j-1]\!]$, and $h_jl_j=\frac{m_\Delta}{4}$ is constant. Thus, the double sum over a rectangle each time is in fact a unique sum over a square of side length $\frac{m_\Delta}{4}$. The sum is then
\begin{align*}
 & \sum_{ 0\leqslant k_1,k_3 <\frac{m_\delta}{4} } 
 \sigma  q^{2\left( (1-)\theta_3 + k_3 \right) - 2\left( (1-)\theta_1 + k_1 \right)} \\
 = & \sum_{ 0\leqslant k_1,k_3 <\frac{m_\delta}{4} } 
(\sigma q^{2\theta_3-2\theta_1}+\sigma q^{2-2\theta_3-2\theta_1}+\sigma q^{2\theta_3+2\theta_1-2}+\sigma q^{2\theta_1-2\theta_3})q^{2k_3-2k_1} . \\
\end{align*}
This sum is computed in the following lemma, leading to the announced result.

\begin{lem}
For $0<\theta,\varphi<1$, one has
\begin{align*}
 S(\theta,\varphi)=& \sum_{0\leqslant i,j<n}(\sigma q^{2(\varphi-\theta)}+\sigma q^{-(2\varphi-2\theta)}+\sigma q^{2(1-\varphi-\theta)}+\sigma q^{-2(1-\varphi-\theta)})q^{2j-2i} \\
 = & \plus{2\varphi-1}\plus{2\theta-1}\frac{\minus{2n}}{\minus{1}\minus{1}}-\frac{2n}{\minus{1}} \left\{ \begin{array}{ll}
\plus{2\varphi-1}\plus{2\theta} &  \text{ if }\theta<\varphi,\ \theta+\varphi<1,\\
\plus{2\varphi-1}\plus{2(1-\theta)} & \text{ if }\theta>\varphi,\ \theta+\varphi>1,\\
\plus{2\varphi}\plus{2\theta-1} & \text{ if }\theta>\varphi,\ \theta+\varphi<1,\\
\plus{2(1-\varphi)}\plus{2\theta-1} & \text{ if }\theta<\varphi,\ \theta+\varphi>1,\\
 \end{array}.\right. \\
\end{align*}
\end{lem}

\begin{proof}
We split the sum in three parts:
\begin{itemize}[label=-]
\item If $j>i$, the exponent is always positive due to the assumption on $\theta$ and $\varphi$, and therefore $\sigma=+1$. Then, we notice that
\begin{align*}
q^{2(\varphi-\theta)}+ q^{-(2\varphi-2\theta)}+ q^{2(1-\varphi-\theta)}+q^{-2(1-\varphi-\theta)} & = \plus{2\varphi-2\theta}+\plus{2-\varphi-2\theta} \\
& = \plus{2\varphi-1}\plus{2\theta-1},\\
\end{align*}
and compute
\begin{align*}
\sum_{j=1}^{n-1}\sum_{i=0}^{j-1} q^{2(j-i)}= & \sum_{j=1}^{n-1}\sum_{i=0}^{j-1} q^{2(j-i)} \\
=&  \sum_{j=1}^{n-1} q^{2j}\frac{1-q^{-2j}}{1-q^{-2}}  \\
= &  \frac{q}{\minus{1}}q^2\frac{q^{2(n-1)}-1}{q^2-1}-(n-1)\frac{q}{\minus{1}} \\
= & \frac{q^{2n}-q^2}{\minus{1}^2}-(n-1)\frac{q}{\minus{1}} . \\
\end{align*}

\item When we add the the sum for $i>j$, the value is the same if we replce $q$ by $q^{-1}$, since the sign exponent $2(j-i)$ changes it sign. The sign $\sigma$ is this time $-1$. Hence, we get
\begin{align*}
 & \frac{q^{2n}-q^2}{\minus{1}^2}-\frac{q^{-2n}-q^{-2}}{\minus{1}^2} - (n-1)\frac{q}{\minus{1}}-(n-1)\frac{q^{-1}}{\minus{1}} \\
 = &  \frac{\minus{2n}}{\minus{1}^2}-\frac{\minus{2}}{\minus{1}^2}-(n-1)\frac{\minus{1}\plus{1}}{\minus{1}^2} \\
 = &  \frac{\minus{2n}}{\minus{1}\minus{1}}-n\frac{\plus{1}}{\minus{1}}.\\
\end{align*}

\item For the central sum, since $i=j$, we have the following factorizations:
	\begin{itemize}
	\item If $\theta<\varphi$ and $\theta+\varphi<1$:
	$$n(q^{2(\varphi-\theta)}- q^{-(2\varphi-2\theta)}+ q^{2(1-\varphi-\theta)}-q^{-2(1-\varphi-\theta)}) = -n\plus{2\varphi-1}\minus{2\theta-1}.$$
	\item If $\theta>\varphi$ and $\theta+\varphi>1$:
	$$n(-q^{2(\varphi-\theta)}+ q^{-(2\varphi-2\theta)}- q^{2(1-\varphi-\theta)}+q^{-2(1-\varphi-\theta)}) = n\plus{2\varphi-1}\minus{2\theta-1}.$$
	\item If $\theta>\varphi$ and $\theta+\varphi<1$:
	$$n(-q^{2(\varphi-\theta)}+ q^{-(2\varphi-2\theta)}+ q^{2(1-\varphi-\theta)}-q^{-2(1-\varphi-\theta)}) = -n\plus{2\theta-1}\minus{2\varphi-1}.$$
	\item If $\theta<\varphi$ and $\theta+\varphi>1$:
	$$n(q^{2(\varphi-\theta)}- q^{-(2\varphi-2\theta)}- q^{2(1-\varphi-\theta)}+q^{-2(1-\varphi-\theta)}) = n\plus{2\theta-1}\minus{2\varphi-1}.$$
	\end{itemize}

\end{itemize}
We can pass from one factorization to another using the symmetries $(\theta,\varphi)\mapsto(\varphi,\theta)$, $\theta\mapsto 1-\theta$, \dots This corresponds to the action of the diedral group over the regions delimited by the inequalities. The total contribution is then
$$ \plus{2\varphi-1}\plus{2\theta-1}\frac{\minus{2n}}{\minus{1}\minus{1}}-\frac{n}{\minus{1}}\left( \plus{1}\plus{2\varphi-1}\plus{2\theta-1}-\minus{1}G(\theta,\varphi)\right),$$
where $G(\theta,\varphi)$ is the above factor without the $n$. If $\theta<\varphi$ and $\theta+\varphi<1$, the second term can be factorized, leading to
$$ \plus{2\varphi-1}\plus{2\theta-1}\frac{\minus{2n}}{\minus{1}\minus{1}}-\frac{2n}{\minus{1}}\plus{2\varphi-1}\plus{2\theta}.$$
The other cases are obtained by symmetry and or treated similarly.
\end{proof}

\end{proof}

	\subsection{Global enumerative problem}	
	
	We now can prove Theorem \ref{theorem paper}, relating $R_{\Delta,s}$ and $N_{\Delta(s)}^{\partial,\text{trop}}$. This theorem is a consequence of Proposition \ref{firstordersol} that proves that the multiplicities of rational tropical curves of degree $\Delta(s)$ given through the correspondence the are proportional to the refined Block-G\"ottsche (Definition \ref{refined multiplicity}) used to compute $N_{\Delta(s)}^{\partial,\mathrm{trop}}$.\\	

Let $\mathcal{P}_t$ be a symmetric configuration of points depending on a parameter $t$, chosen as in section \ref{classical problem}. For a pair of complex conjugated points $p,\overline{p}\in\mathcal{P}_t$, let $\pm\theta_p\pi$ be their aguments, with $0<\theta_p<1$. We keep similar notations for the rest. Being given a parametrized tropical curve $h_0:\Gamma_0\rightarrow N_\RR$ of degree $\Delta(s)$ with $\text{ev}(\Gamma_0)=\mu$, the task of computing its multiplicity amounts to two things. The first is to find the parametrized real tropical curves of degree $\Delta$ having the same image and admitting first order solutions. The second task consists in finding the first order solutions to $\Theta(\chi,\alpha,\beta)=(1,\zeta^\Gamma)$.\\

The first problem is taken care of with the following lemma.

\begin{lem}
\label{lem no flat vertex}
The parametrized real tropical curves with a trivalent flat vertex cannot be the tropicalization of a family of parametrized real rational curves passing through the symmetric configuration $\mathcal{P}_t$.
\end{lem}

\begin{proof}
Let $h_0:\Gamma_0\rightarrow N_\RR$ be a parametrized rational tropical curve of degree $\Delta(s)$ such that $\text{ev}(\Gamma_0)=\mu$. Let $C_\text{trop}=h_0(\Gamma_0)$ be its image, which is a plane tropical curve. The different possible real structures are described in Proposition \ref{realstruc}. Let $(\Gamma,h)$ be one of them.\\

Assume that there exists a trivalent flat vertex $w$ in $(\Gamma,h)$, with two outgoing edges leading to branches exchanged by the involution. Let $f_t:\CC P^1\dashrightarrow N\otimes\CC((t))^*$ be a parametrized real rational curve tropicalizing to $(\Gamma,h)$.\\

First, notice that at each vertex with outer edges directed by $n_i$, by denoting $m_i=\iota_{n_i}\omega$, due to balancing condition, one has $\prod_i\chi_V(m_i)=1$. Thus, for every bounded edge $\gamma$, one has
$$\chi_\hg(m_\gamma)=\prod_{\mathfrak{t}(\gamma')=\hg}\chi_\hg(m_{\gamma'}).$$
Moreover, the vertices in each branch are trivalent, therefore, at first order, $\phi_\gamma=(\pm 1)^{n_{\gamma'}}$, except for edge $\gamma$ such tht $\tg=V$, where $\phi_\gamma=(2i)^{n_\gamma}$. In any case, the transfer relation $\phi_\gamma\cdot\frac{\chi_\tg}{\chi_\hg}\cdot\alpha_\gamma^{n_\gamma}=1$ evaluated at $m_\gamma$ and at first order gives us
$$\chi_\tg(m_\gamma)=\pm\chi_\hg(m_\gamma).$$
Multiplying these relations for every vertex in the branch, for $\gamma$ the one of the complex edges adjacent to $V$, we get
$$\chi_V(m_\gamma)=\pm\prod_e \chi_{v_e}(m_{e}),$$
where the product is indexed by the unbounded ends belonging to the branch, and $v_e$ denotes the adjacent vertex to a branch $e$, and $m_e$ the character used to computes its moment. Due to condition of passing through $\mathcal{P}_t$, the argument of $\chi_{v_e}(m_e)$ is $\pm\theta_e$ or $\pi\pm\theta_e$, \textit{i.e.} the argument of the point through which the curve passes. Therefore, the argument of $\chi_V(m_\gamma)$ is a sum of $\pm\theta_e$. In particular, for generic points, it is not real, which should be the case if $V$ is fixed: $\chi_V$ would be real. Hence, the curve cannot have any trivalent flat vertex if the configuration is generic.
\end{proof}

We now count the first order solutions, which lead to true solutions thanks to the correspondence theorem. It hapens thaht the knowing of the first order is sufficient to recover the sign and the logarithmic area. The refined count of these first order solutions thus gives the right multiplicity used to count tropical curves.

\begin{prop}\label{firstordersol}
Let $\mathcal{P}_t$ be a symmetric real configuration of points as previously chosen, tropicalizing on a family of moments $\mu$. Let $h_0:\Gamma_0\rightarrow N_\RR$ be a parametrized tropical curve of degree $\Delta(s)$ having moments $\mu$. The refined count according to the log-area and sign $\sigma$ is given by
$$m'_{\Gamma_0}=4\frac{2^{|s|}}{\minus{1}^{|s|}} \prod_{V}\minus{\frac{m_V}{2}},$$
where the product is over the trivalent vertices of $\Gamma_0$.
\end{prop}

Before proving this proposition, we can now prove Theorem \ref{theorem paper}.

\begin{proof}[Proof of Theorem \ref{theorem paper}]
This is a consequence of the correspondence theorem that states that each of the first order solutions lifts to a unique solution, given by Proposition \ref{firstordersol}: we obtain $R_{\Delta,s}$ by counting curves with multiplicities $\frac{1}{4}m'_{\Gamma_0}$ while $N_{\Delta(s)}^{\partial,\mathrm{trop}}$ is obtained by counting them with multiplicity $m_{\Gamma_0}^q$. Finally, notice that
$$\frac{1}{4}m_{\Gamma_0}'=\frac{2^{|s|}\minus{\frac{1}{2}}^{m-2-|s|}}{\minus{1}^{|s|}}m_{\Gamma_0}^q.$$
\end{proof}

We now prove Proposition \ref{firstordersol}.

\begin{proof}[Proof of Proposition \ref{firstordersol}]
The proof is made with an induction on the number of vertices of the curve $\Gamma$. Exceptionally, to suit the induction, a vector $n_j$ of $N$ directing an unbounded end $e_j$ may not be primitive. In that case, we denote by $m_j=\omega\left(\frac{n_j}{l(n_j)},-\right)$ the dual vector, but still of lattice length $1$.\\

Let $(\Gamma_0,h_0)$ be a parametrized tropical curve of degree $\Delta(s)$ with $\mathrm{ev}(\Gamma_0,h_0)=\mu$. The real tropical curve that may have first order solutions are real tropical curves $(\Gamma,h,\sigma)$ with the same image and without any flat vertex. The proof proceeds as follows: we compute the refined number of first order solution by evicting one by one all the vertices for any real tropical curve $(\Gamma,h)$. Then, we add all the refined count for every real tropical curve $(\Gamma,h)$ such that $h(\Gamma)=h_0(\Gamma_0)$.\\

Let $(\Gamma,h)$ be a real tropical curve with $h(\Gamma)=h_0(\Gamma_0)$ without flat vertex. For a vertex, we have two possibilities: either it is a complex vertex, or it is a real vertex. In the second case, we assume that it is an extremal vertex of the subgraph $\fix$, meaning that every vertex $W\succ V$ is a complex vertex. Let $V$ such a vertex. We have three possibilities: $V$ is trivalent and adjacent to two real unbounded ends, $V$ is quadrivalent and adjacent to one real unbounded end, or $V$ is pentavalent. The next four points do the following:
\begin{enumerate}[label=(\roman*)]
\item evict a trivalent real vertex,
\item evict a complex branch,
\item evict a quadrivalent vertex,
\item evict a pentavalent vertex.
\end{enumerate}

\begin{enumerate}[label=(\roman*)]
\item First, assume that $V$ is adjacent to two unbounded ends indexed by $0$ and $1$. Let $\gamma$ be the unique edge with $\hg=V$. If $V$ is the only vertex of the tropical curve, the refined count has already been proven to be $4\minus{\frac{m_V}{2}}$. Otherwise, $\gamma$ is bounded and the coordinates associated to $V$ and $\gamma$ are $\alpha_\gamma$ and $\chi_V$. Then, we have to solve for $\chi_V:M\rightarrow\RR^*$ the following system:
$$\left\{\begin{array}{l}
\chi_V(m_0)=\pm\zeta_0^\Gamma\in\RR^* \\
\chi_V(m_1)=\pm\zeta_1^\Gamma\in\RR^* \\
\end{array}\right. .$$
Recall that the vectors $n_0$ and $n_1$ might not be primitive, but $m_0$ and $m_1$ are. According to Proposition \ref{prop trivalent real vertex} , this system leads to $4$ solutions: if $(e_1^*,e_2^*)$ is a basis of $M$, the absolute value of $\chi_V(e_1^*)$ and $\chi_V(e_2^*)$ is uniquely determined, while the sign may be chosen arbitrarily. Notice that some choices of signs for the right-hand side of the system may provide several solutions, while other provide none. Let $m_\gamma$ be the primitive vector dual to $n_\gamma$. Let $\tilde{m}_\gamma$ be such that $(m_\gamma,\tilde{m}_\gamma)$ is a basis of $M$. These $4$ solutions separate themselves into two groups of $2$, according to the sign of $\chi_V(m_\gamma)$. We have the transfer equation
$$\phi_\gamma\cdot\frac{\chi_\tg}{\chi_V}\cdot\alpha_\gamma^{n_\gamma}=1\in N\otimes\RR^*.$$
We evaluate at $m_\gamma$, leading to
$$\phi_\gamma(m_\gamma)\chi_\tg(m_\gamma)=\chi_V(m_\gamma)\in\RR^*.$$
Recall that according to the choice of signs of $\pm\zeta_j$, the sign of $\chi_V(m_\gamma)$ may change. So we migt as well add $\pm$, leading to
$$\phi_\gamma(m_\gamma)\chi_\tg(m_\gamma)=\pm|\chi_V(m_\gamma)|\in\RR^*.$$
Let replace the bounded edge $\gamma$ by an unbounded end with direction $n_\gamma$, leading to a new parametrized tropical curve $(\Gamma',h')$. The above  equation is the equation associated to this new unbounded end in $\Gamma'$, in the corresponding system $\Theta(\chi,\alpha,\beta)=(1,\zeta^\Gamma)$. Thus, we can proceed by induction. Let $4R$ denote the refined signed count of oriented curves lifting $\Gamma'$. These $4R$ oriented curves separate themselves into four groups of $R$ oriented curves according to the value of the signs the function $\phi_\gamma\cdot\chi_\tg$ takes on the basis $(m_\gamma,\tilde{m}_\gamma)$, using the action of the deck transformation group $\{\pm 1\}^2$. On the other side, we get only the $2$ solutions among those for which $\phi_\gamma(m_\gamma)\chi_\tg(m_\gamma)=\chi_V(m_\gamma)$. Last, we need to solve for $\alpha_\gamma$. By evaluating at $\tilde{m}_\gamma$, we get
$$\alpha_\gamma^{\langle n_\gamma,\tilde{m}_\gamma\rangle}=\frac{\chi_V(\tilde{m}_\gamma)}{\phi_\gamma(\tilde{m}_\gamma)\cdot\chi_\tg(\tilde{m}_\gamma)}.$$
Notice that through the presence of $\phi_\gamma$ and $\chi_\tg$, the right hand term depends on which of the $4R$ solutions is chosen. The solving as well as the number of solutions depends on the sign of the right-hand side.
	\begin{itemize}[label=$\star$]
	\item If $n_\gamma$ has odd integer length, then we solve uniquely for $\alpha_\gamma$ for each possible sign of $\chi_V(\tilde{m}_\gamma)$. The sign of $\chi_V(m_\gamma)$ is already determined since we have
	$$\phi_\gamma(m_\gamma)\chi_\tg(m_\gamma)=\chi_V(m_\gamma).$$
Thus, each of the oriented curves in each of the groups have two possible solutions for $\chi_V$. This corresponds to the gluing of two possible curves, one increasing the logarithmic rotation number by one, the other decreasing it by one. In each case, the orientation of the curve propagates, and the signed count becomes
	$$4\times (q^{m_V/2}-q^{-m_V/2})R.$$
	
	\item If $n_\gamma$ has even integer length, then the sign of $\chi_V(m_\gamma)$ is still determined, and the sign of $\chi_V(\tilde{m}_\gamma)$ is forced in order to have at least one solution for $\alpha_\gamma$. In that case, we have only one of the four possible curves that we can glue, but there are two possible values for $\alpha_\gamma$, which means two gluings. One of these choices decreases the logarithmic rotation number by one while the other increases it by one. The signed count then becomes again
	$$4\times (q^{m_V/2}-q^{-m_V/2})R.$$
	\end{itemize}

\item Before considering quadrivalent and pentavalent vertices, we make point about the importance of  the repartition of the complex points inside a complex branch. Consider a complex branch $\mathcal{B}$ of the real tropical curve $\Gamma$. This means that there exists some quadrivalent or pentavalent vertex $V$, and $\mathcal{B}$ is the subgraph of $\Gamma$ accessible via one of the complex edges adjacent to $V$. Let $\gamma_0$ be this edge, called root of the branch. Notice that for a branch $\mathcal{B}$, we have the conjugated branch $\sigma(\mathcal{B})$.\\

For any edge $\gamma$, recall we distinguish between its moment, which is $\chi_\tg(m_\gamma)$, and its primitive moment, which is $\chi_\tg\left(\frac{m_\gamma}{l(m_\gamma)}\right)$. The goal of this paragraph is to express the moment of the root of the branch in term of the moments of the unbounded ends, and then count the possible lifts.\\

Let index by $[\![ 1;p]\!]$ the (complex) unbounded ends belonging to $\mathcal{B}$. Each unbounded end indexed $j$ is associated to a pair of conjugated points having moment $\pm e^{\pm i\pi\theta_j}$. To distinguish between $\mathcal{B}$ and $\sigma(\mathcal{B})$, we assume that the point belonging to $\mathcal{B}$ associated to the first end has moment $e^{i\pi\theta_1}$ or $-e^{-i\pi\theta_1}$. Then, there are $2^{p-1}$ possible repartitions for the other points according to whether it is the point with argument $\theta_j\pi$ or $-\theta_j\pi$ that belongs to $\mathcal{B}$.\\

For any repartition of the arguments, we have:
	\begin{itemize}[label=$\star$]
	\item due to balancing condition, at any vertex $W$ with ingoing edge $\gamma$ and outgoing edges $\gamma_1$ and $\gamma_2$, one has
	$$\chi_V(m_\gamma)=\chi_V(m_1)\chi_V(m_2).$$
	\item for any edge $\gamma$, due to the transfer equation evaluated at $\frac{m_\gamma}{l(m_\gamma)}$, one has
	$$ \chi_\hg\left( \frac{m_\gamma}{l(m_\gamma)} \right)=\phi_\gamma\left( \frac{m_\gamma}{l(m_\gamma)} \right)\chi_\tg\left( \frac{m_\gamma}{l(m_\gamma)} \right)=\pm\chi_\tg\left( \frac{m_\gamma}{l(m_\gamma)} \right) ,$$
	and thus also
	$$ \chi_\hg\left( m_\gamma\right)=\phi_\gamma\left(m_\gamma\right)\chi_\tg\left(m_\gamma \right)=\pm\chi_\tg\left(m_\gamma\right) .$$
	\end{itemize}
	Therefore, if the moments of the unbounded ends are fixed, the moment of every bounded edge is also fixed, and is equal up to sign to the product of the moments of the unbounded ends.\\
	
	We can now compute the number of first order solution for every possible primitive moment of the root $\gamma_0$. Fix one of these possible values. Let $V_0=\mathfrak{h}(\gamma_0)$. Proposition \ref{prop trivalent complex vertex} ensures that there is $\frac{m_{V_0}}{4l(m_{\gamma_0})}$ choices for $\chi_{V_0}$. The $4$ is to account that the multiplicity appearing in \ref{prop trivalent complex vertex} is not the multiplicity of $V_0$ in the curve $\Gamma_0$. This choice now fixes the primitive moment of the edges outgoing from $V_0$ and one can proceed by induction. We get a number of first order  solutions equal to
	$$\prod_{W\in\mathcal{B}-\{V\}}\frac{m_W}{4l(m_{\gamma_W})},$$
where $\gamma_W$ denotes the unique edge such that $\hg=W$. This number of solutions has to be multiplied by the number of possible first order solutions of $\alpha_\gamma$. This solving for $\alpha_\gamma$ is obtained by evaluating the transfer equation at some vector completing $\frac{m_\gamma}{l(m_\gamma)}$ into a basis of $M$. We thus get $l(m_\gamma)$ solutions. In total, we get $\prod_W\frac{m_W}{4}$ first order solutions for each possible value of the primitive moment of $\gamma_0$.\\

To finish this part of the computation, we give the set $\Theta$ of possible values of the moment of the root edge $\gamma_0$. According to the above, it is up to sign the product of the moments of the unbounded ends in $\mathcal{B}$. The possible values are
$$\pm e^{i\pi (\pm\theta_1\pm\theta_2\pm\cdots\pm\theta_p)}.$$
There are $2^{p+1}$ such values, \textit{i.e.} $2^p$ pairs of conjugated values, themselves consisting in $2^{p-1}$ pairs of opposite conjugated pairs. On the other side, as the moment of each unbounded end can take four values, and the conjugated branch takes the conjugated moments, there are $\frac{4^p}{2}=2^{2p-1}$ repartitions of the moments. They are equally spread between the $2^{p}$ pairs of conjugated moments. Thus, each pair of conjugated moments is attained by $\frac{2^{2p-1}}{2^{p}}=2^{p-1}$ repartitions. Another set of $2^{p-1}$ repartitions is assigned to the opposite pair of conjugated moments.\\

\begin{expl}
For $p=1$, there is only one $4$-tuple. For $p=2$, there are two $4$-tuple, each one attained by four repartitions of the arguments of the unbounded ends. There are indeed $8$ possible repartitions: choose $\theta_1$ or $1-\theta_1$ (opposite pair), and the point with him on the branch, of which there are $4$.
\end{expl}

Finally, the multiplicity to account for the first order solutions corresponding to coordinates in the branch is
$$2^{p-1}\prod_W \frac{m_W}{4}=\prod_W \frac{m_W}{2},$$
since there are $p-1$ vertices, and we are left to solve the problem where the branch is replaced by a pair of conjugated ends, and the condition imposed is a fixed moment.

\item Now, assume that $V$ is a quadrivalent vertex adjacent to a real unbounded end and two exchanged complex edges $\gamma^\CC,\sigma(\gamma^\CC)$ which might be unbounded or not. We denote by $\mathcal{B}$ the branch accessible by $\gamma^\CC$. Let $n_0$ be the common slope of $h$ on the complex edges, and $n_1$ be the slope of $h$ on the real unbounded end. Let $V_0=\mathfrak{h}(\gamma^\CC)$ be the first complex vertex in the branch if there is one. Let $\beta$ and $\chi_V$ be the coordinates associated to $V$.\\

Using the previous point, to find $\chi_V$ and $\beta$, we have to solve the moment problem for a parabola, equivalent to the following system:
$$\left\{ \begin{array}{l}
\chi_V(m_0) (\beta-i)^{\left\langle m_0,n_1 \right\rangle }=\pm\chi_{V_0}(m_0) ,\\
\chi_Vm_1(\beta^2+1)^{\left\langle m_1,n_0 \right\rangle}=\pm\zeta^\Gamma_1.\\
\end{array}\right. $$
The first equation is the first order of the transfer equation of $\gamma^\CC$ evaluated at vector $m_0=\iota_{n_0}\omega$, while the second equation is the moment equation for the real unbounded end. If the complex edges adjacent to $V$ are unbounded ends, the first equation is replaced with the moment equation $\chi_V(m_0) (\beta-i)^{\left\langle m_0,n_1 \right\rangle }=\pm\zeta^\Gamma_0\in\CC^*$. This system is the parabola problem for a curve of degree given by the outgoing edges of $V$, and with a pair of complex points having moments $\pm e^{i\pi\theta}$, where $\theta$ takes the values provided by the previous point. For each possible value, the refined count was proven to be
$$4\plus{2\theta-1}\frac{\minus{\frac{m_V}{2}}}{\minus{1}}.$$
This contribution has to be multiplied by $\prod_W\frac{m_W}{2}$ and added over the $2^{p-1}$ values that may be taken by $\theta$.

\item If the vertex $V$ is pentavalent, adjacent to two pairs of exchanged complex edges behind which we find four branches, similarly, the number of solutions in the branches has already been taken care of: the refined count for the moment problem associated to $V$ and each possible value of the moments of the adjacent complex edges is multiplied by $\prod_W\frac{m_W}{2}$, where $W$ goes through the vertices in both branches. Let $0<\theta,\varphi<1$ be such that the moments of the complex edges may take the value $\pm e^{\pm i\pi\theta}$ and $\pm e^{\pm i \pi\varphi}$, where $\varphi$ and $\theta$ may take several values, described by the second point of this proof. Then, the refined count is given by \ref{coro ellipse problem}. Total contribution is computed later on.\\
\end{enumerate}

All the cases of multiplicities have been dealt with for a specific real tropical curve. The last thing to do is to add the various contributions of all the real tropical curves of degree $\Delta$ that can be obtained from a a specific curve of degree $\Delta(s)$. These multiple choices are described by Proposition \ref{realstruc}.\\

Let $\gamma$ be an edge of $\Gamma_0$, and let $\mathcal{B}=\mathcal{B}(\gamma)$ be the branch of $\Gamma_0$ consisting of points accessible via $\gamma$ in the oriented structure of the tropical curve. We distinguish two cases according to whether the lift of $\gamma$ in $\Gamma$ is a pair of complex edges or a fixed edge.
\begin{itemize}
\item In case the lift consist in a pair of complex edges, let $\Theta_\gamma$ be the set of values of normalized arguments (in $]0;1[$) and up to sign that the moment of the edge may take when we vary the moments of the unbounded edges, so that we have exactly one representative by $4$-tuple. In the second point, we proved that if the branch $\mathcal{B}$ has $p$ unbounded ends, $|\Theta_\gamma|=2^{p-1}$.
\item Otherwise, the lift is a fixed edge.
\end{itemize}

Now, for an edge $\gamma$ in a component of $\Gamma_\mathrm{even}$, let $R_\gamma$ be the number of first order solutions in the branch $\mathcal{B}(\gamma)$ if the edge $\gamma$ is lifted to a fixed edge of $\gamma$, and let $C_\gamma$ be the number of first order solutions in the branch if $\gamma$ is lifted to a pair of conjugated edges.\\

Let $\gamma$ be an edge, and let $\gamma_1$ and $\gamma_2$ be the two edges emanating from $\hg$. We drop the index for quantities concerning $\gamma$ and index $1$ and $2$ quantities concerning $\gamma_1$ or $\gamma_2$. We have already proven the following statements.
\begin{itemize}[label=$\star$]
\item The complex contributions being a product over the vertices in the branch, they satisfy
$$C=C_1C_2\frac{m_\hg}{2}.$$
\item If $\gamma$ were an unbounded end having moment $\pm e^{\pm i\pi\theta}$, we would have $\Theta=\{\theta\}$.
\item The possible moments satisfy
$$\Theta=(\Theta_1+\Theta_2)\cup(1+\Theta_1+\Theta_2) \ \mathrm{mod } 1.$$
This relation is compatible with $2\times2^{p_1-1}\times 2^{p_2-1}=2^{p_1+p_2-1}$.
\end{itemize}
Finally, we have the following recursive relation, which emphasizes the following disjonction: if a bounded edge $\gamma$ is lifted to a fixed edge, the lift of $\hg$ in $\Gamma$ is either trivalent real vertex, a pentavalent vertex, or a quadrivalent vertex, and the two adjacent exchanged edges may be a lift of either $\gamma_1$ or $\gamma_2$. Therefore, one has,
\begin{align*}
R= &  \minus{\frac{m_\hg}{2}}R_1R_2+\sum_{(\theta,\varphi)\in\Theta_1\times\Theta_2}S(\theta,\varphi)C_1C_2 \\
& + \sum_{\theta_1\in\Theta_1}\plus{2\theta_1-1}\frac{\minus{\frac{m_\hg}{2}}}{\minus{1}}C_1R_2 + \sum_{\theta_2\in\Theta_2}\plus{2\theta_2-1}\frac{\minus{\frac{m_\hg}{2}}}{\minus{1}}R_1C_2.\\
\end{align*}

\begin{itemize}[label=-]
\item The first term corresponds to $\hg$ being lifted to a trivalent vertex, and the $\minus{\frac{m_\hg}{2}}$ accounts for the refined number of first order solutions associated to $\chi_\hg$).
\item The first sum corresponds to $\hg$ being lifted to a pentavalent vertex. recall that the refined number of local lifts is $S(\theta,\varphi)$, and is computed in \ref{coro ellipse problem}.
\item The second sum corresponds to $\hg$ being lifted to a quadrivalent vertex whose adjacent complex edges are lifts of $\gamma_1$.
\item Same with $\gamma_2$.
\end{itemize}
 
The final step is to compute 
$$\minus{\frac{m_\tg}{2}}R+\frac{\minus{\frac{m_\tg}{2}}}{\minus{1}}C \sum_{\theta\in\Theta}\plus{2\theta-1},$$
for the root edge of a component of $\Gamma_\mathrm{even}$, and a surprising simplification occurs, leading to a mere product over the vertices. Using the recursive formulas, we prove the following lemma.

\begin{lem}
If $\gamma_1$ and $\gamma_2$ are the edges leaving $V$ and $\gamma$ is such that $V=\hg$, then one has
$$R+\frac{C}{\minus{1}} \sum_{\theta\in\Theta}\plus{2\theta-1} = \minus{\frac{m_V}{2}}
\left(R_1+\frac{C_1}{\minus{1}} \sum_{\theta_1\in\Theta_1}\plus{2\theta_1-1}\right)
\left(R_2+\frac{C_2}{\minus{1}} \sum_{\theta_2\in\Theta_2}\plus{2\theta_2-1}\right).$$
\end{lem}

\begin{proof}
We use the recursive relation to get
\begin{align*}
 R+\frac{C}{\minus{1}} \sum_{\theta\in\Theta}\plus{2\theta-1}
= & \minus{\frac{m_V}{2}}\left[ R_1R_2 + R_1\frac{C_2}{\minus{1}}\sum_{\Theta_2}\plus{2\theta_2-1}+R_2\frac{C_1}{\minus{1}}\sum_{\Theta_1}\plus{2\theta_1-1} \right] \\
&  +  \frac{m_V}{2}\frac{C_1C_2}{\minus{1}}\sum_\Theta \plus{2\theta-1}  + C_1C_2\sum_{\Theta_1\times\Theta_2} S(\theta_1,\theta_2) .\\
\end{align*}
We now relate the sum over $\Theta$ and the sum over $\Theta_1\times\Theta_2$.
\begin{itemize}
\item The value of $S(\theta_1,\theta_2)$ given by Corollary \ref{coro ellipse problem} depends on whether $\theta_1<\theta_2$ and $\theta_1+\theta_2<1$. It is a sum of two terms. The first does not depend on this condition while the value of the second depends. Thus, we have
	\begin{align*}
	\sum_{\Theta_1\times\Theta_2}S(\theta_1,\theta_2) = & \frac{\minus{\frac{m_V}{2}}}{\minus{1}\minus{1}}\sum_{\Theta_1\times\Theta_2} \plus{2\theta_1-1}\plus{2\theta_2-1} \\
	& -\frac{m_V}{2\minus{1}}\sum_{\substack{\theta_1<\theta_2 \\ \theta_1+\theta_2<1}} \plus{2\theta_1}\plus{2\theta_2-1} \\
	& -\frac{m_V}{2\minus{1}}\sum_{\substack{\theta_1>\theta_2 \\ \theta_1+\theta_2>1}} \plus{2(1-\theta_1)}\plus{2\theta_2-1} \\
	& -\frac{m_V}{2\minus{1}}\sum_{\substack{\theta_1>\theta_2 \\ \theta_1+\theta_2<1}} \plus{2\theta_1-1}\plus{2\theta_2} \\
	& -\frac{m_V}{2\minus{1}}\sum_{\substack{\theta_1<\theta_2 \\ \theta_1+\theta_2>1}} \plus{2\theta_1-1}\plus{2(1-\theta_2)}. \\
	\end{align*}

\item For $\sum_\Theta \plus{2\theta-1}$, we remember that $\Theta=(\Theta_1+\Theta_2)\cup(1+\Theta_1+\Theta_2)$, where the representatives are taken in $]0;1[$, mod $1$ and up to sign. Since each $\theta$ stands for a quadruple $\pm e^{\pm i\pi\theta}$, the sign corresponds to the conjugation and the mod $1$ to taking to opposite moments. Thus, if $\theta_1,\theta_2\in\Theta_1\times\Theta_2$, with $0<\theta_1,\theta_2<1$, the choice of representative for both sets $\Theta_1+\Theta_2$ and $1+\Theta_1+\Theta_2$ also depends on the inequalities $\theta_1<\theta_2$ and $\theta_1+\theta_2<1$:
\begin{itemize}[label=-]
\item If $\theta_1<\theta_2$ and $\theta_1+\theta_2<1$, representatives in $]0;1[$ can be taken to be $\theta_2-\theta_1$ and $\theta_1+\theta_2$, and we have
$$ \plus{2(\theta_2-\theta_1)-1}\plus{2(\theta_1+\theta_2)-1}=\plus{2\theta_1}\plus{2\theta_2-1} .$$
\item If $\theta_1>\theta_2$ and $\theta_1+\theta_2>1$, representatives in $]0;1[$ can be taken to be $\theta_1-\theta_2$ and $\theta_1+\theta_2-1$, and we have
$$ \plus{2(\theta_1-\theta_2)-1}\plus{2(\theta_1+\theta_2-1)-1}=\plus{2(1-\theta_1)}\plus{2\theta_2-1} .$$
\item If $\theta_1>\theta_2$ and $\theta_1+\theta_2<1$, representatives in $]0;1[$ can be taken to be $\theta_1-\theta_2$ and $\theta_1+\theta_2$, and we have
$$ \plus{2(\theta_1-\theta_2)-1}\plus{2(\theta_1+\theta_2)-1}=\plus{2\theta_1-1}\plus{2\theta_2} .$$
\item If $\theta_1<\theta_2$ and $\theta_1+\theta_2>1$, representatives in $]0;1[$ can be taken to be $\theta_2-\theta_1$ and $\theta_1+\theta_2-1$, and we have
$$ \plus{2(\theta_2-\theta_1)-1}\plus{2(\theta_1+\theta_2-1)-1}=\plus{2\theta_1-1}\plus{2(1-\theta_2)} .$$
\end{itemize}
\end{itemize}
Thus, the sums where the disjonction occurs cancel themselves. We are left with
\begin{align*}
R+\frac{C}{\minus{1}} \sum_{\theta\in\Theta}\plus{2\theta-1}
= & \minus{\frac{m_V}{2}}\left[ R_1R_2 + R_1\frac{C_2}{\minus{1}}\sum_{\Theta_2}\plus{2\theta_2-1}+R_2\frac{C_1}{\minus{1}}\sum_{\Theta_1}\plus{2\theta_1-1} \right. \\
 & \left. +\frac{C_1C_2}{\minus{1}\minus{1}} \sum_{\Theta_1\times\Theta_2}\plus{2\theta_1-1}\plus{2\theta_2-1} \right] \\
= & \minus{\frac{m_V}{2}}\left[ R_1R_2 + R_1\frac{C_2}{\minus{1}}\sum_{\Theta_2}\plus{2\theta_2-1}+R_2\frac{C_1}{\minus{1}}\sum_{\Theta_1}\plus{2\theta_1-1} \right. \\
& \left. +\frac{C_1C_2}{\minus{1}\minus{1}} \left(\sum_{\Theta_1}\plus{2\theta_1-1}\right)\left(\sum_{\Theta_2}\plus{2\theta_2-1}\right) \right] \\
= & \minus{\frac{m_V}{2}}\left( R_1+\frac{C_1}{\minus{1}}\sum_{\Theta_1}\plus{2\theta_1-1} \right)\left( R_2+\frac{C_2}{\minus{1}}\sum_{\Theta_2}\plus{2\theta_2-1} \right).\\
\end{align*}

\end{proof}

With this easier recursive relation, using that for an unbounded end,
$$R+\frac{C}{\minus{1}}\sum_\Theta \plus{2\theta-1} = 0+\frac{1}{\minus{1}}\plus{2\theta-1}=\frac{\plus{2\theta-1}}{\minus{1}},$$
we immediately get that the contribution for a branch $\mathcal{B}$ with $p$ unbounded ends is
$$\frac{1}{\minus{1}^p}\prod_{i=1}^p\plus{2\theta_i-1}\prod_{W\in\mathcal{B}}\minus{ \frac{m_W}{2} }.$$
To conclude, finishing the recursion over the real part of the curve using (i), we get a refined number of first order solutions equal to
$$4\prod_i \frac{\plus{2\theta_i-1}}{\minus{1}}\prod_V \minus{ \frac{m_V}{2} }.$$
Last, to get quantum indices instead of log-areas, it suffices to make each $\theta_i$ go to $1$, finally leading to the announced result
$$4\frac{2^{|s|}}{\minus{1}^{|s|}}\prod_{V}\minus{ \frac{m_V}{2} },$$
since $\plus{0}=2$.

\end{proof}

\bibliographystyle{plain}
\bibliography{biblio}

{\ncsc Institut de Math\'ematiques de Jussieu - Paris Rive Gauche\\[-15.5pt] 

Sorbonne Universit\'e\\[-15.5pt] 

4 place Jussieu,
75252 Paris Cedex 5,
France} \\[-15.5pt]

\medskip

{\it and}

\medskip

{\ncsc D\'epartement de math\'ematiques et applications\\[-15.5pt]

Ecole Normale Sup\'erieure\\[-15.5pt]

45 rue d'Ulm, 75230 Paris Cedex 5, France} \\[-15.5pt]

{\it E-mail address}: {\ntt thomas.blomme@imj-prg.fr}

\end{document}